\documentclass{article}
\usepackage[utf8]{inputenc}

\usepackage{authblk}

\title{Simplicial lists in operad theory I}
\author{Redi Haderi\footnote{redi.haderi@aybu.edu.tr (corresponding author)}}
\author{Özgün Ünlü\footnote{unluo@fen.bilkent.edu.tr}}
\affil{{\small{Department of Mathematics, Ankara Yıldırım Beyazit University, Turkey}}}
\affil{{\small{Department of Mathematics, Bilkent University, Ankara, Turkey}}}

\date{\today} 

\usepackage[english]{babel}
\usepackage[utf8]{inputenc}
\usepackage{amsmath}
\usepackage{graphicx}
\usepackage{quiver}
\usepackage{mathtools}

\usepackage[mathscr]{eucal}

\usepackage{url}

\usepackage{amsthm}

\usepackage{hyperref}
\usepackage{amssymb}

\usepackage{tikz}
\usetikzlibrary{arrows}
\usetikzlibrary{cd}

\usepackage{array}
\usepackage{epigraph}

\usepackage{bbold}

\newtheorem{theorem}{Theorem}[section]
\newtheorem{proposition}{Proposition}[section]

\newtheorem{lemma}{Lemma}[section]
\newtheorem{corollary}{Corollary}[section]

\theoremstyle{remark}

\theoremstyle{remark}

\theoremstyle{remark}
\newtheorem{remark}{Remark}[section]

\theoremstyle{remark}
\newtheorem*{note}{Note}

\theoremstyle{remark}
\newtheorem{construction}{Construction}[section]

\theoremstyle{remark}
\newtheorem{example}{Example}[section]

\theoremstyle{remark}
\newtheorem{notation}{Notation}[section]

\theoremstyle{remark}
\newtheorem{definition}{Definition}[section]

\makeatletter
\def\slashedarrowfill@#1#2#3#4#5{%
  $\m@th\thickmuskip0mu\medmuskip\thickmuskip\thinmuskip\thickmuskip
   \relax#5#1\mkern-7mu%
   \cleaders\hbox{$#5\mkern-2mu#2\mkern-2mu$}\hfill
   \mathclap{#3}\mathclap{#2}%
   \cleaders\hbox{$#5\mkern-2mu#2\mkern-2mu$}\hfill
   \mkern-7mu#4$%
}
\def\rightslashedarrowfill@{%
  \slashedarrowfill@\relbar\relbar\mapstochar\rightarrow}
\newcommand\xslashedrightarrow[2][]{%
  \ext@arrow 0055{\rightslashedarrowfill@}{#1}{#2}}
\makeatother

\newcommand{\cc}[1]{\mathscr{#1}}
\newcommand{\bb}[1]{\mathbb{#1}}
\newcommand{\D}[1]{\Delta^{#1}}
\newcommand{\DD}[1]{\bb{\Delta}^{#1}}
\newcommand{\set}{\mathsf{Set}}
\newcommand{\pset}{\mathsf{Set}_*}
\newcommand{\spa}{\mathsf{List}}

\newcommand{\sset}{\mathsf{sSet}}

\newcommand{\cat}{\mathsf{Cat}}
\newcommand{\scat}{\mathsf{Cat}_\D{}}
\newcommand{\operad}{\mathsf{Op}}
\newcommand{\soperad}{\mathsf{Op}_\D{}}
\newcommand{\moncat}{\mathsf{MonCat}}
\newcommand{\dplus}{\Delta_+}
\newcommand{\assoc}{\mathsf{Assoc}}

\newcommand{\mult}{\mathsf{MGraph}}

\newcommand{\nd}{N(\Delta_+)}

\newcommand{\llist}{\mathsf{List}}
\newcommand{\slist}{\mathsf{sList}}
\newcommand{\slisto}{\mathsf{sList}_{\Op}}

\newcommand{\nc}{N_\D{}}
\newcommand{\ncl}{{N_\D{}}^{\!\!\!\!l\;}}

\newcommand{\monoid}{\mathsf{Monoid}}
\newcommand{\vect}{\mathsf{Vect}}
\newcommand{\Vect}{\mathsf{Vect}_{OU}}
\newcommand{\mat}{\mathsf{Mat}}
\newcommand{\Mat}{\mathsf{Mat}_{OU}}

\newcommand{\Ob}{\text{Ob}}
\newcommand{\Spx}{\mathsf{Spx}}

\newcommand{\Col}{\text{Col}}
\newcommand{\Op}{\text{Op}}
\newcommand{\mor}{\text{Mor}}

\newcommand{\CP}{\cc{P}}
\newcommand{\CL}{\cc{L}}
\newcommand{\Q}{\cc{Q}}
\newcommand{\C}{\cc{C}}
\newcommand{\M}{\cc{M}}
\newcommand{\A}{\cc{A}}
\newcommand{\B}{\cc{B}}
\newcommand{\V}{\cc{V}}
\newcommand{\K}{\cc{K}}
\newcommand{\F}[1]{\cc{F}^{(#1)}}

\newcommand{\CF}{\cc{F}}

\newcommand{\lx}{\cc{L} X}

\newcommand{\llq}{\cc{L} \cc{Q}}
\newcommand{\VVa}{\bb{V}_\alpha}

\newcommand{\VV}{\bb{V}}
\newcommand{\UUa}{\bb{U}_\alpha}

\newcommand{\UU}{\bb{U}}
\newcommand{\LQ}{\cc{L} \cc{Q}}
\newcommand{\LM}{\cc{L} \cc{M}}
\newcommand{\LP}{\cc{L} \cc{P}}
\newcommand{\fc}{\mathfrak{C}}
\newcommand{\fo}{\mathfrak{O}}

\DeclareMathOperator*{\colim}{colim}

\def\ua{{\underline{a}}}
\def\ub{{\underline{b}}}

\def\ux{{\underline{x}}}

\newcommand{\listdelta}{\widetilde{\mathsf{sList}}}
\newcommand{\listtwo}{\mathsf{List}^{2}}
\newcommand{\listdeltatwo}{\widetilde{\mathsf{sList}^2}}
\newcommand{\Ab}{\mathsf{Ab}}
\newcommand{\sAb}{\mathsf{sAb}}
\newcommand{\grAb}{\mathsf{grAb}}
\newcommand{\cAb}{\text{Ch}_{\geq 0}(\mathsf{Ab})}

\usepackage{biblatex}

\begin{document}

\maketitle

\begin{abstract}
    We define a category $\llist$ whose objects are sets and morphisms are mappings which assign to an element in the domain an ordered sequence (list) of elements in the codomain. We introduce and study a category of simplicial objects $\slist$ whose objects are functors $\D{op} \to \llist$, which we call simplicial lists, and morphisms are natural transformations which have functions as components. We demonstrate that $\slist$ supports the combinatorics of (non-symmetric) operads by constructing a fully-faithful nerve functor $N^l : \operad \to \slist$ from the category of operads. This leads to a reasonable model for the theory of non-symmetric $\infty$-operads. 

    We also demonstrate that $\slist$ has the structure of a presheaf category. In particular, we study a subcategory $\slisto$ of operadic simplicial lists, in which the nerve functor takes values. The latter category is also a presheaf category over a base whose objects may be interpreted as leveled trees. We construct a coherent nerve functor which outputs an $\infty$-operad for each operad enriched in Kan complexes.  We also define homology groups of simplicial lists and study first properties.
\end{abstract}

\tableofcontents

\section{Introduction and summary}

This paper presents a new approach to the combinatorics of non-symmetric operads. We propose a simplicial framework via simplicial lists. First, we define a convenient category of sets, denoted $\llist$, whose morphisms are mappings which assign to an element in the domain a sequence of elements in the codomain. A simplicial list is a functor $\D{op} \to \llist$. 

A central contention of ours is that, in a strong sense, (certain) simplicial lists are to non-symmetric operads what simplicial sets are to categories. 
Many of the constructions and theorems in this paper provide justification for this contention. For example, we construct a nerve functor and prove a  nerve theorem which leads us to a notion of \emph{quasi-operad}. 

The intuition underlying most of our constructions and results is deceptively simple. Thus, in order to make the latter more technical sections more digestible, we wish to discuss some preliminary ideas in this introduction. We begin by a short discussion on the notion of operad and $\infty$-operad.

\subsection*{A few words on operads and $\infty$-operads}

A non-symmetric \emph{operad} $\CP$ is comprised of
\begin{itemize}
    \item [-] Objects $a, b, c \dots$, traditionally referred to as \emph{colors}.
    \item[-] Multimorphisms, which have a codomain and, unlike categorical morphisms, have a sequence of colors serving as domain. Such multimorphisms are referred to as \emph{operations} and are typically depicted as corollas
    
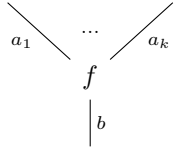
\[\begin{tikzcd}
	{} & {} & {} \\
	& f \\
	& {}
	\arrow["{a_k}", no head, from=1-3, to=2-2]
	\arrow["b", no head, from=2-2, to=3-2]
	\arrow["\dots"{description}, draw=none, from=1-2, to=2-2]
	\arrow["{a_1}"', no head, from=1-1, to=2-2]
\end{tikzcd}\]
    to indicate that $f$ has as domain the sequence $\ua = (a_1, \dots, a_k)$ and codomain $b$. We may also write $f : \ua \to b$.
    \item[-] \emph{Composition} of operations when placing operations on top of the inputs of another operation operation. We may depict such a composition scheme as a tree
    
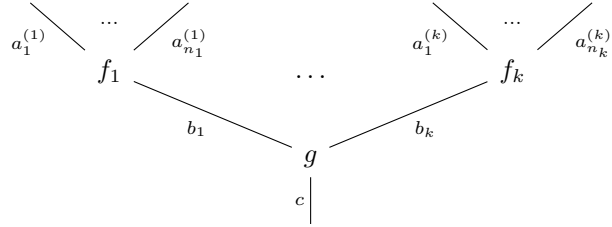
\[\begin{tikzcd}
	{} && {} && {} && {} \\
	& {f_1} && \dots && {f_k} \\
	&&& g \\
	&&& {}
	\arrow["{b_1}"', no head, from=2-2, to=3-4]
	\arrow["{b_k}", no head, from=2-6, to=3-4]
	\arrow["c"', no head, from=3-4, to=4-4]
	\arrow[""{name=0, anchor=center, inner sep=0}, "{a_1^{(1)}}", shorten >=4pt, no head, from=2-2, to=1-1]
	\arrow[""{name=1, anchor=center, inner sep=0}, "{a_1^{(k)}}"', shorten <=4pt, no head, from=1-5, to=2-6]
	\arrow[""{name=2, anchor=center, inner sep=0}, "{a_{n_k}^{(k)}}", shorten <=4pt, no head, from=1-7, to=2-6]
	\arrow[""{name=3, anchor=center, inner sep=0}, "{a_{n_1}^{(1)}}", shorten <=4pt, no head, from=1-3, to=2-2]
	\arrow["\dots"{description}, draw=none, from=0, to=3]
	\arrow["\dots"{description}, draw=none, from=1, to=2]
\end{tikzcd}\]
This tree encodes the data for operations $(f_i : \ua^{(i)} \to b_i)_i$ and $g : \ub \to c$. The composite of these operations is an operation $g \circ (f_1, \dots , f_k) : \ua^{(1)} * \dots * \ua^{(k)} \to c$ with domain the concatenation of the sequences on the top of the tree. Composition is required to be associative.

\item[-] A \emph{identity} operation $1_a : a \to a$ for each color $a$ such that composition is \emph{unital}.
\end{itemize}
In order to abbreviate terminology, throughout this paper, we refer to non-symmetric operads simply as operads. 

It is evident from the above sketch that, although operads, by virtue of having operations which compose, are objects of study in category theory, the theory of operads extends the theory of categories. Another feature is hinted at by the string diagrammatic style of depicting operations: operads generalize monoidal categories. In fact, any monoidal category $(\V, \otimes)$ may be regarded as an operad, typically denoted $\V^\otimes$, with objects those of $\V$ and operations morphisms in $\V$ of the form $f: A_1 \otimes \dots \otimes A_k \to B$.

A morphism of operads $\phi : \CP \to \Q$, which is just a structure preserving map, is also interpreted a $\Q$-valued \emph{algebra} of $\CP$. Interpreting morphisms as algebras is most useful when $\Q = V^\otimes$ originates from some monoidal category, the classical example being the category of sets with the cartesian product $\Q = \set^\times$. 
An operad $\CP$ is only as interesting as its algebras. A vast array of mathematical structures can be encoded as algebras over operads and we refer the reader to textbook accounts such as \cite{leinster2004higher}, \cite{markl2002operads} or \cite{loday2012algebraic} for examples from multiple areas of mathematics and beyond. The original formulation of operads is due to Peter May \cite{may2006geometry} in the context of loop spaces. 

The terminal operad $\assoc$, called the \emph{associative operad}, which consists of a single color and a unique operation with $k$-inputs for each $k \geq 0$, is particularly well-known. Its algebras $\assoc \to \set^\times$ are monoids and more generally its algebras in a monoidal category are monoid objects in that monoidal category. Thus, the operad $\assoc$ may be regarded as a structure which records the notion of monoid.

If we pursue this line of thought, then as soon as we try to implement these ideas in higher dimension we run out of structure. For instance,
we want to say monoidal categories are $\assoc$-algebras in $\cat^\times$, but this would only cover strict monoidal structures. This observation would lead us to consider weak 2-morphisms $\assoc \to \cat^\times$ and implement ideas from 2-category theory (a couple of examples being \cite{haderi2024OMon}, \cite{corner2013operads}). Soon enough though, we would need even higher morphisms and we run into the usual problem of inability to properly keep track of  coherence laws. 

Ultimately, we are interested in \emph{homotopy coherent} algebras $\CP \to C^\times$ which take values in a higher category $C$ (which may be incarnated in various forms, like a simplicial category or quasi-category). In case $\CP = \assoc$ and $C$ is the $\infty$-category of spaces such algebras are $A_\infty$-spaces. The investigation of coherence for algebras goes back at least to the influential work of Boardman and Vogt (\cite{boardman2006homotopy}), while a classical treatment of coherent associativity goes back to Stasheff's paper \cite{MR0158400}. While homotopy coherence can be formulated in the context of topological or simplicially enriched operads via cofibrant resolutions (also known as $W$-construction due to Boardman and Vogt), we run into practical limitations again due to the amount of data involved in resolutions. 

A good theory of $\infty$-operads encodes a notion of operad where coherence is recorded in the structure itself, meaning that composition of operations is well-defined up to coherent homotopy. A beautiful and successful approach in the context of categories is the quasi-categorical model of $\infty$-categories developed by Joyal (\cite{joyal2008notes}) and Lurie (\cite{lurie2009higher}). One of the features which makes quasi-categories practical is the fact that, being simplicial sets, many ideas from category theory, like limits and colimits, may be directly implemented semantically and the resulting notion is automatically coherent. 

The latter idea is implemented by Lurie in \cite{lurie2017higher} to model $\infty$-operads. An operad may be interpreted within category theory as a certain category which is equipped with a certain functor into finite pointed sets (see also \cite{haugseng2023allegedly} for a somewhat allegedly friendly introduction to the idea). Then, the latter is interpreted in the realm of $\infty$-categories using theory developed in \cite{lurie2009higher}. While the resulting structure encodes coherence for operads, one could argue that the direct semantic relationship between Lurie's model and classical operad theory is not transparent (at least when compared to the theory of quasi-categories in relation to ordinary categories). 

The model we present in this paper is developed in analogy with $\infty$-category theory rather than within the theory. We develop a combinatorial framework via the theory of simplicial lists with the initial aim of creating a good environment to interpret classical operad theory in a coherent manner. Certain aspects of our approach are similar to the dendroidal approach developed by Moerdijk and his collaborators (\cite{moerdijk2007dendroidal,cisinski2011dendroidal, cisinski2013dendroidalsegal,cisinski2013dendroidalsimplicial}, see also the dedicated book \cite{heuts2022simplicial}), but there are key differences which will appear as the theory is presented. Other aspects of our work are similar to models for $\infty$-operads based on Segal conditions developed by Barwick (\cite{barwick2018operator}) and Chu-Haugseng (\cite{chu2021homotopy,chu2023enriched}).

Moreover, we believe that the same approach will eventually cover coherent versions of related structures such as properads or props (analogues of the structures in \cite{hackney2015infinity}, for instance). In this paper we give hints of the possibility of such an ambition to come to fruition.

\begin{note}[Symmetries]

Roughly speaking, an operad has symmetries if one can permute inputs of operations to obtain new operations. Even though we believe our approach can extend to cover symmetric operads, in this paper we restrict ourselves to the non-symmetric case. While many examples of interest are symmetric, there are many which are most easily comprehensible without symmetries (like $\assoc$).
    
\end{note}

\subsection*{Our framework}

Some of the features that enable the combinatorics of simplicial sets to support $\infty$-category theory are the following:
\begin{itemize}
    \item [-] Kan complexes are a model for the homotopy theory of spaces.
    \item[-] There is a fully-faithful nerve functor $N : \cat \to \sset$ from the category of categories.
    \item[-] There is a coherent nerve functor $N_\D{} : \cat_\D{} \to \sset$ from the category of simplicial categories which, among other things, maps categories enriched in Kan complexes to $\infty$-categories. 
\end{itemize}

Let $\CP$ be an operad. As nice as simplicial sets are, we cannot define a nerve $N(\CP) : \D{op} \to \set$ such that
$$N(\CP)_0 = \text{colors of } \ \CP \ \ \ , \ \ \ N(\CP)_1 = \text{operations of } \ \CP$$
for the evident reason that the domain assigning function $d_1 $ would not model multimorphisms. A natural line of thought is to declare that the category $\D{}$ itself is deficient, as it is formed by simplex-like objects and their morphisms. The dendroidal setting developed by Moerdijk-Weiss (\cite{moerdijk2007dendroidal}) addresses the issue by replacing $\D{}$ with a category $\Omega$ whose objects are rooted trees and whose morphisms are maps of operads between the free operads generated by rooted trees. This way, one encodes all of the structure of an operad by mapping rooted trees into the operad. 
We will come back to the $\D{}$-replacement strategy.

The other line of thought is to say that the blame on the failure of a nerve construction for operads lies in the notion of function. After all, why does $d_1$ have to map elements of $N(\CP)_1$ to elements of $N(\CP)_0$? This leads us to consider the notion of \emph{listing}.

\begin{definition}[The category of listings] \label{def_list}

    Let $X$ be a set. We denote by $\CL X$ the set of finite sequences in $X$, i.e.
    $$\CL X = \displaystyle \coprod_{n \geq 0} X^n$$
     A listing, denoted $u : A \xslashedrightarrow{} X$, is a function $u : A \rightarrow \CL X$. 
    
     Given composable listings $u : A \xslashedrightarrow{} X$ and $v : X \xslashedrightarrow{} Y$, their composite is defined via concatenation, i.e. in case $a \in A$ and $u(a) = (x_1, \dots ,x_k)$ we have 
    $$(vu)(a) = v(x_1) * \dots * v(x_k)$$
    where $*$ denotes the concatenation operation for sequences.
    
     Let $\spa$ be the resulting category.

\end{definition}

\begin{remark}
    $\llist$ is just the Kleisli category of the free monoid monad on the category of sets. Equivalently, it is the category of free monoids. 
    In light of this, the inclusion $\set \subseteq \llist$, which regards every function as a listing with target lists being of length $1$, has a right adjoint
    \[
    \begin{tikzcd}
        \CL :&[-3em] \llist \arrow[r] & \set
    \end{tikzcd}
    \]
    
    We do not adopt these perspectives in this paper. Rather, we think of listings as certain \emph{multivalued functions}. For instance, here is a listing $u$
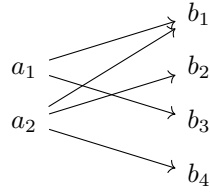
    \[\begin{tikzcd}[row sep=tiny]
	&& {b_1} \\
	{a_1} && {b_2} \\
	{a_2} && {b_3} \\
	&& {b_4}
	\arrow[from=2-1, to=1-3]
	\arrow[from=2-1, to=3-3]
	\arrow[from=3-1, to=1-3]
	\arrow[from=3-1, to=2-3]
	\arrow[from=3-1, to=4-3]
\end{tikzcd}\]
with $u(a_1) = (b_1, b_3)$ and $u(a_2) = (b_1, b_2, b_4)$. This is closer to the notion of span between sets. In fact, a listing $u : A \xslashedrightarrow{} B$ may be regarded as a span $A \xleftarrow{s} U \xrightarrow{t} B$ in which the fibres $s^{-1}(a)$ are finite and equipped with an order. Here, $U$ is the set of arrows and $s, t$ are functions assigning source and target. Note that $s^{-1}(a)$ is allowed to be empty, which means that a listing may assign the empty sequence in the codomain to an element in the domain. 
\end{remark}

For an operad $\CP$, instead of aiming at incarnating its nerve as a simplicial set, we will divert our ambition into creating a \emph{list nerve}
\[
\begin{tikzcd}
    N^l(\CP) :&[-3em] \D{op} \arrow[r] &\llist
\end{tikzcd}
\]
The $d_1$-obstruction vanishes. Indeed, we can declare a listing $d_1 : N^l(\CP)_1 \xslashedrightarrow{} N^l(\CP)_0$ which assigns to each operation its domain sequence of colors.

For the $2$-simplices, we implement the general nerve idea
$$N^l(\CP)_2 = \text{composable pairs of operations in } \ \CP$$
Let us consider the following pair of operations $(f,g)$
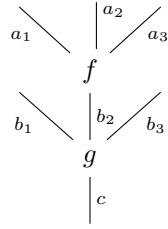
\[\begin{tikzcd}
	{} & {} & {} \\
	{} & f & {} \\
	& g \\
	& {}
	\arrow["{a_1}"', shorten <=6pt, no head, from=1-1, to=2-2]
	\arrow["{a_2}"{pos=0.2}, shift left, no head, from=1-2, to=2-2]
	\arrow["{a_3}", shorten <=6pt, no head, from=1-3, to=2-2]
	\arrow["{b_2}", no head, from=2-2, to=3-2]
	\arrow["{b_1}"', shorten <=6pt, no head, from=2-1, to=3-2]
	\arrow["{b_3}", shorten <=6pt, no head, from=2-3, to=3-2]
	\arrow["c", no head, from=3-2, to=4-2]
\end{tikzcd}\]
If we declare a $2$-simplex $\sigma$ that witnesses the composite $g \circ f$, we may declare that $d_0\sigma = g$, $d_2\sigma = f$ and $d_1\sigma = g \circ f$. However, we can compute that $d_1d_1\sigma = d_1(g \circ f) = (b_1, a_1, a_2, a_3, b_3)$, while on the other hand $d_1d_2\sigma = d_1 f = (a_1, a_2, a_3)$. So, while the depicted $\sigma$ would be a perfectly fine dendrex in the dendroidal nerve by virtue of being recorded as a tree, the simplicial identities seem to fail.

Again, there is a remedy. In the simplicial setting, the simplex $\sigma$ should really be depicted as
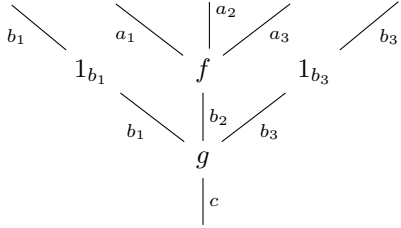
\[\begin{tikzcd}
	{} & {} & {} & {} & {} \\
	& {1_{b_1}} & f & {1_{b_3}} \\
	&& g \\
	&& {}
	\arrow["{a_1}"', shorten <=6pt, no head, from=1-2, to=2-3]
	\arrow["{a_2}"{pos=0.2}, shift left, no head, from=1-3, to=2-3]
	\arrow["{a_3}", shorten <=6pt, no head, from=1-4, to=2-3]
	\arrow["{b_2}", no head, from=2-3, to=3-3]
	\arrow["{b_1}"', no head, from=2-2, to=3-3]
	\arrow["{b_3}", no head, from=2-4, to=3-3]
	\arrow["c", no head, from=3-3, to=4-3]
	\arrow["{b_1}"', shorten <=6pt, no head, from=1-1, to=2-2]
	\arrow["{b_3}", shorten <=4pt, no head, from=1-5, to=2-4]
\end{tikzcd}\]
and, taking advantage of $\llist$, we declare that $d_0\sigma = g$, $d_2\sigma = (1_{b_1}, f, 1_{b_3})$ and $d_1\sigma = g \circ f$. This way, the simplicial identities will hold and we are lead to the general pattern our attempt at a nerve construction should follow: it must be indexed by \emph{leveled trees}.

It turns out that leveled trees of the type we are interested in are encoded by simplices of augmented simplex category category $\dplus$, which is just the category $\D{}$ with the additional empty-set object $\emptyset$ serving as initial object. For notation benefits, we think of the objects of $\dplus$ as finite ordered sets and lie up to unique isomorphism. For example, a $2$-simplex $\alpha \in \nd_2$ of the form
\[
\begin{tikzcd}
  \alpha : &[-2em]  \{c_1, a_1, a_2, a_3, c_2 \} \arrow[r, "\alpha_1"] & \{b_1, b_2, b_3 \} \arrow[r, "\alpha_2"] & \{ c\} \\[-2em]
  & c_1 \arrow[r, mapsto] & b_1 & \\[-2em]
  & a_1, a_2, a_3 \arrow[r, mapsto] & b_2 & \\[-2em]
  & c_3 \arrow[r, mapsto] & b_3 & 
\end{tikzcd}
\]
encodes the following leveled tree
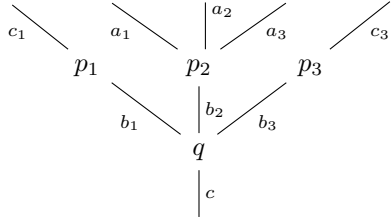
\[\begin{tikzcd}
	{} & {} & {} & {} & {} \\
	& {p_1} & {p_2} & {p_3} \\
	&& q \\
	&& {}
	\arrow["{a_1}"', shorten <=6pt, no head, from=1-2, to=2-3]
	\arrow["{a_2}"{pos=0.2}, shift left, no head, from=1-3, to=2-3]
	\arrow["{a_3}", shorten <=6pt, no head, from=1-4, to=2-3]
	\arrow["{b_2}", no head, from=2-3, to=3-3]
	\arrow["{b_1}"', no head, from=2-2, to=3-3]
	\arrow["{b_3}", no head, from=2-4, to=3-3]
	\arrow["c", no head, from=3-3, to=4-3]
	\arrow["{c_1}"', shorten <=6pt, no head, from=1-1, to=2-2]
	\arrow["{c_3}", shorten <=4pt, no head, from=1-5, to=2-4]
\end{tikzcd}\]
Notice that the morphisms $\alpha_1$ and $\alpha_2$ in the formation of $\alpha$ determine and are determined by the shape of the tree. 

We may form the free operad $T_\alpha$ associated to a simplex $\alpha \in \nd$ and define
$$N^l(\CP)_n = \displaystyle \coprod_{\alpha \in \nd_n^{root}} \operad(T_\alpha, \CP)$$
The notation $\nd^{root}$ stands for \emph{rooted simplices} of $\nd$, meaning those simplices $\alpha : \D{n} \to \dplus$ such that $\alpha(n) \cong [0]$. Note that general $\alpha$ encode ordered sequences of leveled trees\footnote{Such sequences of leveled trees appear in \cite{heuts2016equivalence} as leveled forests, except that ours is a planar variant. They also appear in \cite{barwick2018operator} (we comment on this in detail in Section \ref{subsec : operadic}).}, and this gives rise to the listings in the simplicial structure.  More precisely, we do have an isomorphism
$$\nd_n \cong \CL \nd_n^{root}$$
for all $[n] \in \D{}$.
Given a morphism $\theta : [k] \to [n]$ and $\alpha \in \nd_n^{root}$, we do have a morphism of operads $T_{\theta^* \alpha} \to T_\alpha$ which produces the listing 
\[
\begin{tikzcd}
    N^l(\CP)_n \arrow[r] & \CL N^l(\CP)_k
\end{tikzcd}
\]
by precomposition. 

What we have sketched above assembles into a fully faithful list nerve functor
\[
\begin{tikzcd}
    N^l :&[-3em] \operad \arrow[r] & \slist
\end{tikzcd}
\]

where $\slist$ is the category whose:
\begin{itemize}
    \item [-] Objects are functors $X : \D{op} \to \llist$.
    \item[-] Morphisms are natural transformations $f : X \to Y$ such that for each $[n] \in \D{}$ the component $f_n : X_n \to Y_n$ is a function.
\end{itemize}

Moreover, for an operad $\CP$, all the face and degeneracy maps in the simplicial structure of $N^l(\CP)$ are functional, except for the last faces $d_n : N^l(\CP)_n \xslashedrightarrow{} N^l(\CP)_{n-1}$. The reason is simple: every simplicial operation on a rooted simplex in $\nd$ produces a rooted simplex, except for the last face operation which cuts the root and leaves us with a  list of rooted trees. This leads us to an important class of simplicial lists, which we call \emph{operadic}.

\begin{definition}[Operadic simplicial lists] \label{def_operadic}

A simplicial list $X$ is said to be operadic if for all morphisms $\theta : [k] \to [n]$ in $\D{}$ such that $\theta(k) = n$, the action map $\theta^* : X_n \to X_k$ is a function.
    
\end{definition}

\begin{theorem}[Nerve Theorem]

Let $X$ be a simplicial list. Then, $X$ is isomorphic to $N^l(\CP)$ for some operad $\CP$ if and only if $X$ is operadic and the simplicial set $\CL X$ is (isomorphic to the nerve of) a category.
    
\end{theorem}

One way to interpret the above nerve theorem is to say that leveled trees and their simplicial relations are enough to fully record the structure of an operad. 
Here, $\CL X$ is the simplicial set obtained as a composite $\D{op} \xrightarrow{X} \llist \xrightarrow{\CL} \set$. Its simplices are simply list of simplices of $X$ and hence the listings in the simplicial structure of $X$ become functions in the structure of $\CL X$. 

\begin{definition}[Quasi-operad]

A quasi-operad is an operadic simplicial list $P$ such that the simplicial set $\CL P$ is a quasi-category.
    
\end{definition}

There is a sense in which we are justified in claiming that quasi-operads as defined above do model $\infty$-operads. This is in light of our observations and the well-established quasi-categorical model for the theory of $\infty$-categories. Depending on context, we may refer to quasi-operads as $\infty$-operads, since it is quite common in the literature (the main example being \cite{lurie2009higher}) to refer to quasi-categories as $\infty$-categories.

We will provide an additional piece of evidence which indicates that our notion of $\infty$-operad actually captures some classical theory. 
It is possible to construct simplicial operads $\VVa$ associated with leveled tree shapes $\alpha \in \nd^{root}$, which are well-behaved cofibrant replacements of the operads $T_\alpha$, and define a homotopy coherent list nerve functor
\[
\begin{tikzcd}
    \ncl :&[-3em] \soperad \arrow[r] &\slist
\end{tikzcd}
\]
which assigns to a simplicial operad $\Q$ sets of simplices 
$$(\ncl \Q)_n = \displaystyle \coprod_{\alpha \in \nd^{root}_n} \soperad(\VVa , \Q)$$
for all $[n] \in \D{}$.
\begin{theorem}
    If a simplicial operad $\Q$ is enriched in Kan complexes, $\ncl \Q$ is a quasi-operad.
\end{theorem}

Besides the combinatorial capability of $\slist$ to support both nerves and coherent nerves of operads, it has, at least to us, a surprising property: it is a presheaf category!

\begin{theorem}
    There is a full subcategory $\cc{F} \subseteq \slist$ such that $\slist \cong \set^{\cc{F}^{op}}$.
\end{theorem}

We construct the subcategory $\cc{F}$ combinatorially, but in rather abstract fashion. A purely combinatorial description of $\cc{F}$ remains to be explored. However, when we restrict to operadic simplicial lists, we still have a presheaf structure with base the operadic objects in $\cc{F}$. The combinatorial structure of the latter is summarized in the following theorem.

\begin{theorem}

There is a full subcategory $\Upsilon \subseteq \slisto$ of the category of operadic simplicial lists such that 
$$\slisto \cong \set^{\Upsilon^{op}}$$
Moreover, the category $\Upsilon$ admits the following description:
\begin{itemize}
    \item [-] The objects are the rooted simplices of $\nd$ (leveled trees).
    \item[-]  Given two such simplices $\alpha \in \Upsilon^{(n)}$ and $\beta \in \Upsilon^{(k)}$, where $n$ and $k$ indicate simplicial dimension in $\nd$, a morphism $(\theta , a) : [\beta] \to [\alpha]$ is a pair consisting of a morphism $\theta : [k] \to [n]$ and an element $a \in \alpha(\theta(k))$ such that $\beta = (\theta^*\alpha)_a$. The latter indicates the rooted simplex with root $a$ which appears in the sequence of rooted simplices of $\theta^* \alpha$.
\end{itemize}
    
\end{theorem}

Here is an example of a morphism in $\Upsilon$

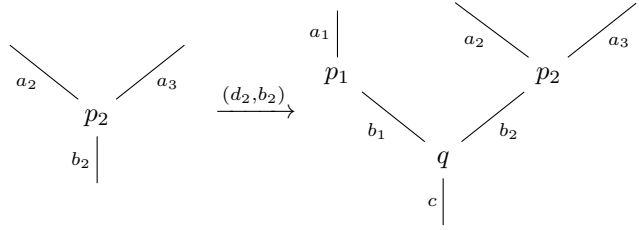
\[\begin{tikzcd}
	{} && {} \\
	& {p_2} \\
	& {}
	\arrow["{a_2}"', no head, from=1-1, to=2-2]
	\arrow["{a_3}", no head, from=1-3, to=2-2]
	\arrow["{b_2}"', no head, from=2-2, to=3-2]
\end{tikzcd}
\xrightarrow{(d_2, b_2)}
\begin{tikzcd}
	{} & {} && {} \\
	{p_1} && {p_2} \\
	& q \\
	& {}
	\arrow["{a_1}"', no head, from=1-1, to=2-1]
	\arrow["{a_2}"', no head, from=1-2, to=2-3]
	\arrow["{a_3}", no head, from=1-4, to=2-3]
	\arrow["{b_1}"', no head, from=2-1, to=3-2]
	\arrow["{b_2}", no head, from=2-3, to=3-2]
	\arrow["c"', no head, from=3-2, to=4-2]
\end{tikzcd}\]
formed by first applying the simplicial operation $d_2$ on the target, which cuts the operation $q$, and then choosing the root $b_2$ in the remaining sequence of rooted trees. 

At the end of the day, we did replace $\D{}$ with a category of trees. However, we still retain the lists perspective which allows us to directly make analogy with the classical theory of simplicial sets. On the other hand, the presheaf structure allows us to perform constructions faster via Yoneda extension. In comparison to the dendroidal setting, our base category $\Upsilon$ has a purely combinatorial description since the morphisms between leveled trees we consider are not all morphisms between the free operads they generate\footnote{This is not to say the dendroidal approach is not combinatorial. The morphisms between dendroidal trees do factor into faces, degeneracies and automorphisms. What we mean is that this combinatorial structure is formal, while the one we present is explicit.}.

Besides dendroidal sets, the referee pointed out another strong connection to already existing ideas, more precisely to the structures developed by Barwick in \cite{barwick2018operator} and Chu-Haugseng in \cite{chu2021homotopy,chu2023enriched}. Indeed, $\Upsilon$ may be regarded as a full subcategory of the category $\D{}_{\Phi}$ in \cite[Definition 2.4]{barwick2018operator} for the case $\Phi = \dplus$. The latter category also appears as an example of an \emph{algebraic pattern} in \cite[Example 3.8]{chu2021homotopy}.

In both of the above bodies of work, a model for $\infty$-operads is defined via Segal conditions, mostly generalizing ideas developed by Lurie in \cite{lurie2017higher} (such as operator categories or the inert-active factorization). In fact, it is easy to derive from the aforementioned references that Segal $\Upsilon$-spaces (say, as in \cite[Definition 2.7]{chu2021homotopy}) do model $\infty$-operads.
In this paper, we work in the quasi-categorical direction as we believe this approach has advantages due to combinatorial tractability.

\subsection*{Organization of this paper} 

We recall some basic elements of the theory of operads in Section \ref{section : operads}, mostly with the intention of fixing notation and terminology. In Section \ref{section : first constructions} we introduce the important notion of multigraph and construct a free-forgetful adjunction with operads. Then, we take a brief detour into the free resolution this adjunction generates. Afterwards, we construct our nerve functor (already sketched to some amount of detail above) and prove our nerve theorem which motivates our model for $\infty$-operads. In Section \ref{sec : structure} we prove our presheaf theorems. This section is the combinatorial heart of this paper, as we introduce some key techniques to deal with simplicial lists. In Section \ref{section : coherent nerve} we define our coherent nerve functor, building on results from the previous sections. We prove that coherent nerves of Kan-enriched operads provide examples of $\infty$-operads. Moreover, we study some properties of the rigidification functor obtained by adjunction. In Section \ref{section : homology} we briefly discuss a natural variant of homology for simplicial lists. The short appendices \ref{appendix_Aug} and \ref{appendix_free} are dedicated to augmented simplicial objects and free resolutions, as these structures appear quite often in this paper. 

\subsection*{Some connections with the literature}

Even though the above ideas are developed independently, many of them are not quite new. The evident connection to dendroidal theory and to the Segal-condition based notions in \cite{barwick2018operator}, \cite{chu2021homotopy} has already been mentioned. The idea of encoding trees via multivalued functions appears in \cite{kock2011polynomial}. In the context of symmetric $\infty$-properads, the ideas proposed in \cite{barkan2022equifibered} seem to be aligned with the spirit of our ideas (with a key difference being that their approach is internal to $\infty$-category theory). The idea of keeping track of levels in trees is present in \cite{heuts2016equivalence} (in the notion of forest) where an equivalence between Lurie's model and the dendroidal model is established, in \cite{bonventre2020rigidification} where a necklace theorem for dendroidal rigidification is presented, but also in the aforementioned \cite{barwick2018operator} , \cite{chu2021homotopy} and \cite{chu2023enriched}. Leveled graphs appear in the context of properads in \cite{chu2022rectification} and are also used in \cite{beardsley2024labelled}. 

\subsection*{Future goals}

While the scent of homotopy theory can be discerned in this paper, we do not speak of a model structure on $\slist$. In a sequel paper, we intend to provide a model structure on $\slist$ which is Quillen equivalent via the coherent nerve to the model structure on $\soperad$ (\cite{robertson2011homotopy}). A study of the comparison to the dendroidal model, but also to Lurie's model and Segal $\Upsilon$-spaces, should be possible via model structures. 
Other model structures such as the cartesian model structure ought to be studied in the context of simplicial lists.
Another objective of ours is to develop an appropriate notion of symmetric simplicial list which models symmetric operads and $\infty$-operads. 

Moreover, we expect to be able to say something about properads and $\infty$-properads via simplicial lists. For instance, we can talk about properadic simplicial list by allowing the first face maps to be listings (besides the last face maps). A priory, it is already known to us that properadic simplicial lists have a presheaf structure. The work to do is in unpacking the combinatorics and construct good nerve functors. 

In fact, at least in the non-symmetric setting, an interesting object of study, which we would like to call compositional network, appears in front of us.

\begin{definition} \label{def: conet}
    A \emph{compositional network} is a simplicial list $X$ such that $\CL X$ is (the nerve of) a category.
\end{definition}

These objects seem to generalize both operads and properads, and their algebras might be interesting. 

Another interesting direction to pursue is to check under what circumstances our constructions can be carried our when we work relative to a given monad $T$ on some category $\C$. Is there a nerve functor which maps a $T$-multicategory (in the sense of Leinster \cite[Definition 4.2.2]{leinster2004higher}) to a simplicial object in the Kleisli category $\C_T$? In fact, the special case of the monad $\textbf{fc}$ on $\cat$ induced by the free -forgetful adjunction with graphs, leads to the notion of virtual double category. A homotopy coherent version has been worked out, in the style of Lurie, in \cite{gepner2015enriched} in order to study enrichment. 

While we briefly discuss a variant of homology for simplicial lists, other variants might be possible and more sophisticated results and applications are to be searched for. 

\subsection*{Some conventions}

In general, we name categories after the name of their objects. For example, $\set$ is the category of sets, $\operad$ is the category of operads, $\Ab$ is the category of abelian groups etc. One exception is the category $\llist$ which is named after its morphisms. 

When it comes to simplicial objects we follow fairly standard notation. $\D{}$ is the simplex category, or equivalently the category of finite ordinals, whose objects we denote $[n] = \{0, \dots, n \}$. A simplicial set if a functor $X : \D{op} \to \set$. We denote by $X_n$ the image of $[n]$ under $X$. For a general morphism $\theta : [n] \to [m]$ we write $\theta^* : X_m \to X_n$ for the image of $\theta$ under the functor $X$, while we denote by $d_i$ and $s_i$ the face and degeneracy functions.

We denote by $\dplus$ the augmented simplex category, obtained by adjoining an initial object $\emptyset$ to $\D{}$. For notational benefits, we regard any finite ordered set as an object in $\dplus$. We are justified in this abuse since for any finite ordered set $A$ there is a unique $[n] \in \dplus$ equipped with a unique isomorphism $A \cong [n]$. 

As for size issues, while we do not comment much on them in this paper, there is nothing problematic going on. For instance, one can apply the technology of Grothendieck universes. We refer the reader to \cite{shulman2008set} for more details on the subject.

\subsection*{Acknowledgements}

The first author would like to thank his friend and colleague Walker Stern for numerous beneficial conversations on higher categories and higher operads, for the many useful suggestions during the preparation of this document and for providing him generous feedback on earlier drafts of this paper. The first author has also benefited from correspondence with Philip Hackney on infinity operads and infinity properads. 

We are also grateful to the anonymous referee for pointing out useful references in the liteature, correcting mistakes and providing beneficial comments on the presentation of the material.

\section{Some elements of the theory of operads} \label{section : operads}

We briefly recall some elements of the theory of operads which are relevant to us, mainly for the purpose of fixing notation and language. We are particularly interested in the relationship between operads and categories.

\begin{definition}[Colored operad] 

A (colored, non-symmetric) operad $\CP$ consists of a collection of colors and operations of the form $f : \ua \to b$ which admit as domain sequences of colors $\ua = (a_1, \dots, a_k)$. These collections are denoted by $\Col(\CP)$ and $\Op(\CP)$ respectively, while $\CP(\ua, b)$ is the set of operations with specified source and target. Typically, an operation $f : \ua \to b$ is depicted as a corolla
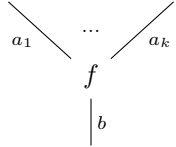
\[\begin{tikzcd}
	{} & {} & {} \\
	& f \\
	& {}
	\arrow["{a_k}", no head, from=1-3, to=2-2]
	\arrow["b", no head, from=2-2, to=3-2]
	\arrow["\dots"{description}, draw=none, from=1-2, to=2-2]
	\arrow["{a_1}"', no head, from=1-1, to=2-2]
\end{tikzcd}\]

An operation $g : \ub \to c$ may be composed with a sequence of operations $(f_i : \ua^{(i)} \to b_i)_i$
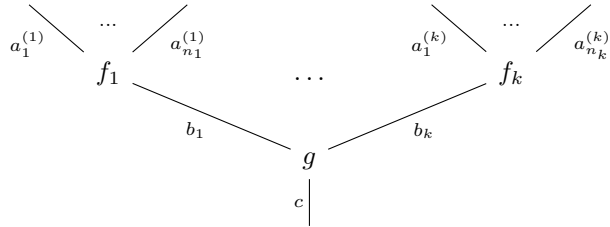
\[\begin{tikzcd}
	{} && {} && {} && {} \\
	& {f_1} && \dots && {f_k} \\
	&&& g \\
	&&& {}
	\arrow["{b_1}"', no head, from=2-2, to=3-4]
	\arrow["{b_k}", no head, from=2-6, to=3-4]
	\arrow["c"', no head, from=3-4, to=4-4]
	\arrow[""{name=0, anchor=center, inner sep=0}, "{a_1^{(1)}}", shorten >=4pt, no head, from=2-2, to=1-1]
	\arrow[""{name=1, anchor=center, inner sep=0}, "{a_1^{(k)}}"', shorten <=4pt, no head, from=1-5, to=2-6]
	\arrow[""{name=2, anchor=center, inner sep=0}, "{a_{n_k}^{(k)}}", shorten <=4pt, no head, from=1-7, to=2-6]
	\arrow[""{name=3, anchor=center, inner sep=0}, "{a_{n_1}^{(1)}}", shorten <=4pt, no head, from=1-3, to=2-2]
	\arrow["\dots"{description}, draw=none, from=0, to=3]
	\arrow["\dots"{description}, draw=none, from=1, to=2]
\end{tikzcd}\]
to produce an operation $g \circ (f_1, \dots , f_k) : \ua^{(1)} * \dots * \ua^{(k)} \to c$ whose domain the concatenation of the sequences on the top of the tree. For each color $a \in \CP$, there is an identity operation $1_a : a \to a$. The composition operation is associative and unital.
    
\end{definition}

\begin{remark}
    In this paper, the term operad always means \emph{colored non-symmetric operad}. The literature varies on this point. Sometimes symmetries are assumed, sometimes a single color is postulated, and sometimes the colors and symmetries are postulated in a definition of operad.
\end{remark}

\begin{remark}
    The general composition scheme postulated in the definition also entails composition along a single input, by inserting identities at the other inputs. For example, in case we have $f : a \to b_1$ and $g : (b_1, b_2) \to c$ we may form $g \circ f = g \circ (f, 1_{b_2})$
    \[\begin{tikzcd}[column sep = small, row sep = 1.5 em]
	{} \\
	f && {} \\
	& g \\
	& {}
	\arrow["a"', no head, from=1-1, to=2-1]
	\arrow["{b_1}"', no head, from=2-1, to=3-2]
	\arrow["c"', no head, from=3-2, to=4-2]
	\arrow["{b_2}", no head, from=2-3, to=3-2]
\end{tikzcd} = 
\begin{tikzcd}[column sep = small, row sep = 1.5 em]
	{} && {} \\
	f && {1_{b_2}} \\
	& g \\
	& {}
	\arrow["a"', no head, from=1-1, to=2-1]
	\arrow["{b_1}"', no head, from=2-1, to=3-2]
	\arrow["c"', no head, from=3-2, to=4-2]
	\arrow["{b_2}", no head, from=2-3, to=3-2]
	\arrow["{b_2}", no head, from=1-3, to=2-3]
\end{tikzcd} \mapsto
\begin{tikzcd}[column sep = small, row sep = 1.5 em]
	{} && {} \\
	& {g \circ f} \\
	& {}
	\arrow["c"', no head, from=2-2, to=3-2]
	\arrow["{b_2}", no head, from=1-3, to=2-2]
	\arrow["a"', no head, from=1-1, to=2-2]
\end{tikzcd}
\]
It is possible to rewrite the definition of operad in terms of single input composites.
\end{remark}

\begin{definition}[Morphisms a.k.a. algebras]
    A morphism of operads $\phi : \CP \to \Q$ is a structure preserving map, i.e. it consists of  functions on colors and operations which respect source, target, identities and composition. Let $\operad$ be the category of operads under this notion of morphism.
    Depending on context, morphisms $\CP \to \Q$ are also referred to as $\CP$-algebras in $\Q$. 
\end{definition}

Note that the definition of operad allows the empty sequence of colors $\emptyset$ in the domain. Operations of the form $\emptyset \to b$ are called \emph{nullary}, while operations of the form $a \to b$ with domain a singleton sequence are called \emph{unary}. Hence, every operad $\CP$ has incorporated in it a category $[\CP]_1$ whose objects are the colors of $\CP$ and morphisms are the unary operations of $\CP$. In fact, we have a functor
\[
\begin{tikzcd}
    \left[- \right]_1 :&[-3em] \operad \arrow[r] & \cat
\end{tikzcd}
\]
which is right adjoint to the inclusion functor $\cat \subseteq \operad$ which regards every category as an operad with unary operations only. This is the first interesting relationship between operads and categories. 

\begin{definition}[Equivalence of operads]

A morphism of operads $\phi : \CP \to \Q$ is said to be an equivalence if it induces bijections on operation sets and it is essentially surjective. The latter means that $[\phi]_1$ is essentially surjective. 
    
\end{definition}

\begin{remark}[Enriched operads]
    Operads accept enrichment over any monoidal category $\V$. This is to say that there is a $\V$-object of operations $\CP(\ua, b)$ for all sequences of colors $\ua$ and color $b$ such that the relevant diagrams in defining an operad commute in $\V$. In this paper we are interested in $\soperad$, the category of operads enriched in simplicial sets. Enrichment for operads is discussed in far greater generality in \cite[Section 6.8]{leinster2004higher}.
\end{remark}

\begin{example}[Monoidal categories]
    As mentioned in the introduction, every monoidal category $(\V, \otimes, 1_\V)$ gives rise to a colored operad $\V^\otimes$ whose colors are the objects of $\V$ and operations $\underline{A} \to B$ are morphisms of the form $A_1 \otimes \dots \otimes A_k \to B$. 
    \[\begin{tikzcd}
	{} && {} \\
	& f \\
	& {}
	\arrow["B"', no head, from=2-2, to=3-2]
	\arrow[""{name=0, anchor=center, inner sep=0}, "{A_k}", no head, from=1-3, to=2-2]
	\arrow[""{name=1, anchor=center, inner sep=0}, "{A_1}"', no head, from=1-1, to=2-2]
	\arrow["\dots"{description}, draw=none, from=1, to=0]
\end{tikzcd} = 
\begin{tikzcd}
	{A_1 \otimes \dots \otimes A_k} \\
	\\
	B
	\arrow["f", from=1-1, to=3-1]
\end{tikzcd}
\]
    
    Morphisms $1_\V \to A$ from the unit of the monoidal structure are regarded as nullary operations, while composition of operations is provided by composition in $\V$ (and the fact that $\otimes$ is functorial). 
    It is easy to see that a monoidal functor gives rise to a morphism of operads, so that the construction $(\V, \otimes, 1_\V) \mapsto \V^\otimes$ defines a functor
    \[
    \begin{tikzcd}
        (-)^\otimes :&[-3em] \moncat \arrow[r] &\operad
    \end{tikzcd}
    \]
    from the category of monoidal categories to operads. 
\end{example}

\begin{example}[The monoidal envelope]
    Each operad $\CP$ gives rise to a monoidal category, called the monoidal envelope of $\CP$, which we denote $\LP$ (recall that for a set $X$, $\CL X$ denotes the set of finite sequences in $X$). This category can be described as follows:
    \begin{itemize}
        \item [-] Its objects are sequences of colors of $\CP$, i.e. $\Ob(\LP) = \CL \Col (\CP)$.
        \item[-] Its morphisms are sequences of operations in $\CP$, i.e. $\mor(\LP) = \CL \Op (\CP)$. For instance,
        \[\begin{tikzcd}
	{(a_1, a_2, a_3, a_4)} \\
	\\
	{(b_1, b_2, b_3)}
	\arrow["{(f_1, f_2, f_3)}"', from=1-1, to=3-1]
\end{tikzcd}
=
\begin{tikzcd}
	{} & {} & {} &&& {} \\
	& {f_1} && {f_2} && {f_3} \\
	& {} && {} && {}
	\arrow["{a_1}"', shorten <=4pt, no head, from=1-1, to=2-2]
	\arrow["{b_1}"', no head, from=2-2, to=3-2]
	\arrow["{a_3}", shorten <=4pt, no head, from=1-3, to=2-2]
	\arrow["{b_2}"', no head, from=2-4, to=3-4]
	\arrow["{b_3}"', no head, from=2-6, to=3-6]
	\arrow["{a_4}"', no head, from=1-6, to=2-6]
	\arrow["{a_2}"'{pos=0.2}, no head, from=1-2, to=2-2]
\end{tikzcd}
\]
        \item[-] Source, target and identities are assigned in the evident manner, while composition is provided by composing operations in $\CP$.  
        \item[-] The monoidal structure is provided by concatenation of sequences. 
    \end{itemize}

    A non-internal way to describe $\LP$ is to say that, for two sequences of colors $\ua : A \to \Col (\CP)$ and $\ub : B \to \Col (\CP)$ indexed by ordered sets $A$ and $B$, we have 
    $$\LP (\ua, \ub) = \displaystyle \coprod_{\alpha : A \to B} \left(\bigtimes_{b \in B} \CP(\ua_b, b) \right)$$
    where $\alpha : A \to B$ signifies a morphism in $\dplus$ and $\ua_b$ indicates the restriction of $\ua$ to $\alpha^{-1}(b)$. We obtain a functor 
    \[
    \begin{tikzcd}
        \CL :&[-3em] \operad \arrow[r] & \moncat
    \end{tikzcd}
    \]
    although many times we disregard the monoidal structure in $\LP$. Notice that the
second description of $\LP$ can be extended to cover the enriched case (if the base
monoidal category has finite coproducts which distribute over the monoidal product, which is the case for complete and cocomplete cartesian monoidal categories).

This construction is discussed in \cite[Section 2.3]{leinster2004higher}, while the term monoidal envelope was introduced by Lurie in \cite[Section 2.2.4]{lurie2017higher} in the context of $\infty$-operads.
\end{example}

\begin{example}[The associative operad]

We already discussed $\assoc$, the terminal operad which records monoid laws. Note that we have an isomorphism $\CL \assoc \cong \dplus$. 
    
\end{example}

\begin{remark}[Simplicial operads] \label{remark : L discrete}

Just as in the case of simplicial categories, we may regard a simplicial operad $\Q$ as a simplicial object $\Q_\bullet : \D{op} \to \operad$ which is \emph{discrete}. For $[n] \in \D{}$, the operad $\Q_n$ has as colors those of $\Q$ and $n$-simplices in the operation spaces of $\Q$ as operations. We say $\Q_\bullet$ is discrete to indicate the fact that simplicial operations are identity-on-colors morphisms of operads. 

In this vein, we may identify $\LQ$ with the discrete simplicial category $\LQ_\bullet$. In other words, the following diagram commutes
\[\begin{tikzcd}
	\soperad & \scat \\
	{\operad^{\D{op}}} & {\cat^{\D{op}}}
	\arrow["\subseteq"', from=1-1, to=2-1]
	\arrow["\CL", from=1-1, to=1-2]
	\arrow["\CL"', from=2-1, to=2-2]
	\arrow["\subseteq", from=1-2, to=2-2]
\end{tikzcd}\]
where the functor on the bottom is obtained by applying $\CL : \operad \to \cat$ dimension-wise. 
We will freely use these identifications in this paper. 
    
\end{remark}

\begin{lemma}
    The functor $\CL$ is left adjoint to $(-)^\otimes$.
\end{lemma}

\begin{proof}
    This is easy to verify.
\end{proof}

\begin{remark}
    The adjunction $(\CL , (-)^\otimes)$ is in fact a free-forgetful one. It is apparent how $\CL$ is free, but perhaps it is not clear how $(-)^\otimes$ is forgetful. See \cite[Theorem 3.3.4]{leinster2004higher} for a characterisation of monoidal categories as operads with some extra conditions. 
\end{remark}

\begin{example}[The $\mathsf{Hom}$ operad] \label{example : hom operad}

Let $S$ be a fixed set. There is an operad $\mathsf{Hom}_S$ with set of colors $S \times S$ and, for each sequence of elements $a_1, \dots a_n \in S$, a unique operation $((a_1, a_2), \dots , (a_{n-1}, a_n)) \to (a_1, a_n)$, and for each element $a \in S$ a unique nullary operation $\emptyset \to (a,a)$. Composition is defined in the evident manner. An algebra $\mathsf{Hom}_S \to \set^\times$ is precisely a category with set of objects $S$, while algebras values in a monoidal category $\V$ are $\V$-enriched categories with set of objects $S$. 

This operad appears as Example 2.4 in \cite{moerdijk2007dendroidal} for instance. Our choice of notation is to indicate that, in creating an algebra, we equip the set $S$ with hom-objects in a monoidal category of choice.
    
\end{example}

On the homotopical front, we are interested in the behaviour of the functor $\CL$. The connected components functor $\pi_0 : \sset \to \set$ preserves products and hence induces a functor
\[
\begin{tikzcd}
    \pi_0 :&[-3em] \soperad \arrow[r] &\operad
\end{tikzcd}
\]
which takes connected components of operation spaces in a simplicial operad. 

\begin{lemma}
    The following square commutes (up to natural isomorphism)
    \[\begin{tikzcd}
	\soperad & \scat \\
	\operad & \cat
	\arrow["{\pi_0}"', from=1-1, to=2-1]
	\arrow["{\pi_0}", from=1-2, to=2-2]
	\arrow["\CL", from=1-1, to=1-2]
	\arrow["\CL"', from=2-1, to=2-2]
\end{tikzcd}\]
\end{lemma}

\begin{proof}
    Let $\Q$ be a simplicial operad. The categories $\pi_0(\LQ)$ and $\CL(\pi_0 \Q)$ both have $\CL \Col (\Q)$ as set of objects. 
    Let $\ua : A \to \Col (\Q)$ and $\ub : B \to \Col (\Q)$ be two sequences of colors. The mapping space in $\LQ$ is given by
    $$(\LQ)(\ua, \ub) = \displaystyle \coprod_{\alpha : A \to B} (\bigtimes_{b \in B} \Q( \ua_b,  b))$$
    Since $\pi_0$ preserves coproducts (besides products), we have a natural isomorphism $\pi_0(\LQ)(\ua, \ub) \cong \CL (\pi_0 \Q(\ua, \ub)) $.
    
\end{proof}

The canonical model structure on $\cat$, which has as weak equivalences the equivalences of categories and as fibrations isofibrations\footnote{Recall that a functor $F : \A \to \B$ between categories is an isofibration in case any isomorphism in $\B$ of the form $F(a) \xrightarrow{\cong} b$ lifts to an isomorphism $a \xrightarrow{\cong} b^\prime$ in $\A$.}, extends to operads in straightforward fashion.

\begin{theorem}[Canonical model structure]

There is a model structure on $\operad$ in which a morphism $\phi : \CP \to \Q$ is declared to be:
\begin{itemize}
    \item[-] A weak equivalence if it is an equivalence of operads.
    \item[-] A fibration if the underlying functor $[\phi]_1 : [\CP]_1 \to [\Q]_1$ is an isofibration. 
\end{itemize}
    
\end{theorem}

\begin{proof}
    The statement and proof appears as Theorem 1.6.2 in \cite{weiss2007dendroidal}, but also in \cite[Theorem 2]{robertson2011homotopy}. The proofs are written for the symmetric case, but hold without symmetries as well.
\end{proof}

\begin{lemma}
    The functor $\CL : \operad \to \cat$ preserves weak equivalences and fibrations (and hence trivial fibrations).
\end{lemma}

\begin{proof}
    Let $\phi : \CP \to \Q$ be a morphism of operads. It is clear that if $\phi$ is an equivalence of operads, then $\CL \phi$ is an equivalence of categories. Assume $\phi$ is a fibration. Any isomorphism $\phi(\ua) \xrightarrow{\cong} \ub$ in $\LQ$ is determined by a sequence of isomorphisms $(\phi (a_i) \xrightarrow{\cong} b_i)_i$ in $[\Q]_1$. Since $[\phi]_1$ is an isofibration, we can lift the latter sequence term by term to an isomorphism $\ua \to \ub^{\prime}$ in $\LP$. 
\end{proof}

In similar fashion, the Bergner model structure on simplicial categories (\cite{bergner2007model}) extends to simplicial operads.

\begin{theorem}[Model structure on simplicial operads]

The category $\soperad$ of simplicial operads has model structure in which a morphism $\phi : \K \to \Q$ is said to be a weak equivalence if 
\begin{itemize}
    \item[(W1)] For all sequences of colors $\ux$ and color $y$ in $\K$, the morphism of simplicial sets $\phi_{\ux,y} : \K(\ux, y) \to \Q(\phi \ux, \phi y)$ is a weak equivalence of simplicial sets.
    
    \item [(W2)] The morphism of categories $\pi_0[\phi]_1 : \pi_0[\K]_1 \to \pi_0[\Q]_1$ is an equivalence.
    
\end{itemize}
and a fibration if
\begin{itemize}
    \item[(F1)] For all sequences of colors $\ux$ and color $y$ in $\K$, the morphism of simplicial sets $\phi_{\ux,y} : \K(\ux, y) \to \Q(\phi \ux, \phi y)$ is a fibration of simplicial sets.
    
    \item [(F2)] The morphism of categories $\pi_0[\phi]_1 : \pi_0[\K]_1 \to \pi_0[\Q]_1$ is an isofibration.
    
\end{itemize}
    
\end{theorem}

\begin{proof}
    This is stated and proved for symmetric operads in \cite[Theorem 6]{robertson2011homotopy}, but holds without symmetries as well.
\end{proof}

\begin{lemma} \label{lemma : Lwefib}
    The functor $\CL : \soperad \to \scat$ preserves weak equivalences and fibrations (and hence, trivial fibrations).
\end{lemma}

\begin{proof}
    Suppose $\phi : \K \to \Q$ is a weak equivalence of simplicial operads. Let $\ua : A \to \Col (\K)$ and $\ub : B \to \Col (\K)$ be two objects in $\CL \K$. The morphism on mapping spaces
    \[
    \begin{tikzcd}
        (\CL \phi)_{\ua, \ub} :&[-2em] \displaystyle \coprod_{\alpha : A \to B} \left(\bigtimes_{b \in B} \K(\ua_b, b)\right) \arrow[r] & \displaystyle \coprod_{\alpha : A \to B} \left(\bigtimes_{b \in B} \Q(\phi \ua_b, \phi b)\right)
    \end{tikzcd}
    \]
    is componentwise a weak equivalence of simplicial sets, and hence it is a weak equivalence itself. The functor $\pi_0 (\CL \phi) : \pi_0(\CL \K) \to \pi_0 (\LQ)$ is also an equivalence of categories because, by the previous lemma we have $\pi_0(\CL \phi) \cong \CL (\pi_0 \phi)$ and, since $\pi_0 \phi$ is an equivalence of operads and $\CL$ preserves equivalences, $\CL (\pi_0 \phi)$ is an equivalence. 

    In case $\phi$ is a fibration, it is clear by the above expression of the map induced by $\CL \phi$ on mapping spaces that the latter is a fibration of simplicial sets (fibrations of simplicial sets are closed under products and disjoint unions). The fact that $\pi_0(\CL \phi)$ is an isofibration, as in the previous argument, follows from the fact that $\CL$ commutes with $\pi_0$ and that on the level of categories $\CL$ preserves isofibrations.
\end{proof}

\begin{remark}
    We have used the letter $\CL$ in denoting both the functor $\CL : \llist \to \set$, which maps a set $X$ to the set of lists $\CL X$, and the monoidal envelope $\CL : \operad \to \cat$. One reason for this slight abuse is the fact that, for an operad $\CP$, the category $\LP$ is formed by taking lists of colors and operations in $\CP$. Further justification will be provided when we construct nerves of operads in the next section. 

    An interesting fact to point out is that the functor $\CL : \llist \to \set$ is right adjoint to the inclusion $\set \subseteq \llist$. On the other hand, the functor $\CL : \operad \to \cat$ is left adjoint when we modify the target to be monoidal categories. But, we can also exhibit $\CL$ as a right adjoint by modifying the notion of morphism between operads.

    More precisely, we may form a category $\widetilde{\operad}$ in which the objects are operads and morphisms between two operads $\CP$ and $\Q$ are monoidal functors $\LP \to \LQ$, which we regard as \emph{listings} between operads. In other words, $\widetilde{\operad}$ is the category of free monoidal categories (in the sense that they are generated by operads). Equivalently, $\widetilde{\operad}$ is the Kleisli category associated the free monoidal category monad on operads. Notice that if $\CP$ and $\Q$ are sets (meaning that they consist only of colors) a listing of operads is the same as a listing of sets. 

    This way, it is easy to see that the inclusion $\cat \subseteq \widetilde{\operad}$, which regards every category as an operad with unary operations only, has as right adjoint the functor $\CL : \widetilde{\operad} \to \cat$ which maps $\CP $ to $ \LP$ and each morphism $\LP \to \LQ$ to itself. Indeed, for a category $\A$ and an operad $\CP$ we have bijections

    $$\cat(\A, \LP) \cong \moncat(\CL \A, \LP) \cong \widetilde{\operad}(\A, \CP)$$
    
\end{remark}

\section{First constructions} \label{section : first constructions}

\subsection{Simplicial lists} \label{subsec : slist}

The category $\llist$ has sets as objects and its morphisms $A \xslashedrightarrow{} X$ are functions $A \to \CL X$, while composition happens by concatenation (Definition \ref{def_list}). We are interested in the following categories:
\begin{itemize}
    \item [-] $\widetilde{\slist} = \llist^{\D{op}}$, the category of simplicial objects in $\llist$. 
    \item[-] $\slist$, the wide subcategory of $\widetilde{\slist}$ with morphisms required to have functions as components. 
    \item[-] $\widetilde{\slisto}$ and $\slisto$, the full subcategories of $\widetilde{\slist}$ and $\slist$ spanned by operadic simplicial lists (Definition \ref{def_operadic}).  The latter are simplicial objects $X : \D{op} \to \llist$ such that all face and degeneracy operations are functions, except possibly for the last face maps $d_n : X_n \xslashedrightarrow{} X_{n-1}$, $n \geq 1$, which are allowed to be listings.
\end{itemize}

For multiple reasons, we concentrate on the categories $\slist$ and $\slisto$. Operadic simplicial lists are important because they “look" like operads. For instance, if $X \in \slisto$, a $1$-simplex $f \in X_1$ has a face $d_0f$ which is a single element of $X_0$, while the face $d_1f$ may be a sequence in $\CL X_0$, thus allowing us to imagine $f$ in the form of a corolla
 
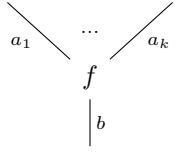
\[\begin{tikzcd}
	{} & {} & {} \\
	& f \\
	& {}
	\arrow["{a_k}", no head, from=1-3, to=2-2]
	\arrow["b", no head, from=2-2, to=3-2]
	\arrow["\dots"{description}, draw=none, from=1-2, to=2-2]
	\arrow["{a_1}"', no head, from=1-1, to=2-2]
\end{tikzcd}\]
We will prove in Section \ref{subsec : operadic} that higher simplices in operadic simplicial lists look like leveled trees. 

Then, a key reason to want morphisms between simplicial lists to be functional is the fact that morphisms of operads preserve the number of inputs in operations. For instance, let $\phi : X \to Y$ be a morphism in $\slisto$ and let $f \in X_1$ be a $1$-simplex as above. The simplicial identities dictate that $d_1 \phi(f) = \phi (d_1 f)$. In other words, we must have that $d_1 \phi(f) = (\phi a_1, \dots , \phi a_k)$, so that $\phi(f)$ is depicted as a corolla
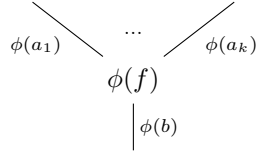
\[\begin{tikzcd}
	{} & {} & {} \\
	& \phi(f) \\
	& {}
	\arrow["{\phi (a_k)}", no head, from=1-3, to=2-2]
	\arrow["\phi (b)", no head, from=2-2, to=3-2]
	\arrow["\dots"{description}, draw=none, from=1-2, to=2-2]
	\arrow["{\phi (a_1)}"', no head, from=1-1, to=2-2]
\end{tikzcd}\]
with $k$ input leaves. 

Another important reason to require morphisms to be functional has to do with \emph{colimits}. The category $\llist$ is not cocomplete and hence, neither is $\widetilde{\slist}$. But, since the existence of colimits is not determined by the objects of a category, but by its morphisms, the problem vanishes when working in $\slist$. We prove that $\slist$ is cocomplete in Section \ref{sec : structure}, but it is not difficult to believe this result given that morphisms are component-wise functions of sets.

    \begin{note}

There is also an explanatory account from the double categorical perspective. 
Listings and functions are two notions of morphisms between sets. We may form a Kleisli double category with listings as horizontal morphisms, functions as vertical morphisms and 2-cells being commutative squares in $\llist$ of the form
\[\begin{tikzcd}
	A & X \\
	B & Y
	\arrow["f"', from=1-1, to=2-1]
	\arrow["g", from=1-2, to=2-2]
	\arrow["u", "\shortmid"{marking}, from=1-1, to=1-2]
	\arrow["v"', "\shortmid"{marking}, from=2-1, to=2-2]
\end{tikzcd}\]
This way, we may interpret $\slist$ as the category of functors from $\D{op}$ into the horizontal part of this double category. Then, our notion of morphism coincides with the notion of \emph{vertical transformation} (see Definition 3 in \cite{haderi2023simplicial}), which appears in the study of double colimits.

\end{note}

The functor $\CL : \llist \to \set$, which is right adjoint of the inclusion $\set \subseteq \llist$, gives us a functor 
\[
\begin{tikzcd}
    \CL :&[-3em] \slist \arrow[r] &\sset
\end{tikzcd}
\]
by post-composition. For all $X \in \slist$, the simplicial set $\CL X$ has as $n$-simplices lists of $n$-simplices in $X$. In particular, $\CL X$ is a monoid object in $\sset$. So, this is just like the monoidal envelope functor for operads . In fact, we will see in \ref{subsec : nerve} that it extends the monoidal envelope functor. 

$\CL$ is no longer right adjoint though, but it is so if we extend the domain from $\slist$ to $\widetilde{\slist}$. This is where the latter category becomes useful. For instance, given a simplicial list $X$, we have the following bijections
$$\CL X_n \cong \sset(\D{n}, \CL X) \cong \widetilde{\slist}(\D{n}, X)$$
This tells us that a simplex $x \in X_n$ is represented by a listing $(x) : \D{n} \xslashedrightarrow{} X$ when we regard $x$ as a singleton sequence in $\CL X_n$. This is an important fact, since we don't know much else about what simplices in a general simplicial list look like. We illustrate the latter point with an example.

\begin{example}[A wild simplicial list]

We may define a (truncated) simplicial list $X : \D{op}_{\leq 1} \to \llist$ with $X_0 = \{ a \}$ and $X_1 = \{ s^\prime, s^{\prime \prime} \}$. Let the simplicial operations be as follows:
$$d_0s^\prime = \emptyset \ , \  d_1s^\prime = a \ , \ d_0 s^{\prime \prime} = a \ , \  d_1s^{\prime \prime} = \emptyset \ , \  s_0 a = (s^\prime, s^{\prime \prime})$$
The action of the degeneracy map $s_0$ on $a$ may be depicted operadically as follows:
\[\begin{tikzcd}
	{} \\
	{}
	\arrow["a"', no head, from=1-1, to=2-1]
\end{tikzcd} \ \ 
\mapsto \ \ 
\begin{tikzcd}[column sep = small]
	{} \\
	{s^\prime} & {s^{\prime \prime}} \\
	& {}
	\arrow["a"', no head, from=1-1, to=2-1]
	\arrow["a", no head, from=2-2, to=3-2]
\end{tikzcd}
\]
The simplicial identities are visibly satisfied. The “wild" part has to do with the fact that both $s^\prime$ and $s^{\prime \prime}$ are non-degenerate, while the list $(s^\prime, s^{\prime \prime})$ is. 

We can make the example more wild by letting $X_1 = \{ s^\prime, s^{\prime \prime} , x \}$ with $s^\prime$ and $s^{\prime \prime}$ as above and say that $d_0 x = d_1 x = \emptyset$ while $s_0 a = (s^\prime, s^{\prime \prime}, x)$. This doesn't violate the simplicial identities. The situation can get more complex if we allow higher simplices. We believe this flexibility of general simplicial lists can help model interesting phenomena. 
    
\end{example}

Besides the monoidal envelope, we also have an underlying simplicial set functor 
    \[
    \begin{tikzcd}
        \left[-\right]_1 :&[-3em] \slist \arrow[r] &\sset
    \end{tikzcd}
    \]
which assigns to a simplicial list $X$ the simplicial set whose $n$-simplices are given by $[X]_1([n]) = \slist(\D{n}, X)$. $[-]_1$ is the right adjoint of the inclusion functor $\sset \subseteq \slist$.

\subsection{Multigraphs and free operads} \label{sec_multi}

In ordinary category theory, an important tool in describing categories of interest is the notion of free category generated by a graph (see Appendix \ref{appendix_free}). We develop an analog in the context of operads, that is, free operads generated from \emph{multigraphs}. 

\begin{definition}[Multigraph] Let $\D{op}_{\leq 1}$ be the full subcategory of $\D{}$ spanned by the objects $[0]$ and $[1]$.
 A (reflexive, directed) multigraph is a functor 
 \[
 \begin{tikzcd}
     \M :&[-3em] \D{op}_{\leq 1} \arrow[r] & \spa
 \end{tikzcd}
 \]
 which is operadic when regarded as a simplicial list. 

 A morphism of multigraphs is a natural transformation in which the components are functions. We denote by $\mult$ the resulting category.
\end{definition}

A multigraph $M$ consists of two sets $\M_0$ and $\M_1$ related by operations
\[\begin{tikzcd}
	{\M_1} & {\M_0}
	\arrow["{d_0}", shift left=3, from=1-1, to=1-2]
	\arrow["{d_1}"', "\shortmid"{marking}, shift right=3, from=1-1, to=1-2]
	\arrow["{s_0}"{description}, from=1-2, to=1-1]
\end{tikzcd}\]
which satisfy the simplicial identities. The condition of being operadic specifies that $s_0$ and $d_0$ are functions, while $d_1$ is allowed to be a listing. Therefore, an edge of $\M$ has a unique target in $\M_0$, but possibly a source consisting whole sequence in $\M_0$. 
Our notion of multigraph is a reflexive variant of Leinster's $T$-graphs (\cite[Definition 4.2.4]{leinster2004higher}), when taking $T$ to be the free monoid monad on sets (see also the remark at the end of the section). 
Although more general, multigraphs can be used as models for trees. 

\begin{example}[Trees as multigraphs] \label{example_tree}
Consider the following tree:
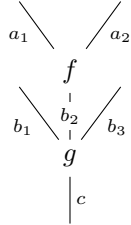
\[\begin{tikzcd}[column sep=small]
	{} && {} \\
	{} & f & {} \\
	& g \\
	& {}
	\arrow["{a_1}"', no head, from=1-1, to=2-2]
	\arrow["{a_2}", no head, from=1-3, to=2-2]
	\arrow["{b_1}"', no head, from=2-1, to=3-2]
	\arrow["{b_2}"{description}, no head, from=2-2, to=3-2]
	\arrow["{b_3}", no head, from=2-3, to=3-2]
	\arrow["c", no head, from=3-2, to=4-2]
\end{tikzcd}\]
We may define a multigraph $\M$ with $\M_0 = \{ a_1 , a_2, b_1 , b_2 , b_3 , c \}$ and $\M_1 = \{ f, g \} \amalg \{ 1_x \ : \ x \in \M_0  \}$, with $d_0$, $d_1$ and $s_0$ defined in the evident manner. 
We see that $\M$ encodes all the information about the tree we depicted.

\end{example}

There is a forgetful functor
\[
\begin{tikzcd}
    U :&[-3em] \operad \arrow[r] & \mult
\end{tikzcd}
\]
For an operad $\cc{P}$, the multigraph $U\cc{P}$ is given by
$$U\cc{P}_0 = \Col (\CP) \ \ \ , \ \ \ U\cc{P}_1 = \Op (\CP)$$
where the simplicial operations $d_0$, $d_1$ and $s_0$ assign target, source and identity. 

We discuss and construct the left adjoint free operad functor 
\[
\begin{tikzcd}
    F :&[-3em] \mult \arrow[r]  &\operad
\end{tikzcd}
\]
In some sense, given a multigraph $\M$, it is clear how to construct $F\M$ : compose everything you can, let the degeneracies serve as identities and make sure composition is associative. If we accept that this construction yields a left adjoint functor $F$, most of the results which rely on $F$ may be proved without reference to an explicit construction.
However, for the sake of both rigor and explanatory power, we introduce some combinatorial formalism. 

\begin{definition}[Operation matrix] Let $\M$ be a multigraph. An operation vector in $\M$ 
\[
\begin{bmatrix}
    f_1 & \dots & f_n 
\end{bmatrix}
\]
is a composable sequence of 1-simplices in $\CL \M_1$ such that $f_n$ is a singleton sequence. We denote the set of such operation vectors $\vect(\M)$. 

An operation matrix is a matrix 
\[
\begin{bmatrix}
    f_1^{(1)} & \dots & f_n^{(1)} \\
    \vdots & \ddots & \vdots \\
    f_1^{(k)} & \dots & f_n^{(k)}
\end{bmatrix}
\]
where each row is an operation vector. We denote the set of such operation matrices by $\mat(\M)$. 
    
\end{definition}

\begin{remark}[Empty operations]

In an operation vector, we allow for a symbol $1_\emptyset$, which we interpret as a dummy operation with no inputs and outputs. This originates from the fact that there is an empty list $\emptyset \in \CL \M_0$. $1_\emptyset$ represents the degeneracy of the empty list in $\CL \M_1$.

For instance, in case $f : \emptyset \to a$ is a 1-simplex in $\M$ with output $a$ and empty set of inputs, the vector 
\[
\begin{bmatrix}
    1_\emptyset & 1_\emptyset & f
\end{bmatrix}
\]
is a valid element in $\vect(\M)$.
    
\end{remark}

\begin{notation}[Degenerate operations]

Whenever we write the entry “$1$" in an operation matrix we mean that the $1$-simplex in $\LM$ this entry represents is degenerate. For instance, if $f : (a_1, a_2) \to b$ is an operation in $\M_1$, then in the operation vector 
\[
\begin{bmatrix}
    1 & f
\end{bmatrix}
\]
we read the entry $1$ to mean the $1$-simplex $\binom{s_0a_1}{s_0a_2}$ in $\LM_1$.
    
\end{notation}

\begin{note} \label{note : vectors trees}
    Informally, we can say that an operation vector in $\M$ is a leveled tree of operations in $\M$. For example, a leveled tree in $\M$ depicted as 
    
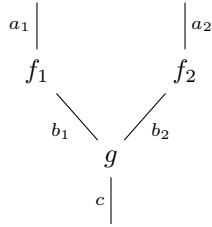
\[\begin{tikzcd}[column sep=small]
	{} && {} \\
	{f_1} && {f_2} \\
	& g \\
	& {}
	\arrow["{a_1}"', no head, from=1-1, to=2-1]
	\arrow["{b_1}"', no head, from=2-1, to=3-2]
	\arrow["{b_2}", no head, from=2-3, to=3-2]
	\arrow["c"', no head, from=3-2, to=4-2]
	\arrow["{a_2}", no head, from=1-3, to=2-3]
\end{tikzcd}\]
    is recorded via the operation vector 
    \[
    \begin{bmatrix}
        \binom{f_1}{f_2} & g
    \end{bmatrix}
    \]
    In this vein, operation matrices record sequences of leveled trees (or, leveled forest). A formal formulation of this idea is given in the next Section \ref{subsec : nerve}, Remark \ref{remark : vectors trees}.
    
\end{note}

Given composable operation matrices $M$ and $N$ (meaning that the output sequence of $M$ coincides with the input sequence of $N$), we may form their \emph{concatenation}, which we denote $M \odot N$, in the evident manner. For instance,
\[
\begin{bmatrix}
    f_1 \\
    f_2
\end{bmatrix} \odot
\begin{bmatrix}
    g
\end{bmatrix} = 
\begin{bmatrix}
    \binom{f_1}{f_2} & g
\end{bmatrix}
\]
We may write concatenation as a function 
\[
\begin{tikzcd}
    \odot :&[-2em] \mat(\M) \times_{\LM_0} \mat(\M) \arrow[r] & \mat(\M)
\end{tikzcd}
\]

In spirit, the free operad $F\M$ has as colors the set $\M_0$ and as operations the set $\vect(\M)$, with composition provided by concatenation. However, we must impose identifications in order to make sure unitality laws are satisfied.
Hence, we are interested in the following equivalence relations relations on $\vect(\M)$ and $\mat(\M)$:
\begin{itemize}
    \item [(U)] (Unitality). Two operation vectors are equivalent if one can be obtained by the other by adding or inserting $1$'s in the vector formation, i.e. this is the equivalence relation generated by 
    \[
    \begin{bmatrix}
    f_1 & \dots & f_n 
\end{bmatrix} \equiv
    \begin{bmatrix}
    f_1 & \dots & 1 & \dots & f_n 
\end{bmatrix}
    \]

    \item[(RU)] (Row-wise unitality). Two operation matrices are equivalent if they are equivalent row-wise.  

    \item[(OU)] (Operadic unitality). We say that two operation matrices $M$ and $N$ are equivalent in case we can write $M = M_1 \odot K \odot M_2$ and $N = M_1 \odot K^\prime \odot M_2$ and we have that $K$ is equivalent to $K^\prime$ under the relation (RU).
\end{itemize}
We denote by $\Vect(\M)$ and $\Mat(\M)$ the quotients of $\vect(\M)$ and $\mat(\M)$ with respect to the operadic unitality relation (OU) respectively. Notice that the relation (OU) is stronger that (RU). 

\begin{lemma}
    The concatenation operation on representatives induces a well-defined operation
    \[
    \begin{tikzcd}
        \odot :&[-2em] \Mat(\M) \times_{\LM_0} \Mat(\M) \arrow[r] & \Mat(\M)
    \end{tikzcd}
    \]
\end{lemma}

\begin{proof}
    This is clear by construction.
\end{proof}

For a multigraph $\M$, each matrix in $\mat(\M)$ splits into its sequence of row vectors. We may write this as a function
\[
\begin{tikzcd}
    \mathsf{split} :&[-3em] \mat(\M) \arrow[r] & \CL \vect(\M)
\end{tikzcd}
\]
However, given a sequence of operation vectors, we cannot obtain an operation matrix from them since the chosen list of rows may contain elements of various lengths. This obstruction vanishes when we apply the relations (U) and (RU), and hence also (OU).

\begin{lemma}[The split-pack bijection] \label{lemma_splitpack}

The function
\[
\begin{tikzcd}
    \mathsf{Split} :&[-3em] \Mat(\M) \arrow[r] & \CL \Vect(\M)
\end{tikzcd}
\]
which splits a matrix into its list of rows is well-defined and a bijection. 
    
\end{lemma}

\begin{proof}
    The function $\mathsf{Split}$ is well defined since the relation (RU) is defined row by row. We may define the inverse 
    \[
    \begin{tikzcd}
        \mathsf{Pack} :&[-3em] \CL \Vect(\M) \arrow[r] & \Mat(\M)
    \end{tikzcd}
    \]
    to be given by choosing representatives of the same length in the sequence in the domain (this is always possible, by adding or removing $1$'s), packing them into a matrix and mapping the matrix into its class in $\Mat(\M)$. $\mathsf{Split}$ and $\mathsf{Pack}$ are clearly inverses of each other.
\end{proof}

\begin{construction}[Free operad]

Let $\M$ be a multigraph. We define the free operad $F\M$ as follows:
\begin{itemize}
    \item [-] On colors, let $\Col(F\M) = \M_0$.
    \item[-] On operations, let $\Op(F\M) = \Vect(\M)$.
    \item[-]  For a color $a$, let the identity operation be defined by the class of the operation vector consisting of the degenerate $1$-simplex at $a$, i.e. $1_a = [s_0 a]$.
    \item[-] Composition provided by concatenation.  
\end{itemize}
Notice that in defining composition, we are implicitly using the split-pack bijection of the previous lemma, in the following sense. Given an operation vector $G \in \Vect(\M)$ and a sequence of vectors $G_1, \dots , G_k \in \Vect(\M)$ which are composable in the operadic sense, we define the composite as
$$G \circ (G_1 , \dots , G_k) = \mathsf{Pack}(G_1, \dots, G_k) \odot G $$
It is easy to see that this composition operation is unital and associative.
\end{construction}

\begin{example} \label{example_freeoperad}
    Let $\M$ be the multigraph which encodes the following tree
    
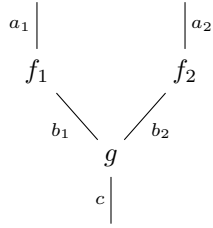
\[\begin{tikzcd}[column sep=small]
	{} && {} \\
	{f_1} && {f_2} \\
	& g \\
	& {}
	\arrow["{a_1}"', no head, from=1-1, to=2-1]
	\arrow["{b_1}"', no head, from=2-1, to=3-2]
	\arrow["{b_2}", no head, from=2-3, to=3-2]
	\arrow["c"', no head, from=3-2, to=4-2]
	\arrow["{a_2}", no head, from=1-3, to=2-3]
\end{tikzcd}\]
The free operadic composites $g \circ (f_1 , f_2)$, $g \circ f_1$ and $g \circ f_2$ are encoded in $F\M$ by the following operation vectors 
\[
\begin{bmatrix}
    \binom{f_1}{f_2} & g
\end{bmatrix} , 
\begin{bmatrix}
    \binom{f_1}{1_{b_2}} & g
\end{bmatrix} , 
\begin{bmatrix}
    \binom{1_{b_1}}{f_2} & g
\end{bmatrix}
\]

Notice the importance of imposing (OU) on operation vectors. Consider the vector
\[
\begin{bmatrix}
    \binom{1}{f_2} & \binom{f_1}{1} & g
\end{bmatrix}
\]
There is no entry $1$ in this vector. Nonetheless, we do have
\[
\begin{bmatrix}
    \binom{1}{f_2} & \binom{f_1}{1} & g
\end{bmatrix} = 
\begin{bmatrix}
    1 & f_1 \\
    f_2 & 1
\end{bmatrix} \odot 
\begin{bmatrix}
    g
\end{bmatrix}
\]
and since 
\[
\begin{bmatrix}
    1 & f_1 \\
    f_2 & 1
\end{bmatrix} \equiv
\begin{bmatrix}
    f_1 \\
    f_2
\end{bmatrix}
\]
we obtain 
\[
\begin{bmatrix}
    \binom{1}{f_2} & \binom{f_1}{1} & g
\end{bmatrix} \equiv 
\begin{bmatrix}
    \binom{f_1}{f_2} & g
\end{bmatrix}
\]
as operation vectors. Hence, as expected, they represent the same operation in $F\M$.

\end{example}

\begin{proposition}
    The free operad construction $\M \mapsto F\M$ is functorial. Moreover, it defines a left adjoint to the forgetful functor $U : \operad \to \mult$. 
\end{proposition}

\begin{proof}
    Functoriality of $F$ is clear. Also, we do have an evident inclusion $\eta: \M \to UF(\M)$ which is identity on $0$-simplices and maps every $1$-simplex $f \in \M_1$ into the operation vector $[f] \in \Vect(UF(\M))$. 

    We also have a natural morphism of operads $\epsilon_\CP : FU(\CP) \to \CP$ which is identity on colors and defined on operations by composition, i.e. 
    \[
    \epsilon 
    \begin{bmatrix}
        f_1 & \dots & f_n
    \end{bmatrix} = 
    f_n \circ \dots \circ f_1
    \]
    It is easy to verify that the triangle identities are satisfied.
\end{proof}

\begin{remark}[Non-reflexive variant]

The notion of multigraph we employed is a special case of the notion of $T$-graph defined by Leinster (\cite[Definition 4.2.4]{leinster2004higher}) when taking $T$ to be the free monoid monad on the category of sets, except that it is reflexive (meaning we include degeneracies). We find the reflexive version a more natural recipient for the forgetful functor from operads. 

The two notions are functionally the same: we can turn any non-reflexive multigraph into a reflexive one by adding degeneracies and then apply the free operad construction as above. A saddle issue to point out is the fact that defining the free operad construction without degeneracies in the multigraph would require us to formulate a notion of tree of operations, as operation vectors would not suffice. Following the general pattern described in the introduction, we want to work with leveled trees, and hence the utility of reflexive multigraphs. 
    
\end{remark}

\subsection{Interlude : free resolutions of operads} \label{subsec : free resolution}

Let $\CP$ be an operad and $\Q$ be a simplicial\footnote{$\Q$ may also be  assumed to be enriched in topological spaces, but our default setting is enrichment over simplicial sets.} operad. There are various ways in which one can codify the notion of a $\CP$-indexed homotopy coherent algebra in $\Q$, like using the classical $W$-construction of Boardman and Vogt (\cite{boardman2006homotopy}, \cite[Section 1.7]{heuts2022simplicial}), a simplicial version of which is constructed  in \cite[Section 2.7.7]{heuts2022simplicial} in terms of trees. The idea is the same: replace $\CP$ with a thick variant $W\CP$ and study morphisms $W\CP \to \Q$.

Our formalism offers a direct path towards thickening.  The free-forgetful adjunction $F : \mult \leftrightarrows \operad : U$ induces a free resolution functor
\[
\begin{tikzcd}
    FU_* :&[-3em] \operad \arrow[r] & \operad_\D{}
\end{tikzcd}
\]
and thus, for each operad $\CP$, we obtain a thickened version $FU_* \CP$.
This way, in analogy with traditional simplicial theory, we could define a \emph{homotopy coherent algebra} of an operad $\CP$ valued in a simplicial operad $\Q$ to be a morphism of simplicial operads $FU_* \cc{P} \rightarrow \cc{Q}$.

Following intuition from the free resolution for categories (Appendix \ref{appendix_free}) and our constructions in the previous section, we could sketch a description $FU_* \CP$ as follows:
\begin{itemize}
    \item [-] The colors of $FU_* \CP$ are those of $\CP$.
    
    \item[-] The $0$-operations are represented by operation vectors in $\CP$
    \[
    \begin{bmatrix}
        \cdot & \dots & \cdot
    \end{bmatrix}
    \]
    
    \item[-] The $1$-operations are represented by operation vectors in $FU \CP$. These will be vectors whose entries are operation matrices in $\CP$, for instance 
    \[
    \begin{bmatrix}
        \begin{bmatrix}
            \cdot & \cdot \\
            \cdot & \cdot
        \end{bmatrix} 
        \begin{bmatrix}
            \cdot \cdot
        \end{bmatrix}
    \end{bmatrix}
    \]

    \item [-] The face map $d_0$ concatenates the matrices inside the operation vector. This can be thought of as removing the inner brackets of the matrices. For instance,
    \[
   d_0 \begin{bmatrix}
        \begin{bmatrix}
            \cdot & \cdot \\
            \cdot & \cdot
        \end{bmatrix} 
        \begin{bmatrix}
            \cdot \cdot
        \end{bmatrix}
    \end{bmatrix} = 
    \begin{bmatrix}
        \binom{\cdot}{\cdot} & \binom{\cdot}{\cdot} & \cdot & \cdot
    \end{bmatrix}
        \]
        The face map $d_1$ composes the entries inside each (row of each) matrix and then removes brackets
        \[
        d_1 \begin{bmatrix}
        \begin{bmatrix}
            \cdot & \cdot \\
            \cdot & \cdot
        \end{bmatrix} 
        \begin{bmatrix}
            \cdot \cdot
        \end{bmatrix}
    \end{bmatrix} = 
    \begin{bmatrix}
        \binom{\cdot \circ \cdot}{\cdot \circ \cdot} & \cdot & \cdot
    \end{bmatrix}
        \]
        The degeneracy $s_0$ acts by applying inner brackets
        \[
        s_0 \begin{bmatrix}
            \cdot & \dots & \cdot 
        \end{bmatrix} = 
        \begin{bmatrix}
            \begin{bmatrix}
            \cdot & \dots & \cdot 
        \end{bmatrix}
        \end{bmatrix}
        \]

        \item[-] The $2$-operations will be represented by operation vectors whose entries are matrices whose entries are matrices. And so on in higher dimension. Faces are obtained by removing brackets and concatenating, except for the face $d_0$ which composes entries. Degeneracies are formed by adding a layer of brackets. 

\end{itemize}
While the above sketch provides some intuition, there are relations in each layer of the resolution. This makes an explicit form of the resolution hard to track! We illustrate with an example.

\begin{example}
    Let $\cc{P}$ be the free operad discussed in Example \ref{example_freeoperad}, generated by the multigraph (which encodes a tree) depicted as 
     
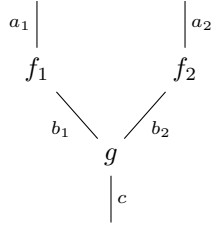
\[\begin{tikzcd}[column sep=small]
	{} && {} \\
	{f_1} && {f_2} \\
	& g \\
	& {}
	\arrow["{a_1}"', no head, from=1-1, to=2-1]
	\arrow["{a_2}", no head, from=1-3, to=2-3]
	\arrow["{b_1}"', no head, from=2-1, to=3-2]
	\arrow["{b_2}", no head, from=2-3, to=3-2]
	\arrow["c", no head, from=3-2, to=4-2]
\end{tikzcd}\]
Let us compute the mapping space $FU_*\CP (a_1, a_2 ; c)$.

We have four classes of operation vectors with inputs $a_1, a_2$ and output $c$:
\[
\begin{bmatrix}
    g \circ \binom{f_1}{f_2} 
\end{bmatrix}
,
\begin{bmatrix}
    \binom{1}{f_2} & g \circ \binom{f_1}{1} 
\end{bmatrix}
,
\begin{bmatrix}
    \binom{f_1}{1} & g \circ \binom{1}{f_2}
\end{bmatrix}
,
\begin{bmatrix}
    \binom{f_1}{f_2} & g
\end{bmatrix}
\]
We also have the following configuration of $1$-simplices between the above $0$-simplices:

\begin{center}
    \begin{tikzpicture}
    \node (A) at (0,3) { $\begin{bmatrix}     g \circ \binom{f_1}{f_2}  \end{bmatrix}$};
    \node (B) at (6,3) {$\begin{bmatrix}     \binom{1}{f_2} & g \circ \binom{f_1}{1}  \end{bmatrix}$};
    \node (C) at (0, 0) {$\begin{bmatrix}     \binom{f_1}{1} & g \circ \binom{1}{f_2} \end{bmatrix}$};
    \node (D) at (6,0) {$\begin{bmatrix}     \binom{f_1}{f_2} & g \end{bmatrix}$};
    \path[->,font=\scriptsize,>=angle 90]
(A) edge node[above]{$\begin{bmatrix} \begin{bmatrix}     \binom{1}{f_2} & g \circ \binom{f_1}{1}  \end{bmatrix} \end{bmatrix}$} (B)
(A) edge node[left]{$\begin{bmatrix} \begin{bmatrix}     \binom{f_1}{1} & g \circ \binom{1}{f_2} \end{bmatrix} \end{bmatrix}$} (C)
(B) edge node[right]{$\begin{bmatrix} \begin{bmatrix} 1 \\ f_2 \end{bmatrix} & \begin{bmatrix} \binom{f_1}{1} & g \end{bmatrix} \end{bmatrix}$} (D)
(C) edge node[below]{$\begin{bmatrix} \begin{bmatrix} f_1 \\ 1 \end{bmatrix} & \begin{bmatrix} \binom{1}{f_2} & g \end{bmatrix} \end{bmatrix}$} (D)
(A) edge node[above right] {$\begin{bmatrix} \begin{bmatrix}  \binom{f_1}{f_2} & g \end{bmatrix} \end{bmatrix}$} (D)
;
\end{tikzpicture}
\end{center}

Notice that the edges on the top, the right and the diagonal symbolically behave as in the free resolution for categories, in the sense that the faces are apparent. For the edge on the bottom, the face $d_1$ is clear, but the face $d_0$ is correctly depicted by virtue of the relation (OU) since we have
\[
d_0 \begin{bmatrix} \begin{bmatrix} f_1 \\ 1 \end{bmatrix} & \begin{bmatrix} \binom{1}{f_2} & g \end{bmatrix} \end{bmatrix} =  \begin{bmatrix} \binom{f_1}{1}  &  \binom{1}{f_2} & g \end{bmatrix} 
\equiv
\begin{bmatrix}
    \binom{f_1}{f_2} & g
\end{bmatrix}
\]
Similarly for the edge on the right. 

Finally, we do have $2$-simplices which fill the top and bottom triangle, written respectively
\[
\begin{bmatrix}
    \begin{bmatrix} \begin{bmatrix} 1 \\ f_2 \end{bmatrix} & \begin{bmatrix} \binom{f_1}{1} & g \end{bmatrix} \end{bmatrix}
\end{bmatrix} 
,
\begin{bmatrix}
    \begin{bmatrix} \begin{bmatrix} f_1 \\ 1 \end{bmatrix} & \begin{bmatrix} \binom{1}{f_2} & g \end{bmatrix} \end{bmatrix}
\end{bmatrix}
\]
as expected. We conclude that 
$$FU_*\CP (a_1, a_2 ; c) \cong \D{1} \times \D{1}$$

\end{example}

\begin{note}
    The above computation agrees with the simplicial version of the $W$-construction in \cite[Section 2.7.7]{heuts2022simplicial}. In the language of trees, the mapping space between $(a_1, a_2)$ and $c$ in the above example is taken by definition to be the product over inner leaves between $a_1$, $a_2$ and $c$ of copies of $\D{1}$. Since we have two inner leaves $b_1$ and $b_2$, the mapping space is $\D{1} \times \D{1}$. We believe the two constructions are isomorphic, but here we only prove they have the same homotopy type.
\end{note}

While the mapping spaces in the free resolution $FU_* \CP$, as the above example indicates, may be hard to describe in explicit terms, this resolution possesses a number of good properties by virtue of the manner in which it is constructed. 

\begin{proposition} \label{prop : weakeqop}
    The morphism $FU_* \CP \to \CP$ induced by composing in $\CP$ is a weak equivalence of simplicial operads, where $\CP$ is regarded as discrete.
\end{proposition}

\begin{proof}
We have an augmentation with extra-degeneracies $FU_* \CP \dashrightarrow \CP$, obtained exactly as in the classical case of categories (see appendices \ref{appendix_Aug} and \ref{appendix_free}). The augmentation induces a morphism of simplicial operads $FU_* \CP \to \CP$ (regarding $\CP$ as discrete) which is identity on colors. Moreover, since we can carry the original augmentation to operation spaces, meaning that for all sequences of colors $\ua$ and color $b$ of $\CP$ we have an augmentation of simplicial sets $(FU_* \CP)(\ua, b) \dashrightarrow \CP(\ua, b)$ with extra degeneracies, we deduce that the induced map $(FU_* \CP)(\ua, b) \to \CP(\ua, b)$ is a weak equivalence of simplicial sets.
    
\end{proof}

\begin{remark}[Operad valued simplicial computads and cofibrancy]

Using the notion of multigraph and the free operad construction, we can also implement the classical idea of simplicial computad (\cite[Section 16.2]{riehl2014categorical}) in the context of operads. We sketch this idea here, but do not proceed with proofs.

Let $\CP$ be an operad. An operation $h \in \Op(\CP)$ is said to be \emph{atomic} if, in case $h = g \circ (f_1, \dots , f_k)$  for some operations $g, f_1, \dots ,f_k$, either $g$ is an identity or $f_1, \dots , f_k$ are all identities. We say that an operad $\CP$ is \emph{free} if it is freely generated by the multigraph of its atomic operations.

Next, we may define an \emph{operad valued simplicial computad} to be a discrete simplicial object
\[
\begin{tikzcd}
    \cc{A} :&[-3em] \D{op} \arrow[r] & \operad
\end{tikzcd}
\]
such that 
\begin{itemize}
    \item [-] For each $[n] \in \D{}$, $\A_n$ is free.
    \item[-] Degeneracies of atomic operations are atomic. 
\end{itemize}

We believe it is true that operad-valued simplicial computads are the cofibrant objects in $\soperad$. This would immediately imply that, for every operad $\CP$, the simplicial operad $FU_* \CP$ is cofibrant and hence it serves as a cofibrant replacement. 

We leave the above claim without proof for two reasons. First, it should follow from well-known arguments (see \cite[Lemma 16.2.2]{riehl2014categorical}) which rely on niceness properties of the model structure which $\soperad$ possesses. Second, we do not use this characterization of cofibrant objects in our proofs. In fact, we will introduce another model for the cofibrant replacement of an operad in Section \ref{section : coherent nerve} which has nicer features.
    
\end{remark}

We end with a proposition about the relationship between $\CL$ and $FU_*$. Notice that the square 
\[\begin{tikzcd}
	\operad && \cat \\
	\soperad && \scat
	\arrow["{FU_*}"', from=1-1, to=2-1]
	\arrow["\CL", from=1-1, to=1-3]
	\arrow["\CL"', from=2-1, to=2-3]
	\arrow["{FU_*}", from=1-3, to=2-3]
\end{tikzcd}\]
does not commute. In fact, the above square is far from being commutative. For an operad $\CP$, the resolution $FU_* (\LP)$ has more simplices in the mapping spaces compared to the simplicial category $\CL FU_* \CP$. Nonetheless, the aforementioned simplicial categories have the same homotopy type.

\begin{proposition} \label{prop : natwe}
    There is a natural weak equivalence
    \[\begin{tikzcd}
	\operad && \cat \\
	\soperad && \scat
	\arrow["{FU_*}"', from=1-1, to=2-1]
	\arrow["\CL", from=1-1, to=1-3]
	\arrow["\CL"', from=2-1, to=2-3]
	\arrow["{FU_*}", from=1-3, to=2-3]
	\arrow["\sim"{description}, shorten <=5pt, shorten >=9pt, Rightarrow, from=1-3, to=2-1]
\end{tikzcd}\]
\end{proposition}

\begin{proof}
    For a fixed operad $\CP$ we define a simplicial functor $D : FU_*(\LP) \to \CL( FU_* \CP)$ as follows. On the level of $0$-simplices in the mapping spaces, notice that we do have a functor
    \[
    \begin{tikzcd}
        D_0 : &[-3em] FU(\LP) \arrow[r] & \CL FU(\CP)
    \end{tikzcd}
    \]
    Morphisms in $FU(\LP)$ are represented by chains of morphisms in $\LP$, which are just operation matrices in $\CP$. By splitting these matrices into rows, we obtain a list of representatives for morphisms in $\CL FU(\CP)$. It is not difficult to check that we obtain a well-defined functor $D_0$ with such a splitting process. 
    Notice that, by construction, we have a commutative triangle
    \[\begin{tikzcd}
	{FU(\LP)} && {\CL FU(\CP)} \\
	& \LP
	\arrow["{\epsilon_{\LP}}"', from=1-1, to=2-2]
	\arrow["{\CL \epsilon_\CP}", from=1-3, to=2-2]
	\arrow["{D_0}", from=1-1, to=1-3]
\end{tikzcd}\]
    Moreover, $D_0$ is natural in $\CP$, so that we may regard it as a natural transformation $D_0 : FU \CL \Rightarrow \CL FU$.

    On the next iteration of both resolutions, we define the functor $D_1$ as the composite
    \[
    \begin{tikzcd}
    D_1 :&[-2em] FUFU(\LP) \arrow[r, "{FU(D_0)}"] & FU (\CL FU \CP) \arrow[r, "{D_0}"] & \CL (FUFU\CP) 
    \end{tikzcd}
    \]
    In higher dimension, we define $D_n : FU^{(n)} \to \CL (FU^{(n)} \CP)$ recursively by “pulling $\CL$ from the left to the right step by step".
    In conclusion, we obtain a morphism of simplicial categories $D : FU_*(\LP) \to \CL FU_*(\CP)$ which is natural in $\CP$.
    
    By construction, we have a commutative triangle
    \[\begin{tikzcd}
	{FU_*(\LP)} && {\CL FU_*(\CP)} \\
	& \LP
	\arrow["\sim"', from=1-1, to=2-2]
	\arrow["\sim", from=1-3, to=2-2]
	\arrow["D", from=1-1, to=1-3]
\end{tikzcd}\]
    The weak equivalence on the left is the structure map of the resolution of $\LP$ as a category. The map on the right is obtained by applying $\CL$ to the structure map $FU_*\CP \to \CP$ of the resolution for $\CP$. The latter is a weak equivalence by Proposition \ref{prop : weakeqop}. Thus, the morphism displayed in the above diagram is a weak equivalence since  $\CL$ preserves weak equivalences (Lemma \ref{lemma : Lwefib}).  We conclude that $D$ is also a weak equivalence by the two-out-of-three property of weak equivalences. 
\end{proof}

\subsection{Nerves and quasi-operads} \label{subsec : nerve}

Now we fill in the details for the sketch of the nerve construction provided in the introduction. We begin by encoding leveled trees as multigraphs.

\begin{notation}[Simplices of $\nd$] 

We denote elements of $\nd$ by Greek letters $\alpha$, $\beta$ , \dots . Given $\alpha : \D{n} \to \dplus$, we denote by $A_i$ the set $\alpha(i)$ and by $\alpha_i : A_{i-1} \to A_i$ the image under $\alpha$ of the arrow $(i-1) \to i$ in $\D{}$, $1 \leq i \leq n$, so that $\alpha$ is depicted as
\[\begin{tikzcd}
	{\alpha : } &[-2em] {A_0} & {A_1} & \dots & {A_n}
	\arrow["{\alpha_1}", from=1-2, to=1-3]
	\arrow["{\alpha_2}", from=1-3, to=1-4]
	\arrow["{\alpha_n}", from=1-4, to=1-5]
\end{tikzcd}\]

We denote by $\alpha_{i,j} : A_i \to A_j$ the image of the arrow $i \to j$ in $\D{n}$, for all $0 \leq i \leq j \leq n$.

Similarly, for a simplex $\beta$, we use the symbols $B_i$, $\beta_i$ and $\beta_{i,j}$.
    
\end{notation}

\begin{construction}[Key free operads] \label{const_key}

Let $\alpha \in \nd_n$.
We define (the non-reflexive) multigraph $M_\alpha$ with sets of simplices
$$M_{\alpha, 0} = \coprod_{i = 0}^n A_i \ \ \ , \ \ \ M_{\alpha, 1} = \coprod_{i=1}^n \{ p_i^{(a)} \ : \ a \in A_i \}$$

Let $d_0 p_i^{(a)} = a$ and $d_1 p_i^{(a)} = \alpha_i^{-1}(a)$ in case $a \in A_i$, $1 \leq i \leq n$. The sets $A_i$ are ordered, so we may regard the preimage $\alpha_i^{-1}(a)$ as a sequence in $M_{\alpha , 0}$. The simplicial identities are clearly satisfied. 

We define the operad
$$T_\alpha = FM_\alpha$$
In particular, given indices $i \leq j$, and $a \in A_j$ we have an operation 
$$p_{i,j}^{(a)} : \alpha_{i,j}^{-1}(a) \to a$$
in $T_\alpha$ obtained by the evident chain of generating operations in $M_\alpha$. These operations satisfy
$$p_{i,j}^{(a)} \circ \left(p_{j,k}^{(b)} \right)_{b \in \alpha_{i,j}^{-1}(a)} = p_{i,k}^{(a)}$$
for all indices $i \leq j \leq k$. Under this notation, we have $1_a = p_{i,i}^{(a)}$.
\end{construction}

\begin{remark} \label{remark : vectors trees}
    Let $\M$ be a multigraph. Following up on Note \ref{note : vectors trees}, we can say that an operation vector in $\M$ is a morphism of multigraphs $M_\alpha \to \M$ for some rooted simplex $\alpha \in \nd^{root}$, while allowing $\alpha$ to be a general simplex recovers the notion of operation matrix.
\end{remark}

\begin{lemma}
    The construction $\alpha \mapsto T_\alpha$ is functorial in the simplex category, i.e. we have a functor
    \[
    \begin{tikzcd}
        T_* :&[-3em] \int \nd \arrow[r] & \operad
    \end{tikzcd}
    \]
\end{lemma}

\begin{proof}
    Let $\alpha \in \nd_n$, $\theta : [k] \to [n]$ be a morphism in $\D{}$ and say $\beta = \theta^* \alpha$. We specify the induced morphism of operads $\lambda_\theta : T_\beta \to T_\alpha$ by saying that on the generating multigraph $M_\beta$ we have
    $$\lambda_\theta (p_i^{(b)} )= p_{\theta(i-1), \theta(i)}^{(b)}$$
    for all $i = 1, \dots, k-1$ and $b \in B_i$. This construction clearly respects composition.

\end{proof}

\begin{notation}
    For this section, we retain the notation $\lambda_\theta$ in the proof of the above lemma.
\end{notation}

\begin{remark}
    The morphisms $\lambda_\theta$ can be interpreted as follows. In case $\theta = d_0$ or $\theta = d_n$, the induced morphism $\lambda_\theta : T_{d_i \alpha} \to T_\alpha$ is the evident inclusion. In case $\theta = d_i$ is an inner face map, the induced morphism $\lambda_{d_i} : T_{d_i\alpha} \to T_\alpha$ is obtained by composing in the $i$-th level of the operations. In case $\theta = s_i$ is a degeneracy map, the induced morphism $\lambda_{s_i} : T_{s_i \alpha} \to T_\alpha$ is obtained by mapping the $i$-th level of operations in the domain to identities in the codomain.
\end{remark}

For each $[n] \in \D{}$, the set $\nd_n$ has the structure of a monoid given by ordinal sum. To be more precise, for $\alpha, \beta \in \nd_n$ we define 
$$(\alpha \amalg \beta)(n) = A_n \amalg B_n.$$

\begin{lemma} \label{lemma : coprod}
    For all $[n] \in \D{}$ and $\alpha, \beta \in \nd_n$ we have $T_{\alpha \amalg \beta} \cong T_\alpha \amalg T_\beta$.
\end{lemma}

\begin{proof}
    This is clear. In fact, on the level of multigraphs we have $M_{\alpha \amalg \beta} \cong M_\alpha \amalg M_\beta$.
\end{proof}

One reason the operads $T_\alpha$ are key is the following proposition.

\begin{proposition} \label{prop : simplex classification}
    Let $\CP$ be an operad. There is bijection
    $$N(\LP)_n \cong \coprod_{\alpha \in \nd_n} \operad(T_\alpha, \CP)$$
    for all $[n] \in \D{}$ which is natural in $\CP$ and $[n]$.
\end{proposition}

\begin{proof}
    Consider an $n$-simplex in $N(\LP)_n$
    \[\begin{tikzcd}
	{\ua^{(0)}} & {\ua^{(1)}} & \dots & {\ua^{(n)}}
	\arrow[from=1-1, to=1-2]
	\arrow[from=1-2, to=1-3]
	\arrow[from=1-3, to=1-4]
\end{tikzcd}\]
Each object in the chain is a sequence $\ua^{(i)} : A_i \to \Col(\CP)$, $A_i \in \dplus$, and underlying each morphism in the chain there is a morphism $\alpha_{i} : A_{i-1} \to A_i$ in $\dplus$. This determines and is determined by specifying an $n$-simplex $\alpha \in \nd$. The rest of the data determines and is determined by a morphism of multigraphs $M_\alpha \to U\CP$, or equivalently a morphism of operads $T_\alpha \to \CP$. 

Naturality in $\CP$ is evident. By naturality in $[n]$ we mean that for all $\theta : [k] \to [n]$ in $\D{}$ the following square commutes
\[\begin{tikzcd}
	{N(\LP)_n} & {\displaystyle \coprod_{\alpha \in \nd_n} \operad(T_\alpha, \CP)} \\
	{N(\LP)_k} & {\displaystyle \coprod_{\beta \in \nd_k} \operad(T_\beta, \CP)}
	\arrow["{\theta^*}"', from=1-1, to=2-1]
	\arrow["\cong", from=1-1, to=1-2]
	\arrow["\cong"', from=2-1, to=2-2]
	\arrow["{\coprod \lambda_\theta^*}", from=1-2, to=2-2]
\end{tikzcd}\]
\end{proof}

\begin{definition}[Rooted simplex]
    A simplex $\alpha \in \nd_n$ is called rooted if $A_n \cong [0]$. We denote by $\nd_n^{root}$ the set of rooted $n$-simplices in $\nd$.
\end{definition}

The multigraphs $M_\alpha$ are key in understanding the simplicial aspects of operad theory, since we can perform simplicial operations on simplices of $\dplus$. For a rooted simplex $\alpha$, the multigraph $M_\alpha$ encodes precisely the data for a rooted tree. While in general, $M_\alpha$ records a forest consisting of a sequence of rooted trees. This sequencing is precisely what generates listings in our approach.

\begin{construction} [Root restriction] \label{const : rooted restriction}

Let $\alpha \in \nd_n$ and $a \in A_n$. The rooted restriction of $\alpha$ at $a$ is defined to be the simplex $\alpha_a \in \nd_n^{root}$ given by pulling back the chain of morphisms in $\alpha$ along the inclusion $\{a \} \subseteq A_n$ as indicated in the diagram
\[\begin{tikzcd}
	{\alpha_a:} &[-2em] {A_{0,a}} & \dots & {A_{n-1,a}} & {\{a \}} \\
	{\alpha :} &[-2em] {A_0} & \dots & {A_{n-1}} & {A_{n}}
	\arrow["{\alpha_{n}}", from=2-4, to=2-5]
	\arrow[from=1-2, to=1-3]
	\arrow["{\alpha_1}", from=2-2, to=2-3]
	\arrow["{\alpha_{n-1}}", from=2-3, to=2-4]
	\arrow[from=1-3, to=1-4]
	\arrow[hook', from=1-5, to=2-5]
	\arrow[hook', from=1-4, to=2-4]
	\arrow[hook', from=1-2, to=2-2]
	\arrow[from=1-4, to=1-5]
\end{tikzcd}\]
    
\end{construction}

\begin{lemma}[Rooted decomposition] \label{lemma : rooted decomposition}
    For all $[n] \in \D{}$, level-wise coproduct of ordered sets exhibits $\nd_n$ as a free monoid generated by the set $\nd_n^{root}$, i.e.
    $$\nd_n \cong \CL \nd_n^{root}$$
\end{lemma}

\begin{proof}
    For each $\alpha \in \nd_n$ we have
    $$\alpha = \coprod_{a \in A_n} \alpha_a$$
    This representation is unique, given that $A_n$ is ordered.
\end{proof}

\begin{example}[$\nd^{root}$ as a simplicial list]

The nerve $\nd$ is a simplicial set by construction, but restricting to sets of rooted simplices breaks the simplicial structure since some simplicial operations may cut the root of the trees these simplices encode. However, we can assemble the sets of rooted simplices $\nd^{root}_n$, $[n] \in \D{}$, in a simplicial list
\[
\begin{tikzcd}
    \nd^{root} :&[-2em] \D{op} \arrow[r] & \llist
\end{tikzcd}
\]
The action of a morphism $\theta : [k] \to [n]$ in $\D{}$ on a rooted simplex $\alpha \in \nd_n^{root}$ is given by the rooted decomposition of the simplex $\theta^* \alpha \in \nd_k$, i.e.
$$\theta^* (\alpha) = ((\theta^* \alpha)_a)_{a \in A_{\theta((k)}}$$
In fact, this is an operadic simplicial list since, if $\theta(k) = n$, we will have $A_{\theta(k)} = A_n \cong [0]$ and the output sequence on the right will have a single element.
\end{example}

\begin{construction}[List nerves of operads]

For an operad $\CP$ we construct the \emph{list nerve} $N^l \CP \in \slist$ with set of $n$-simplices given by
$$N^l(\CP)_n = \coprod_{\alpha \in \nd_n^{root}} \operad(T_\alpha, \CP)$$
for all $[n] \in \D{}$. 

Let $\theta : [k] \to [n]$ be a morphism in $\D{}$, $\alpha \in \nd_n^{root}$ and $x : T_\alpha \to \CP$ be an $n$-simplex in the list nerve of $\CP$. The simplex $\theta^*\alpha \in \nd_k$ has a rooted decomposition  $\theta^* \alpha = \displaystyle \coprod_{a \in A_{\theta(k)}} (\theta^* \alpha)_a$ (Lemma \ref{lemma : rooted decomposition}), which induces an isomorphism $$T_\alpha \cong \displaystyle \coprod_{a \in A_{\theta(k)}} T_{(\theta^* \alpha)_a} $$ (Lemma \ref{lemma : coprod}). For $a \in A_{\theta(k)}$, let $(\theta, a)^* x \in N(\CP)_k$ be the $k$-simplex represented by the composite 
\[\begin{tikzcd}
	{(\theta, a)^*x} : &[-2em] {T_{(\theta^* \alpha)_a}} \arrow[r, "\subseteq"] & {\displaystyle \coprod_{a \in A_{\theta(k)}} T_{(\theta^* \alpha)_a} \cong T_{\theta^*\alpha}} \arrow[r, "{\lambda_\theta}"] & {T_\alpha} \arrow[r, "x"] & \CP
\end{tikzcd}\]
Finally, we define the list $\theta^* x \in \CL N^l(\CP)_k$ as
$$\theta^* x = ((\theta, a)^* x)_{a \in A_{\theta(k)}}$$
The simplicial identities are evident. Also, the simplicial list $N^l(\CP)$ is clearly operadic. 

The nerve construction extends to a fully-faithful functor
\[
\begin{tikzcd}
    N^l :&[-3em] \operad \arrow[r] & \slist
\end{tikzcd}
\]
This is easy to check, but we note that fully-faithfulness is due to the fact that morphisms between simplicial lists are defined to have functional components. 

\end{construction}

\begin{example}
    We have an isomorphism of simplicial lists
    $$N^l \assoc \cong \nd^{root}$$
    Indeed, for each $\alpha \in \nd^{root}$ there is a unique $\alpha$-shaped simplex $T_\alpha \to \assoc$. Therefore we have $(N^l \assoc)_n \cong \nd_n^{root}$.
\end{example}

\begin{proposition} \label{prop : L commutes N}
    The following squares commute (up to isomorphism)
    \[\begin{tikzcd}
	\cat & \operad \\
	\sset & \slist
	\arrow["\subseteq", from=1-1, to=1-2]
	\arrow["\subseteq"', from=2-1, to=2-2]
	\arrow["N"', from=1-1, to=2-1]
	\arrow["N^l", from=1-2, to=2-2]
\end{tikzcd} \ \  , \ \ 
\begin{tikzcd}
	\operad & \cat \\
	\slist & \sset
	\arrow["N^l"', from=1-1, to=2-1]
	\arrow["N", from=1-2, to=2-2]
	\arrow["\CL", from=1-1, to=1-2]
	\arrow["\CL"', from=2-1, to=2-2]
\end{tikzcd}
\]
\end{proposition}

\begin{proof}
    The square on the left says that the nerve construction for operads extends the ordinary nerve construction for categories. This is clear, since, for a category $\C$ and $\alpha \in \nd^{root}$, there is a morphism $T_\alpha \to \C$ if and only if $A_i \cong [0]$ for all $i$. The latter implies that $T_\alpha \cong \D{n}$ (regarding $\D{n}$ as a category and subsequently as an operad with unary operations only). Hence, the operadic nerve of $\C$ gives back the categorical nerve. 

    The square on the right says that the functor $\CL : \slist \to \sset$ extends the monoidal envelope functor (in particular, this justifies notation). Proposition \ref{prop : simplex classification} and Lemma \ref{lemma : rooted decomposition} provides us respectively the isomorphisms $\nd_n \cong \CL(\nd_n^{root})$ and $N(\LP)_n \cong \amalg_{\alpha \in \nd_n} \operad(T_\alpha, \CP)$. From this we deduce the sequence of equalities
    \begin{align*}
        \CL (N^l\CP)_n &= \CL (N^l(\CP)_n) \\
        &= \CL \left( \coprod_{\alpha \in \nd_n^{root}} \operad(T_\alpha, \CP) \right) \\
        &\cong \coprod_{(\alpha_1, \dots , \alpha_k) \in \CL(\nd_n^{root})} \operad(\amalg_i T_{\alpha_i} , \CP) \\
        &\cong \coprod_{\alpha \in \nd_n} \operad(T_\alpha, \CP) \\
        &\cong N(\LP)_n
    \end{align*}
    where the isomorphism of the last line is natural in $\CP$ and compatible with
simplicial identities by the naturality in $[n]$, hence, the commutation of the
square up to natural isomorphism.
\end{proof}

\begin{theorem}[Nerve Theorem]

A simplicial list $X$ is (isomorphic to the list nerve of) an operad if and only if $X$ is operadic and $\CL X$ is (isomorphic to the nerve of) a category.
    
\end{theorem}

\begin{proof}
    One direction is clear. For an operad $\CP$, the nerve $N^l \CP$ is operadic by construction and we have $\CL N^l(\CP) \cong N(\LP)$ by the previous proposition.

    For the converse, let $X \in \slisto$ such that $\CL X \cong N(\C)$ for some category $\C$. We may define an operad $\CP_X$ as follows:
    \begin{itemize}
        \item [-] The set of colors is given by $X_0$.
        \item[-] The set of operations is given by $X_1$.
        \item[-] Identity operations are given by degeneracies. This is well-defined since $s_0 : X_0 \to X_1$ is a function by $X$ being operadic.
        \item[-] Composition is given by the composition operation in $\C$, coupled with the fact that $X$ is operadic.
    \end{itemize}
    To formalize the last item somewhat, given a $1$-simplex $g : \ub \to c$ and a sequence of $1$-simplices $\underline{f} = (f_i : \ua_i \to b_i)_i \in \CL X_1$ which is composable with $g$, their composite exists in $\C$ (under the identification $\lx \cong \C$). A priory this only gives us back a sequence $g \circ \underline{f} = \underline{h} \in \CL X_1$. This sequence is forced to be a singleton (and hence an element of $X_1$) by the fact that $X$ is operadic. Indeed, we do have that $d_0 \underline{h} = d_0 g$ by the simplicial identities and since $d_0 g = c$, we conclude that $\underline{h} = h \in X_1$. In similar fashion, we may argue that associativity and unitality of composition is inherited by $\C$. 

    It is also easy to check that $X \cong N^l(\CP_X)$. An $n$-simplex $x \in X_n$, when regarded as an element $(x) \in N^l(\CL X)_n$, is uniquely determined by its spine. The latter determines and is determined by a choice of $\alpha \in \nd_n$ and a morphism $M_\alpha \to U\CP_X$. Note that since $X$ is operadic, the last vertex of any simplex $x \in X$ is a singleton sequence in $X_0$. This forces the corresponding $\alpha$ to be rooted.  
\end{proof}

\begin{remark}
    The fact that $\CL X$ is a category can be transposed as a strict inner horn filling condition in $\widetilde{\slist}$, i.e. if $\Lambda_i^n \subseteq \D{n}$ is a horn inclusion, $0 < i < n$, then any listing $\Lambda_i^n \xslashedrightarrow{} X$ extends uniquely to a listing $\D{n} \xslashedrightarrow{} X$.
    \[\begin{tikzcd}
	{\Lambda_i^n} & X \\
	{\D{n}}
	\arrow["\subseteq"', from=1-1, to=2-1]
	\arrow["\shortmid"{marking}, from=1-1, to=1-2]
	\arrow["{\exists ! \  \text{listing}}"', dashed, from=2-1, to=1-2]
\end{tikzcd}\]
    The condition of $X$ being operadic ensures that composition of operations does not behave erratically, but provides well-defined operations (rather than sequences) where it is reasonable to expect so. We obtain our model for $\infty$-operads by requiring the inner horn filling condition to be weak, i.e. by removing the uniqueness condition. 
\end{remark}

\begin{definition}[Quasi-operad]

A  \emph{quasi-operad} is an operadic simplicial list $P$ such that the simplicial set $\CL P$ is an quasi-category.
    
\end{definition}

\begin{example}
    For an operad $\CP$, the simplicial list $N^l(\CP)$ is a quasi-operad.
\end{example}

\begin{note}
    As trivial as it is from the nerve theorem, the above example is important because it tells us that our model for $\infty$-operads captures the theory
of $1$-operads. In Section \ref{subsection : coherent nerve} we provide another class of examples coming from operads enriched in spaces. 

Even without this promise, we can see how regarding quasi-operads as $\infty$-operads is justified by the well-established quasi-categorical model for $\infty$-categories. An quasi-operad $P$ will encode the following:
\begin{itemize}
    \item [-] A set of colors given by $P_0$.
    \item[-] A set of operations given by the set $P_1$, which have a sequence of colors as source and a single color as target.
    \item[-] A composition operation which is well-defined, unital and associative up to coherent homotopy recorded in the $\infty$-category $\CL P$, which is well-behaved by virtue of $\CP$ being operadic as a simplicial list.   
\end{itemize}

A full justification for quasi-operads to be a model for $\infty$-operads would be to construct a model structure on $\slisto$ with fibrant objects being quasi-operads and which is Quillen equivalent to the model structure on simplicial operads. We do not pursue this line in this paper, but intend to do so in a sequel paper. 
\end{note}

\begin{note}[Compositional networks]

As mentioned in the introduction, the nerve theorem for operads leads us to consider a more general type of structure which we named \emph{compositional network} in Definition \ref{def: conet}. These are (not neccesarily operadic) simplicial lists $X$ such that $\lx$ is (the nerve of) a category. 

We could read the nerve theorem for operads to say that operads are operadic compositional networks. Similarly, one can easily imagine that properads may be identified with properadic compositional networks. At this point one can only wonder whether there are algebraic structures of interest which can be modeled by more general compositional networks. The terminology is inspired by current developments in applied category theory, where networks of composable operations (mostly expressed in the language of monoidal categories) play a prominent role (see \cite{fong2018seven} for many examples).
    
\end{note}

\section{The structure of simplicial lists} \label{sec : structure}

\subsection{The perfect-function factorization}

\begin{definition}
    A listing $u : A \xslashedrightarrow{} X$ is said to be \emph{perfect} if
    \begin{itemize}
        \item [-] Each element $x \in X$ appears as a term in $u(a)$ for a unique $a \in A$.
        \item[-] For each $a \in A$, the sequence $u(a)$ has as terms distinct elements of $X$. 
    \end{itemize}
    We say that a listing $u : A \xslashedrightarrow{} X$ between simplicial lists is \emph{perfect} if it is so pointwise, i.e. if the listing $u_n : A_n \xslashedrightarrow{} X_n$ is perfect for all $n \geq 0$.
\end{definition}

\begin{example}
    We depict a few examples of listings $\{ a , b \} \xslashedrightarrow{} \{ x, y, z, w \}$
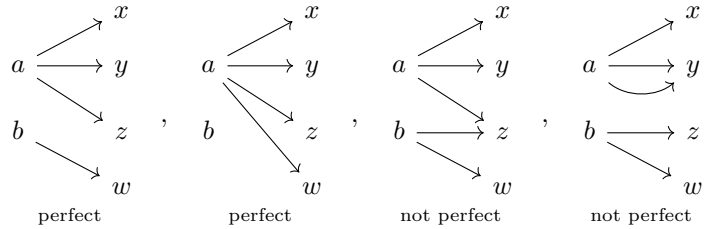
\[\begin{tikzcd}[row sep=small]
	& x \\
	a & y \\
	b & z \\
	& w \\
	{} & {}
	\arrow[from=2-1, to=1-2]
	\arrow[from=2-1, to=2-2]
	\arrow[from=3-1, to=4-2]
	\arrow[from=2-1, to=3-2]
	\arrow["{\text{perfect}}", draw=none, from=5-1, to=5-2]
\end{tikzcd}  \ \ , \ \ 
\begin{tikzcd}[row sep=small]
	& x \\
	a & y \\
	b & z \\
	& w \\
	{} & {}
	\arrow[from=2-1, to=1-2]
	\arrow[from=2-1, to=2-2]
	\arrow[from=2-1, to=3-2]
	\arrow["{\text{perfect}}", draw=none, from=5-1, to=5-2]
	\arrow[from=2-1, to=4-2]
\end{tikzcd} \ \  , \ \  
\begin{tikzcd}[row sep=small]
	& x \\
	a & y \\
	b & z \\
	& w \\
	{} & {}
	\arrow[from=2-1, to=1-2]
	\arrow[from=2-1, to=2-2]
	\arrow[from=2-1, to=3-2]
	\arrow["{\text{not perfect}}", draw=none, from=5-1, to=5-2]
	\arrow[from=3-1, to=3-2]
	\arrow[from=3-1, to=4-2]
\end{tikzcd} \ \ , \ \ 
\begin{tikzcd}[row sep=small]
	& x \\
	a & y \\
	b & z \\
	& w \\
	{} & {}
	\arrow[from=2-1, to=1-2]
	\arrow[from=2-1, to=2-2]
	\arrow["{\text{not perfect}}", draw=none, from=5-1, to=5-2]
	\arrow[from=3-1, to=3-2]
	\arrow[from=3-1, to=4-2]
	\arrow[curve={height=12pt}, from=2-1, to=2-2]
\end{tikzcd}
\]
\end{example}

\begin{example}
    A function is perfect if and only if it is an isomorphism.
\end{example}

\begin{proposition} \label{prop : factorisation}
    Every listing $u : A \xslashedrightarrow{} X$ may be factored as a perfect listing $\title{u} : A \xslashedrightarrow{} U$ followed by a function $f_u : U \to X$
    \[\begin{tikzcd}
	A && X \\
	& U
	\arrow["u", "\shortmid"{marking}, from=1-1, to=1-3]
	\arrow["{\tilde{u}}"', "\shortmid"{marking}, from=1-1, to=2-2]
	\arrow["{f_u}"', from=2-2, to=1-3]
\end{tikzcd}\]
Moreover, the choice of $U$ is unique (up to isomorphism).
\end{proposition}

\begin{proof}
    Let $u : A \xslashedrightarrow{} X$ be a listing. Define the set
    $$U = \left\{ (a, i)  :  a \in A  ,  1 \leq i \leq |u(a)| \right\}$$
    where $| u(a) |$ indicates the length of the sequence $u(a)$ for $a \in A$. There is a listing $\Tilde{u} : A \xslashedrightarrow{} U$, $a \mapsto ((a,1), \dots , (a, |u(a)|))$ which is perfect and a function $f_u : U \to X$, $(a,i) \mapsto u(a)_i$ where $u(a)_i$ is the $i$-th term of the sequence $u(a)$. This provides the perfect-function factorization.

    In case of another perfect-function factorization of $u$, say 
    \[\begin{tikzcd}
	A && X \\
	& {U^\prime}
	\arrow["u", "\shortmid"{marking}, from=1-1, to=1-3]
	\arrow["{\tilde{u}^\prime}"', "\shortmid"{marking}, from=1-1, to=2-2]
	\arrow["{f_u^\prime}"', from=2-2, to=1-3]
\end{tikzcd}\]
notice that we must have $|\tilde{u}^\prime(a)| = |u(a)| =  |\tilde{u}(a)|$ for all $a \in A$. Thus,  there is a well-defined map $U \to U^\prime$, $(a,i) \mapsto \tilde{u}^\prime(a)_i$, which is a bijection since $\tilde{u}^\prime$ is a perfect listing.
\end{proof}

    For the rest of the paper, we refer to the above factorization as \emph{the perfect-function} factorization of a listing.

\begin{note}
    A perfect listing $A \xslashedrightarrow{} X$ can be identified with a function $X \to A$ whose fibres are ordered. The above lemma is trivial from this angle, since every listing can be presented as a span $A \leftarrow U \to X$ where the fibres of $U \to A$ are ordered. 
\end{note}

\begin{proposition}
The perfect-function factorization is functorial, in the sense that it extends into a functor
\[
\begin{tikzcd}
    \spa^{\D{1}} \arrow[r] & \spa^{\D{1}} \times_\spa  \widetilde{\set^{\D{1}}} \\ [-2em]
    A \xslashedrightarrow{u} X \arrow[r, maps to] & (A \xslashedrightarrow{\tilde{u}} U, U \xrightarrow{f_u} X)
\end{tikzcd}
\]
where $\widetilde{\set^{\D{1}}}$ is the subcategory of $\spa^{\D{1}}$ consisting of functions. 
\end{proposition}

\begin{proof}
    We rephrase the statement of the proposition as follows. Given the solid arrow diagram of listings and functions 
\[\begin{tikzcd}[row sep=small]
	A && B \\
	& U && V \\
	X && Y
	\arrow["u"', "\shortmid"{marking}, from=1-1, to=3-1]
	\arrow["v"', "\shortmid"{marking}, from=1-3, to=3-3]
	\arrow["p", "\shortmid"{marking}, from=1-1, to=1-3]
	\arrow["q"', "\shortmid"{marking}, from=3-1, to=3-3]
	\arrow["{\tilde{u}}", "\shortmid"{marking}, from=1-1, to=2-2]
	\arrow["{f_u}", from=2-2, to=3-1]
	\arrow["{\tilde{v}}", "\shortmid"{marking}, from=1-3, to=2-4]
	\arrow["{f_v}", from=2-4, to=3-3]
	\arrow[curve={height=6pt}, dashed, from=2-2, to=2-4]
\end{tikzcd}\]
which depicts listings $u$ and $v$ and a morphism $(p,q) : u \rightarrow v$ in $\llist^{\D{1}}$ (i.e. a commutative square in $\llist$), there is a unique listing ( the dotted arrow) $r : U \xslashedrightarrow{} V$ which makes the diagram commute. 

Using notation from the proof of the previous proposition, for $(a, i) \in U$, we define
$$r(a,i) = \{ (b,j)  :  b \in p(a) \ \text{and} \ v(b)_j \in q(u(a)_i) \}$$
The above set, in lexicographic order, represents a sequence in $V$. It is not difficult to check that $r$ commutes with the solid arrows in the above diagram. 
\end{proof}

\begin{corollary}
    Every listing $u : A \xslashedrightarrow{} X$ between simplicial lists factors functorially as a perfect listing $\tilde{u} : A \xslashedrightarrow{} U$ followed by a functional morphism $f_u : U \to X$. In particular, the simplicial list $U$ is unique up to unique isomorphism.
\end{corollary}

\begin{proof}
    This follows by applying the perfect-function factorization pointwise and then the above proposition. More precisely, for each $[n] \in \D{}$ the listing $u_n : A_n \xslashedrightarrow[]{} X_n$ has a perfect-function factorization $A_n \xslashedrightarrow[]{\tilde{u}_n} U_n \xrightarrow{(f_u)_n} X$ and for each $\theta : [k] \to [n]$ in $\D{}$ the commutative square in $\llist$ 
    \[\begin{tikzcd}
	{A_n} && {X_n} \\
	{A_k} && {X_k}
	\arrow["{u_n}", "\shortmid"{marking}, from=1-1, to=1-3]
	\arrow["{\theta^*}"', "\shortmid"{marking}, from=1-1, to=2-1]
	\arrow["{\theta^*}", "\shortmid"{marking}, from=1-3, to=2-3]
	\arrow["{u_k}"', "\shortmid"{marking}, from=2-1, to=2-3]
\end{tikzcd}\]
induces a listing $\theta^* : U_n \xslashedrightarrow[]{} U_k$ (the dotted arrow in terms of the diagram in the proof of the previous proposition).
    Canonical choice follows from the fact that invertible listings are isomorphisms.
\end{proof}

\subsection{Fundamental simplicial lists}

A nice feature of simplicial sets is that for a simplicial set $X$, an $n$-simplex $x \in X_n$ is represented by a morphism $x : \D{n} \to X$. Moreover, each simplicial set $X$ is the colimit of its simplices. These facts are consequences of the category of simplicial sets being a presheaf category, but they also capture the geometric aspects of simplex formations. We develop a similar understanding for simplicial lists.

Let $X \in \slist$. An $n$-simplex $x \in X_n$, when regarded as a singleton sequence $(x) \in \CL X_n$, is represented by a listing $(x) : \D{n} \xslashedrightarrow{} X$. By applying the perfect-function factorization
\[\begin{tikzcd}
	{\D{n}} && X \\
	& {U_x}
	\arrow["{(x)}", "\shortmid"{marking}, from=1-1, to=1-3]
	\arrow["perfect"', "\shortmid"{marking}, from=1-1, to=2-2]
	\arrow["function"', from=2-2, to=1-3]
\end{tikzcd}\]
we may represent $x$ functionally via a morphism $U_x \to X$ for some simplicial list $U_x$ which is determined up to unique isomorphism. We may say that the simplex $x$ is of \emph{shape} $U_x$. This way, we are lead to the following class of simplicial lists which serve as candidate shapes for simplices.

\begin{definition}[Fundamental simplicial list]

A simplicial list $U$ is said to be \emph{fundamental} if it is equipped with a perfect listing $u : \D{n} \xslashedrightarrow{} U$ for some $n \geq 0$, such that $u(1_{[n]})$ is a singleton sequence. In this case we identify $u \in U_n$ as an $n$-simplex referred to as the \emph{fundamental $n$-simplex} of $U$. 

We denote the collection of fundamental simplicial lists indexed by a perfect listing from $\D{n}$ by $\F{n}$ and by $\CF \subseteq \slist$ the full subcategory of simplicial lists spanned by the fundamental ones.  
    
\end{definition}

\begin{notation}
    We write $(U,u) \in \F{n}$ to indicate that $U$ is a fundamental simplicial list indexed by a perfect listing from $\D{n}$ with fundamental simplex $u$. 
\end{notation}

\begin{example}
    For all $n \geq 0$, the simplicial set $\D{n}$ is a fundamental simplicial list via the identity morphism. The fundamental $n$-simplex is the identity function $1_{[n]} \in \D{n}_n$.
\end{example}

Being indexed by objects of $\D{}$, fundamental simplicial lists inherit a number of pleasant properties.

\begin{lemma} \label{lemma : properties of F}
    Let $(U,u) \in \F{n}$, $n \geq 0$. 
    \begin{itemize}
        \item [(i)] Each simplex $u^\prime \in U_k$, $k \geq 0$, appears exactly once in sequence $\theta^* u$ for a unique simplicial operation $\theta : [k] \to [n]$ in $\D{}$.

        \item[(ii)] For a simplicial list $X$, a morphism $U \to X$ is uniquely determined by the image of $u$.

        \item[(iii)] Let $(V, v) \in \F{k}$, $k \geq 0$. A morphism $\gamma : V \to U$ in $\CF$ is uniquely determined by a simplicial operation $\theta : [k] \to [n]$ and an index $i$ such that $\gamma(v) = (\theta^* u)_i$. 
    \end{itemize}
\end{lemma}

\begin{proof}
    \begin{itemize}
        \item [(i)]  Since $u : \D{n} \xslashedrightarrow{} U$ is perfect, at the $k$-th level the listing between sets $u_k : \D{n}_k \xslashedrightarrow{} U_k$ is perfect. Thus, given a simplex $u^\prime \in U_k$ there is a unique $\theta : [k] \to [n]$ and index $i$ such that $u^\prime$ is the $i$-th term of the sequence $\theta^* u$.

        \item[(ii)] Let $f : U \to X$ be a morphism of simplicial lists and let $x = f(u)$. Since any other simplex $u^\prime \in U_k$ is of the form $(\theta^* u)_i$ by (i), it follows that $f(u^\prime) = (\theta^* x)_i$ in $X$ as $f$ is structure preserving. This means that the image of $u$ determines the image of any other simplex of $U$ under $f$.

        \item[(iii)]  In case $(V,v) \in \F{k}$ and $\gamma : V \to U$ is a morphism of simplicial lists, by (ii), it is uniquely determined by the image of the fundamental simplex $\gamma(v) \in U_k$. By (i), $\gamma(v) = (\theta^* u)_i$ for some $\theta : [k] \to [n]$ and index $i$. 
    \end{itemize}
\end{proof}

\begin{notation}
    For $(U,u) \in \F{n}$ and $(V,v) \in \F{k}$
    we denote by $\gamma_{\theta, i} : V \to U$ the unique morphism in $\CF$ such that $\gamma_{\theta, i}(v) = (\theta^* u)_i$. Here, $\theta : [k] \to [n]$ is a morphism in $\D{}$ and $i$ is an indexing number. 
\end{notation}

\begin{remark}
    It follows from the above lemma that in case two fundamental simplicial lists are isomorphic, then this isomorphism is unique. Indeed, an invertible morphism $(U, u) \to (V,v)$ has to map $u \mapsto v$. This justifies the following abuse: we identify $\F{n}$ with the set of isomorphism classes of fundamental simplicial list indexed by $\D{n}$ and thus regard $\cc{F}$ to be skeletal. This is mostly to avoid set-theoretic issues, as for instance in the following proposition.
\end{remark}

As mentioned in the beginning of the section, fundamental simplicial lists are designed to be \emph{simplex classifiers}. 

\begin{proposition}[Classification of simplices]
Let $X$ be a simplicial list. There is a natural isomorphism
$$X_n \cong \coprod_{U \in \F{n}} \slist(U , X)$$
\end{proposition}

\begin{proof}
    Let $x \in X_n$. As pointed out previously, by regarding $(x) \in \lx_n$ as a singleton sequence, we may present it with a listing $(x) : \D{n} \xslashedrightarrow[]{} X$. This listing factors as a perfect listing $u_x : \D{n} \xslashedrightarrow[]{} U_x$ followed by a function $f_x : U \to X$ for a unique $U_x$. Notice that $u_x(1_{[n]})$ has to be a singleton sequence in $\CL (U_x)_n$ since $f_x \circ u_x$ classifies a singleton sequence in $\lx_n$ and $f_x$ is a function pointwise. Thus, $U_x \in \F{n}$ and we have a well-defined mapping $x \mapsto (f_x : U_x \to X)$. 

    On the other hand, given $u : \D{n} \xslashedrightarrow[]{} U$ in $\F{n}$ and a map $f : U \to X$, we obtain a simplex $f(u) \in X_n$ (Lemma \ref{lemma : properties of F} (ii)). This gives us a mapping $(f : U \to X) \mapsto f(u)$.

    By the uniqueness of the perfect-function factorization, these mappings are inverses of each other. Naturality is clear.
\end{proof}

Let $X$ be a simplicial list. In light of the above, we can see that, besides the fact that simplices $x \in X_n$ are represented by morphisms $U \to X$, $u \mapsto x$ from fundamental simplicial lists, a simplicial relation in $X$ between simplices of the form $y = (\theta^* x)_i$ is represented by a commutative triangle
\[\begin{tikzcd}
	V && U \\
	& X
	\arrow["x", from=1-3, to=2-2]
	\arrow["y"', from=1-1, to=2-2]
	\arrow["{\gamma_{\theta, i}}", from=1-1, to=1-3]
\end{tikzcd}\]

Commutative triangles as a above are precisely the morphisms in the comma category $(\CF \downarrow X)$. We regard the latter as the \emph{category of simplices} of $X$ and denote it by $\Spx(X)$. In other words, $\Spx(X)$ is defined by saying that:
\begin{itemize}
    \item [-] Simplices $x \in X_n$, $n \geq 0$, are the objects.
    \item[-] There is a morphism $y \to x$ between $x \in X_n$ and $y \in X_k$ in case of a simplicial relation $y = (\theta^* x)_i$.
\end{itemize}
There is a natural forgetful functor 
\[
\begin{tikzcd}
    \chi_X : &[-3em] \Spx(X) \arrow[r] & \slist \\[-2em]
    & U \to X \arrow[r, maps to] & U
\end{tikzcd}
\]

\begin{proposition}[Colimit of simplices] \label{prop : colimit of simplices}

Every simplicial list $X$ is a colimit of its simplices, i.e. there is an isomorphism
$$X \cong \displaystyle \colim_{U \to X} U$$
    
\end{proposition}

The above proposition can be derived as a corollary of (the proof) of the following key theorem of this paper, so we present its proof afterwards.

\begin{theorem}[Presheaf Theorem]
    There is an equivalence of categories
    $$\slist \cong \set^{\CF^{op}}$$
\end{theorem}

\begin{proof}
    Let $X \in \slist$. We may form the functor
    \[
    \begin{tikzcd}
        H_X :&[-3em] \CF^{op} \arrow[r] &\set \\[-2em]
        & U  \arrow[r, maps to] & \slist(U ,X)
    \end{tikzcd}
    \]
    Moreover, $X \mapsto H_X$ gives us a functor $H_\bullet : \slist \to \set^{\CF^{op}}$.

    Let $G : \CF^{op} \to \set$ be a presheaf. We defıne a simplicial list 
    \[
    \begin{tikzcd}
        Y_G : &[-3em] \D{op} \arrow[r] & \llist \\[-2em]
        & \left[n \right]  \arrow[r, maps to] & \displaystyle \coprod_{U \in \F{n}} G(U)
    \end{tikzcd}
    \]
    If we denote $G(\gamma_{\theta, i}) = G_{\theta, i}$, given $x \in G(U)$, we define $(\theta^* x)_i = G_{\theta, i}(x)$. The mapping $G \mapsto Y_G$ gives rise to a functor $Y_\bullet : \set^{\CF^{op}} \to \slist$.

    Let us see that the above functors are weak inverses of each other. One direction is easy. For $X \in \slist$, the simplicial list $Y_{H_X}$ has as $n$-simplices the set $\displaystyle \coprod_{U \in \F{n}} \slist(U, X)$. The latter is naturally isomorphic to $X_n$ (Proposition \ref{prop : simplex classification}).

    For the other direction, let $G : \CF^{op} \to \set$ be a presheaf. The presheaf $H_{Y_G}$ assigns to an object $U \in \F{n}$ the set $\slist(U, Y_G)$. But, every simplex of $Y_G$ which belongs to the set $G(U)$ is $U$-shaped. This is so because if we are given $y \in G(U)$, we can form a morphism of simplicial lists $\eta_y : U \to Y_G$, $u \mapsto y$. This much information is enough to construct $\eta_y$, since given a simplex $\gamma_{\theta, i} : V \to U$ in $U$ (Lemma \ref{lemma : properties of F} (iii)) there is a structure function $G_{\theta, i} : G(U) \to G(V)$ and we can declare $\eta_y(\gamma_{\theta,i}) = G_{\theta,i}(y)$. 
    
    Thus, we obtain a bijection $G(U) \cong \slist(U, Y_G)$, $y \mapsto \eta_y$, which is natural in $U$ by construction and therefore results in a natural isomorphism $G \cong H_{Y_G}$.
\end{proof}

\begin{corollary}
    The category $\slist$ is complete and cocomplete.
\end{corollary}

\begin{remark}
    Notice that the equivalence $\slist \cong \set^{\cc{F}^{op}}$, $X \mapsto H_X$, constructed in the above proof is easily identified as the right adjoint induced by the inclusion $\cc{F} \subseteq \slist$ via Yoneda extension.
\end{remark}

\begin{proof} (of Proposition \ref{prop : colimit of simplices})

Let $X$ be a simplicial list. Following up on the above remark and using notation from the above theorem we have that
\begin{align*}
    X &\cong Y_{G_X} \\
    &\cong \displaystyle \colim_{U \in H_X(U)} U \\
    &\cong \displaystyle \colim_{U \to X} U
\end{align*}
    
\end{proof}

\begin{note}[Two notions of pointwise] \label{note_pointwise}

We have two notions of pointwise for simplicial lists: one from the simplicial structure and one from the presheaf structure. For example, limits of simplicial lists are computed pointwise in the presheaf structure, but they are not constructed simplicially pointwise. For instance, for simplicial lists $X$ and $Y$ we have that $$(X \times Y)_n \cong \displaystyle \coprod_{U \in \F{n}} (\slist(U, X) \times \slist(U, Y))$$ which means that the cartesian product in $\slist$ is formed by pairs of simplices of the same shape. This is different from $X_n \times Y_n$. On the other hand, colimits are computed pointwise in both senses. 
    
\end{note}

\subsection{Types of simplicial lists}

We have already encountered two types of simplicial lists, simplicial sets and operadic simplicial lists. Both are defined by requiring that a certain class of simplicial operations are functions. This leads us to consider the following generalisation.

\begin{definition}
    Let $S \subseteq \mor(\D{})$ be a subset of morphisms in $\D{}$. We say that a simplicial list $X$ is of type $S$ if the structure map $\theta^*$ of $X$ is a function for all $\theta \in S$. We denote by $\slist_S \subseteq \slist$ the full subcategory spanned by simplicial lists of type $S$.
\end{definition}

\begin{example}
    If $S = \mor(\D{})$, then $\slist_S = \sset$.
\end{example}

\begin{example}
    If $S = \{ \theta : [k] \to [n] : \theta(k) = n \}$, then $\slist_S = \slisto$. 
\end{example}

\begin{lemma}
    Let $S \subseteq \mor(\D{})$. If there is a morphism $f : X \to Y$ in $\slist$ and $Y$ is of type $S$, then so is $X$.
\end{lemma}

\begin{proof}
    Let $\theta : [k] \to [n]$ be a morphism in $S$. We have a commutative square in $\llist$
    \[\begin{tikzcd}
	{X_n} & {X_k} \\
	{Y_n} & {Y_k}
	\arrow["{\theta^*}"', from=2-1, to=2-2]
	\arrow["{f_n}"', from=1-1, to=2-1]
	\arrow["{f_k}", from=1-2, to=2-2]
	\arrow["{\theta^*}", "\shortmid"{marking}, from=1-1, to=1-2]
\end{tikzcd}\]
    Since $Y$ is of type $S$, the map on the bottom is a function. It follows by inspection that the map on the top is also a function.
\end{proof}

\begin{proposition}
    Let $S \subseteq \mor(\D{})$, and let $\CF_S \subseteq \CF$ be the full subcategory spanned by objects of type $S$. Then, there is an equivalence of categories
    $$\slist_S \cong \set^{\CF_S^{op}}$$
\end{proposition}

\begin{proof}
    Let $X \in \slist_S$. In the presheaf formation of $X$, by the previous lemma, the set $\slist(U, X) = \emptyset$ unless $U \in \CF_S$. Hence the desired conclusion.
\end{proof}

\begin{notation}
    We denote by $\Upsilon$ the category of fundamental operadic simplicial lists.
\end{notation}

\begin{corollary}
    There is an equivalence of categories 
    $$\slisto \cong \set^{\Upsilon^{op}}$$
\end{corollary}

\begin{remark}
    The reason we choose this level of generality, besides the easy derivation of the above important corollary, is that we believe interesting classes emerge by letting $S$ vary. For instance, we would call a simplicial list \emph{properadic} in case $S = \{ \theta : [k] \to [n] : \theta(0)= 0 , \theta(k) = n \}$. Choosing the subset $S$ of interest restricts the type of shapes from $\CF$ to $\CF_S$ according to the type of structure we are studying. In this vein, simplicial sets are the \emph{categorical} simplicial lists.
\end{remark}

\subsection{Operadic simplicial lists} \label{subsec : operadic}

While it is important to know that $\slisto$ is a presheaf category, in order to work properly with the presheaf structure we need a good understanding of the base $\Upsilon$. We construct and classify fundamental operadic simplicial lists.

\begin{construction} \label{const : Ua}
    Let $\alpha \in \nd_n$ for some $n \geq 0$. Recall the usual convention of writing $\alpha$ as
\[\begin{tikzcd}
	{\alpha : } &[-2em] {A_0} & {A_1} & \dots & {A_n}
	\arrow["{\alpha_1}", from=1-2, to=1-3]
	\arrow["{\alpha_2}", from=1-3, to=1-4]
	\arrow["{\alpha_n}", from=1-4, to=1-5]
\end{tikzcd}\]
and writing $\alpha_{i,j} : A_i \to A_j$ for the appropriate composite of morphisms in the chain. We define the simplicial list $U_\alpha$ with set of $k$-simplices given by
$$U_{\alpha,k} = \coprod_{\theta : [k] \to [n]} A_{\theta(k)}$$
In other words, a $k$-simplex in $U_\alpha$ is a pair $(\theta, a)$, where $\theta : [k] \to [n]$ is a morphism in $\D{}$ and $a \in A_{\theta(k)}$, for all $k \geq 0$.

Let $\eta : [l] \to [k]$ be a morphism in $\D{}$ and $(\theta, a)$ be a $k$-simplex in $U_\alpha$. The action of $\eta$ on this simplex is given by
$$\eta^*(\theta , a) = \{ (\theta \eta , b) \ : \ \alpha_{\theta \eta (l), \theta(k)}(b) = a \}$$
The order on the set on the right is inherited from $A_{\theta \eta (l)}$ and we regard it as a sequence of $l$-simplices. 
We see that $U_\alpha$ is operadic, since, in case $\eta(l) = k$, the sequence on the right is just the singleton $(\theta \eta, a)$.

Moreover, there is an evident listing 
\[
\begin{tikzcd}
    u_\alpha : &[-3em] \D{n}  \arrow[r] & U_\alpha \\[-2em]
    & \theta : [k] \to [n]  \arrow[r, maps to] & ((\theta, a))_{a \in A_{\theta(k)}}
\end{tikzcd}
\]
which is perfect since, by construction, every simplex in $U_\alpha$ appears exactly once as a term in an image sequence of the $u_\alpha$. 
\end{construction}

\begin{theorem}
    An operadic simplicial list $U$ is fundamental if and only if $U \cong U_\alpha$ for some $\alpha \in \nd^{root}$. 
\end{theorem}

\begin{proof}
    Let $\alpha \in \nd_n^{root}$. As demonstrated above, the listing $u_\alpha : \D{n} \xslashedrightarrow{} U_\alpha$ is perfect for all $\alpha$. By definition, we have $u_\alpha(1_{[n]}) \cong A_n$. Hence, this $U_\alpha$ is fundamental just in case $A_n \cong [0]$, i.e. in case $\alpha$ is rooted. 

    For the converse, let $u : \D{n} \xslashedrightarrow{} U$ be a fundamental operadic simplicial list. Let us denote $u(i) : A_i \to U_0$ for some $A_i \in \dplus$, $0 \leq i \leq n$. Since $u$ is perfect, we have $U_0 \cong \displaystyle \coprod_{i=0}^n A_i$. 
    Similarly, for $\theta : [k] \to [n]$ in $\D{n}_k$, let us denote $u(\theta) : A_\theta \to U_k$. Since $u$ is perfect, we must have $U_k \cong \displaystyle \coprod_{\theta : [k] \to [n]} A_\theta$. In particular, since $U$ is fundamental, we obtain $A_n \cong A_{1_{[n]}} \cong [0]$.

    Let $t_k : [0] \to [k]$ be the last vertex mapping $0 \mapsto k$. We have a commutative square in $\llist$
    \[\begin{tikzcd}
	{\D{n}_k} & {U_k} \\
	{\D{n}_0} & {U_0}
	\arrow["{u_k}", "\shortmid"{marking}, from=1-1, to=1-2]
	\arrow["{u_0}"', "\shortmid"{marking}, from=2-1, to=2-2]
	\arrow["{t_k^*}", from=1-2, to=2-2]
	\arrow["{t_k^*}"', from=1-1, to=2-1]
\end{tikzcd}\]
    The listings on the top and bottom are perfect, while the structure map on the right is a function since $U$ is operadic. Therefore, when restricting to a particular $\theta : [k] \to [n]$ in $\D{n}_k$, we deduce that 
    $$A_\theta \cong A_{\theta(k)}$$

    Now, we recover a simplex $\alpha \in \nd^{root}_n$ with $\alpha(i) = A_i$.
    In case $k=1$, we have a commutative square in $\llist$ 
    \[\begin{tikzcd}
	{\D{n}_1} & {U_1} \\
	{\D{n}_0} & {U_0}
	\arrow["{u_1}", "\shortmid"{marking}, from=1-1, to=1-2]
	\arrow["{u_1}"', "\shortmid"{marking}, from=2-1, to=2-2]
	\arrow["{d_1^*}"', from=1-1, to=2-1]
	\arrow["{\underline{d}_1^*}", "\shortmid"{marking}, from=1-2, to=2-2]
\end{tikzcd}\]
    Every map $[1] \to [n]$ amounts to choosing a pair $i,j \in [n]$ with order $i \leq j$. 
     When restricting the above commutative square to a single element $(i,j) \in \D{n}_1$, we have a commutative square in $\llist$
    \[\begin{tikzcd}
	{\{(i,j) \}} & {A_{(i,j)} \cong A_j} \\
	{\{i \}} & {A_i}
	\arrow["{u_1}", "\shortmid"{marking}, from=1-1, to=1-2]
	\arrow["{u_0}"', "\shortmid"{marking}, from=2-1, to=2-2]
	\arrow["{d_1^*}", "\shortmid"{marking}, from=1-2, to=2-2]
	\arrow["{d_1^*}"', from=1-1, to=2-1]
\end{tikzcd}\]
    The listings on top and bottom are still perfect. Hence, we obtain a function
    $$\alpha_{i,j} : A_i \to A_j$$
    by declaring that $\alpha_{i,j}(a) = b$ in case $a$ appears as a term in the sequence $d_1^*(b)$. It is not difficult to show that the simplicial structure of $U$ is precisely that of $U_\alpha$ constructed above. 
    
\end{proof}

\begin{note} [$U_\alpha$ as a partial operad] \label{note: partial operad}
    The above combinatorics is somewhat involved, but there is an angle on the construction of $U_\alpha$ we find explanatory. Let $\alpha \in \nd_n^{root}$. Combining a little jargon from operads and simplices, we may regard the set of $0$-simplices $\displaystyle \coprod_{i=1}^n A_i$ as the \emph{colors} of $U_\alpha$. The \emph{operations} ($1$-simplices) are declared  in the form $(\theta, a)$ where $\theta : [1] \to [n]$ and $a \in A_{\theta(1)}$. 
    The simplex $\theta$ is just the choice of two indexes $i,j$ with order $i \leq j$, while the specification of an element $a \in A_j$ just says that this operation is of the form $\alpha_{i,j}^{-1}(a) \to a$.
    
    In similar fashion, $2$-simplices are encoded in the form $((i,j,k), a)$, where $i,j,k$ are indexes with order $i \leq j \leq k$ and $a \in A_k$. These record a \emph{partial composition} operation of the appropriate operations. We may think of the composition scheme as follows: an operation $g$ may be composed with a sequence of operations $\underline{f} = (f_1, \dots , f_l)$ if and only if $d_0 g = d_1 \underline{f}$ and moreover all the terms of $\underline{f}$ belong to the same level. 
\end{note}

We will not formulate a notion of partial operad here, even though it is a useful heuristic in understanding the formation of $U_\alpha$ and a thick version of it we will construct in the next section. Let us illustrate with an example.

\begin{example}
    Let $\alpha$ be the rooted $2$-simplex in $\nd_2^{root}$ which records the rooted tree
     
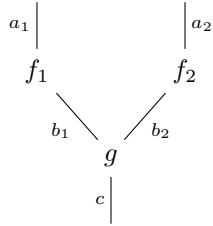
\[\begin{tikzcd}[column sep=small]
	{} && {} \\
	{f_1} && {f_2} \\
	& g \\
	& {}
	\arrow["{a_1}"', no head, from=1-1, to=2-1]
	\arrow["{b_1}"', no head, from=2-1, to=3-2]
	\arrow["{b_2}", no head, from=2-3, to=3-2]
	\arrow["c"', no head, from=3-2, to=4-2]
	\arrow["{a_2}", no head, from=1-3, to=2-3]
\end{tikzcd}\]
    The $0$-simplices, or colors, of $U_\alpha$ are just $a_1, a_2, b_1, b_2, c$. The non-degenerate $1$-simplices, or operations, are $f_1, f_2 ,g$ and moreover, we have another (non-depicted) $1$-simplex $h : (a_1, a_2) \to c$. We may think of $h$ as the “composite" $g \circ (f_1, f_2)$. This composition is witnessed by the unique non-degenerate 2-simplex of $U_\alpha$ (the fundamental one $u_\alpha$). 

    Notice that in $U_\alpha$ we do not have other compositions, since composition is only allowed level-wise. For instance, $g \circ (f_1, 1_{b_2})$ is not allowed since $f_1$ is a term in the sequence $u_\alpha ((0,1))$ while $1_{b_2}$ is a term in the sequence $u_\alpha((0,0))$, where $(0,1)$ and $(1,1)$ are the indexing $1$-simplices in $\D{2}$. 
\end{example}

\begin{proposition}
    The category $\Upsilon$ of fundamental operadic simplicial lists can be described as follows:
    \begin{itemize}
        \item [-] The objects are rooted simplices $\alpha \in \nd_n^{root}$, $n \geq 0$. 
        \item[-] A morphism $(\theta, a) : \beta \to \alpha$ between $\beta \in \Upsilon^{(k)}$ and $\alpha \in \Upsilon^{(n)}$ consists of a map $\theta : [k] \to [n]$ and an element $a \in A_{\theta(k)}$ such that $\beta = (\theta^*\alpha)_a$, which means that $\beta$ is the rooted restriction (Construction \ref{const : rooted restriction}) of $\theta^* \alpha$ at $a$.
    \end{itemize}
\end{proposition}

\begin{proof}
    The above proposition tells us that fundamental operadic simplicial lists are of the form $U_\alpha$ for $\alpha \in \nd^{root}$. Since $U_\alpha$ is uniquely determined by $\alpha$, we may identify the objects of $\Upsilon$ with rooted simplices of the category $\dplus$. 
    
    Let $\phi : U_\beta \to U_\alpha$ be a morphism in $\Upsilon$.
    By the general description of morphisms in $\CF$ (Lemma \ref{lemma : properties of F}), $\phi$ is uniquely determined by the image of the fundamental $k$-simplex in $U_\beta$. By construction of $U_\alpha$, we may construct a $k$-simplex of shape $\beta$ in $U_\alpha$ if and only if we have a map $\theta : [k] \to [n]$ such that $\beta = (\theta^* \alpha)_a$ for some $a \in A_{\theta(k)}$. Hence, we may identify $\phi$ with the pair $(\theta, a)$. It is easy to check that the above determine an isomorphism of categories.
\end{proof}

\begin{remark} 
    In \cite[Definition 2.4]{barwick2018operator}, Barwick constructs, for any operator category $\Phi$, a category $\D{}_\Phi$ called the category of $\Phi$-\emph{sequences}. Specializing to the case $\Phi = \dplus$, this category has the following description:
    \begin{itemize}
        \item[-] Objects are simplices $\alpha : \D{n} \to \dplus$, $[n] \in \D{}$.
        \item [-] For $\alpha \in N(\dplus)_n$ and $\beta \in \nd_k$, a morphism $(\theta, \phi) : \beta \to \alpha$ consists of a map $\theta : [k] \to [n]$ in $\D{}$ and a natural transformation $\phi : \beta \to \alpha \circ \theta$
        \[\begin{tikzcd}
	{\D{k}} && {\D{n}} \\
	& \dplus
	\arrow["\theta", from=1-1, to=1-3]
	\arrow[""{name=0, anchor=center, inner sep=0}, "\beta"', from=1-1, to=2-2]
	\arrow["\alpha", from=1-3, to=2-2]
	\arrow["\phi"{description}, shorten <=9pt, Rightarrow, from=0, to=1-3]
\end{tikzcd}\]
such that the following conditions are satisfied:
\begin{itemize}
    \item [(i)] For all $i \in [k]$, the component function $\phi_i : B_i \to A_{\theta(i)}$ is an interval inclusion (meaning an inclusion such that the image consists of consecutive terms of the target).
    \item[(ii)] For all $i \leq j$ in $[k]$, the square 
    \[\begin{tikzcd}
	{B_i} && {A_{\theta(i)}} \\
	{B_j} && {A_{\theta(j)}}
	\arrow["{\phi_i}", from=1-1, to=1-3]
	\arrow[from=1-1, to=2-1]
	\arrow[from=1-3, to=2-3]
	\arrow["{\phi_j}"', from=2-1, to=2-3]
\end{tikzcd}\]
is a pullback square.
\end{itemize}
    \end{itemize}

    Notice that a morphism $(\theta, \phi) : \beta \to \alpha$ as above is completely determined by $\theta : [k] \to [n]$ and an interval inclusion $B_k \subseteq A_{\theta(k)}$, as the rest of the simplex $\beta$ and components of $\phi$ are determined by condition $(ii)$. 
    
    In particular, in case $\beta$ is rooted, the latter amounts to the choice of an element $a \in A_{\theta(k)}$. Thus, this notion of morphism coincides with the one between rooted simplices of $\nd$ described in the above proposition. This way, we see that $\Upsilon$ can be interpreted as the full subcategory of $\D{}_\Phi$, $\Phi = \dplus$, spanned by rooted simplices. 

    This not only offers another description of morphisms in $\Upsilon$, but we may also conclude that Segal $\Upsilon$-spaces, meaning functors $\Upsilon^{op} \to \mathsf{Kan}$ valued in Kan complexes (or any other model for spaces, or in the $\infty$-category of spaces) which satisfy a Segal condition (see for instance \cite[Definition 2.6]{barwick2018operator} or \cite[Definition 2.7]{chu2021homotopy}), do model $\infty$-operads. For instance, an equivalence with Lurie's model is established in \cite[Section 10]{barwick2018operator}.

    We do not unpack the Segal conditions, although the reader can easily guess. In fact, we work in the quasi-categorical direction towards establishing a direct connection between simplicial operads and list quasi-operads. One reason for this is that simplicial lists and list quasi-operads are far easier to describe that Segal $\Upsilon$-spaces on the level of combinatorics. We intend to pursue a proper comparison in future work.
\end{remark}

    We use the description of $\Upsilon$ in the statement of the above proposition, while we write
    \[
    \begin{tikzcd}
        U_\bullet : &[-3em] \Upsilon \arrow[r] & \slisto \\[-2em]
        & \alpha \arrow[r, maps to] & U_\alpha
    \end{tikzcd}
    \]
    for the Yoneda embedding into $\slisto$. 

Having done the combinatorial heavy-lifting, we may reinterpret the nerve functor $N^l : \operad \to \slisto$ to be the right adjoint to the Yoneda extension of the functor $\alpha \mapsto T_\alpha$
\[\begin{tikzcd}
	& \slisto \\
	\Upsilon && \operad
	\arrow["{\alpha \mapsto T_\alpha}"', from=2-1, to=2-3]
	\arrow["{U_\bullet}", from=2-1, to=1-2]
	\arrow["N^l"{description}, from=2-3, to=1-2]
	\arrow["{\tau_1}"{description}, curve={height=-18pt}, dashed, from=1-2, to=2-3]
\end{tikzcd}\]
In fact, the presheaf structure indicates the nerve functor ought to be what it is, since the most evident operad one can construct out of a rooted simplex $\alpha \in \nd$ is the operad $T_\alpha$. Moreover, the Yoneda extension provides us with a \emph{fundamental operad}, or \emph{free operad}, functor
\[
\begin{tikzcd}
    \tau_1 :&[-3em] \slisto \arrow[r] & \operad
\end{tikzcd}
\]
for free. The fact that $\tau_1 U_\alpha = T_\alpha$, which follows by definition, is sensible when we regard $U_\alpha$ as a partial operad (Note \ref{note: partial operad}), in which case $T_\alpha$ is the free operad which contains the missing composites. 

\begin{note}
    A major difference between the category $\Upsilon$ and the category $\D{}$ is the fact that the representables $U_\alpha$ are not nerves of operads, while the simplicial sets $\D{n}$ are. This is not the case in dendroidal theory, where trees are fundamentally identified with the free operads they generate. On the other hand, the objects in our base $\Upsilon$ do behave like simplices in many other ways and morphisms between them are tractable combinatorially and there is far less of them compared to the dendroidal setting.  
\end{note}

\section{The coherent nerve} \label{section : coherent nerve}

In order to define a list coherent nerve functor $\ncl : \soperad \to \slisto$, we construct an appropriate functor $\VV : \Upsilon \to \soperad$ and proceed by Yoneda extension to obtain a \emph{rigidification} functor $\fo : \slisto \to \soperad$ which has as right adjoint $\ncl$. Taking our cue from the pattern developed in the ordinary nerve construction, we define simplicial operads $\VVa$ for $\alpha \in \Upsilon$ which, not only are cofibrant replacements of the operads $T_\alpha$, but are also simplicially well-behaved.

\subsection{Fundamental simplicial objects in $\slist$}

We begin with a preliminary discussion. Let $\Q$ be a simplicial operad.
The monoidal envelope $\llq$ inherits a simplicial enrichment from $\Q$ (see Remark \ref{remark : L discrete}). 
An $n$-simplex in the coherent nerve, say $x \in  \nc(\LQ)_n$, is represented by a simplicially enriched functor
\[
\begin{tikzcd}
    x :&[-3em] \DD{n} \arrow[r] & \LQ
\end{tikzcd}
\]
where $\DD{n}$ is the ordinary simplicial thickening of the category $\D{n}$ (see Appendix \ref{appendix_free}).

Consider the associated discrete simplicial objects $\DD{n}_\bullet , \LQ_\bullet : \D{op} \to \cat$. Recall that we have $(\LQ)_\bullet \cong \CL(\Q_\bullet)$, where $\Q_\bullet : \D{op} \to \operad$ is the discrete simplicial object associated to $\Q$ (Remark \ref{remark : L discrete}). Now, consider the morphism $x_\bullet : \DD{n}_\bullet \to \LQ_\bullet$ associated to $x$. 

By applying the nerve functor $N : \cat \to \sset$ dimension-wise, and using $N \CL \cong \CL N^l$ (Proposition \ref{prop : L commutes N}), we may regard $x_\bullet$ as a morphism $N(\DD{n}_\bullet) \to \CL N^l(\Q_\bullet)$ in $\sset^{\D{op}}$. Equivalently, $x_\bullet$ is a dimension-wise listing
\[\begin{tikzcd}[column sep=scriptsize]
	{x_\bullet :} &[-2em] {N(\DD{n}_\bullet)} && {N^l(\Q_\bullet)}
	\arrow["\shortmid"{marking}, from=1-2, to=1-4]
\end{tikzcd}\]
in $\widetilde{\slist}^{\D{op}}$. 
We may apply the perfect-function factorization dimension-wise to obtain 
\[\begin{tikzcd}[column sep=scriptsize]
	{N(\DD{n}}_\bullet) && N^l(\Q_\bullet) \\
	& \UU_x
	\arrow["x_\bullet", "\shortmid"{marking}, from=1-1, to=1-3]
	\arrow["{u_x}"', "\shortmid"{marking}, from=1-1, to=2-2]
	\arrow["{f_x}"', from=2-2, to=1-3]
\end{tikzcd}\]
for some simplicial object $\UU_x : \D{op} \to \slisto$. 

We see that the simplex $x$ is represented functionally by $f_x$. The list nerve functor $N^l$ is right adjoint to the free operad functor $\tau_1 : \slisto \to \operad$. Thus, $f_x$ can be transposed into a morphism $g_x : \VV_x \to \Q_\bullet$ in $\operad^{\D{op}}$, where $\VV_x = \tau_1 \UU_x$ (dimension-wise). 

Finally, it is easy to see (say, from the proof of Proposition \ref{prop : factorisation}) that $\UU_x$ has the following property: the simplicial lists $(\UU_x)_m$, $[m] \in \D{}$, have the same set of $0$-simplices and all simplicial operations are identity on $0$-simplices. Thus, the simplicial object $\VV_x \in \operad^{\D{op}}$ is discrete (in the sense of Remark \ref{remark : L discrete}) and may be regarded as a simplicial operad and $g_x : \VV_x \to \Q$ a morphism in $\soperad$.

Reverse engineering the above steps, we may conclude the following:
\begin{quote}
    For a simplicial operad $\Q$, simplices in the coherent nerve $\nc(\LQ)_n$, $[n] \in \D{}$, are classified by morphism of simplicial operads $\VV \to \Q$ for simplicial operads $\VV$ which are freely generated by simplicial objects $\UU : \D{op} \to \slisto$ which are indexed by a perfect listing with domain $N(\DD{n}_\bullet)$.
\end{quote}

In what follows, we classify the aforementioned simplicial objects $\UU$. As expected, such objects are indexed by simplices in the nerve of the category $\dplus$. 

\begin{construction}[Thick version of $U_\alpha$] \label{const_thickua}

Let $\alpha \in \nd_n$. We adopt the usual notation. For $[k], [m] \in \D{}$, define the set 
$$\UUa [k, m] = \coprod_{\theta : \D{k} \to \DD{n}_m} A_{\theta(k)} .$$
The mapping $[k, m] \mapsto \UUa [k,m]$ is functionally simplicial in the variable $m$ and listing simplicial in the variable $k$ (in the same manner $U_\alpha$ is, as described in Construction \ref{const : Ua}). 

This way, we may regard $\UUa$ as a simplicial object
\[
\begin{tikzcd}
    \UUa :&[-3em] \D{op} \arrow[r] & \slist \\[-2em]
    & \left[m \right] \arrow[r, maps to] & \UUa[- , m]
\end{tikzcd}
\]
For ease, we will denote $\UUa^{(m)} = \UUa[-,m]$ for all $m \geq 0$.

Notice that for all $[m] \in \D{}$ we have $$\UUa [0, m] \cong \coprod_{i \in [n]} A_i$$ and moreover, the simplicial operations are identities on the latter set. Hence, applying the free operad functor produces a simplicially enriched operad which we denote 
$$\VVa = \tau_1 \UUa .$$
    
\end{construction}

\begin{remark}
    $\UUa$ may be regarded to be obtained by freely resolving the partial level-wise composition operation in $U_\alpha$ discussed in Note \ref{note: partial operad}.
\end{remark}

\begin{proposition} \label{prop : UUa augmentation}
    For all $\alpha \in \nd$, there is an augmentation $\UUa \dashrightarrow U_\alpha$ with extra degeneracies. 
\end{proposition}

\begin{proof}
    The simplicial category $\DD{n}$, by virtue of being a free resolution, is equipped with an augmentation $\DD{n} \dashrightarrow \D{n}$ with extra degeneracies
    \[\begin{tikzcd}
	\dots & {\DD{n}_1} && {\DD{n}_0} & {\D{n}}
	\arrow["{d_0}", from=1-4, to=1-5]
	\arrow["{d_0}", shift left=3, from=1-2, to=1-4]
	\arrow["{d_1}"', shift right=3, from=1-2, to=1-4]
	\arrow["{s_0}"{description}, from=1-4, to=1-2]
	\arrow["{s_{-1}}", curve={height=-12pt}, dashed, from=1-5, to=1-4]
	\arrow["{s_{-1}}", curve={height=-24pt}, dashed, from=1-4, to=1-2]
\end{tikzcd}\]
    This induces an augmentation with extra degeneracies $\UUa \dashrightarrow U_\alpha$
    \[\begin{tikzcd}
	\dots & {\UUa^{(1)}} && {\UUa^{(0)}} & {U_\alpha}
	\arrow["{d_0}", from=1-4, to=1-5]
	\arrow["{d_0}", shift left=3, from=1-2, to=1-4]
	\arrow["{d_1}"', shift right=3, from=1-2, to=1-4]
	\arrow["{s_0}"{description}, from=1-4, to=1-2]
	\arrow["{s_{-1}}", curve={height=-12pt}, dashed, from=1-5, to=1-4]
	\arrow["{s_{-1}}", curve={height=-24pt}, dashed, from=1-4, to=1-2]
\end{tikzcd}\]
    The morphism $d_0 : \UUa^{(0)} \to U_\alpha$ is defined, for $\theta : \D{k} \to \DD{n}_0$ and $a \in A_{\theta(k)}$, by
    $$d_0(\theta, a) = (d_0 (\theta), a)$$
    where $d_0(\theta)$ is the composite of $\theta$ with the augmentation map $d_0 : \DD{n}_0 \to \D{n}$. Similarly, the extra degeneracies $s_{-1} : \UUa^{(m-1)} \to \UUa^{(m)}$ are defined, for $\theta : \D{k} \to \DD{n}_{m-1}$ and $a \in A_{\theta(k)}$ by
    $$s_{-1}(\theta, a) = (s_{-1}(\theta), a)$$
    where $s_{-1}(\theta)$ is the composite of $\theta$ with the extra degeneracy $s_{-1} : \DD{m-1} \to \DD{m}$ pertaining to the resolution $\DD{n}$. 
\end{proof}

\begin{corollary}
    There is a weak equivalence of simplicial operads $\VVa \to T_\alpha$ for all $\alpha \in \nd$.
\end{corollary}

\begin{proof}
    By applying the fundamental operad functor $\tau_1$ to $\UUa \dashrightarrow U_\alpha$, we obtain an augmentation with extra degeneracies $\VVa \dashrightarrow T_\alpha$. The induced map $\VVa \to T_\alpha$ is identity on objects by construction. In particular, given any sequence of colors $\ua$ and color $b$ in $\VVa$, we have an augmentation with extra degeneracies $\VVa(\ua, b) \dashrightarrow T_\alpha(\ua, b)$ from which we conclude that the induced map $\VVa(\ua, b) \rightarrow T_\alpha(\ua, b)$ is a weak equivalence of simplicial sets. Hence the conclusion.
\end{proof}

\begin{theorem} \label{prop_classsslist}
    Let $\UU : \D{op} \to \slisto$ be a simplicial object. Then, there exists a pointwise perfect listing $\DD{n} \xslashedrightarrow{} \UU$ if and only if $\UU \cong \UUa$ for some $\alpha \in \nd$.
\end{theorem}

\begin{proof}
    Let $\alpha \in \nd_n$.
    We may define a listing 
    \[
    \begin{tikzcd}
        \eta_\alpha :&[-3em] \DD{n} \arrow[r, "\shortmid"{marking}] & \UUa \\[-2em]
        &\theta : \D{k} \to \DD{n}_m \arrow[r, maps to] & ((\theta , a))_{a \in A_{\theta(k)}}
    \end{tikzcd}
    \]
    for all $k,m \geq 0$. This listing is perfect by construction of $\UUa$.

    For the converse, let $\UU \in \slisto^{\D{op}}$ and let $\eta : \DD{n} \xslashedrightarrow{} \UU$ be a perfect listing. We write $\UU[k,m]$ for the $k$-simplices of the simplicial list $\UU_m$. This way, for all $k,m \geq 0$ we have a perfect listing of sets
    \[
    \begin{tikzcd}
        \eta_{k,m} :&[-3em] N(\DD{n}_m)_k \arrow[r, "\shortmid"{marking}] & \UU[k,m]
    \end{tikzcd}
    \]
    For $\theta : \D{k} \to \DD{n}_m$, let us denote $\eta_{k,m}(\theta) : A_\theta \to \UU[k,m]$, where $A_\theta \in \dplus$. Since $\eta$ is perfect, we have an isomorphism $$\UU[k,m] \cong \coprod_{\theta : \D{k} \to \DD{n}_m} A_\theta.$$

    Consider the case $k=0$. Let $A_i$ be the ordered set corresponding as above to the morphism $\{ i \} : \D{0} \to \DD{n}_0$ which picks the object $i \in \DD{n}_0$. In general, let $A_i^{(m)}$ denote the ordered set corresponding to the morphism $\{ i \} : \D{0} \to \DD{n}_m$. Consider the terminal map $s : [m] \to [0]$ in $\D{}$. We have a commutative square in $\llist$
    \[\begin{tikzcd}
	{N(\DD{n}_0)_0} && {\UU[0,0]} \\
	{N(\DD{n}_m)_0} && {\UU[0,m]}
	\arrow["{s^*}"', from=1-1, to=2-1]
	\arrow["{\eta_{0,m}}"', "\shortmid"{marking}, from=2-1, to=2-3]
	\arrow["{\eta_{0,0}}", "\shortmid"{marking}, from=1-1, to=1-3]
	\arrow["{s^*}", from=1-3, to=2-3]
\end{tikzcd}\]
    The listings on the top and bottom are perfect. By restricting to elements $\{ i\} : \D{0} \to \DD{n}_0$ we see that there is an isomorphism $A_i \cong A_i^{(m)}$ for all $i$ and for all $[m] \in \D{}$.

    For general $k$, consider the morphism $t_k : [0] \to [k]$, $0 \mapsto k$, in $\D{}$. We have a commutative square in $\llist$
    \[\begin{tikzcd}
	{N(\DD{n}_m)_k} && {\UU[k,m]} \\
	{N(\DD{n}_m)_0} && {\UU[0,m]}
	\arrow["{t_k^*}"', from=1-1, to=2-1]
	\arrow["{\eta_{0,m}}"', "\shortmid"{marking}, from=2-1, to=2-3]
	\arrow["{\eta_{k,m}}", "\shortmid"{marking}, from=1-1, to=1-3]
	\arrow["{t_k^*}", from=1-3, to=2-3]
\end{tikzcd}\]
    Again, the horizontal listings are perfect. The map on the right is a function since $\UU_m$ is assumed to be operadic. By restricting to some element $\theta : \D{k} \to \DD{n}_m$ in $N(\DD{n}_m)_k$, we see that we must have an isomorphism $A_\theta \cong A_{\theta(k)}^{(m)}$. Combining with the previous result, we conclude that $A_\theta \cong A_{\theta(k)}$ for all $\theta$.

    By looking at the spine of $\DD{n}$ it is not difficult to recover morphisms $\alpha_{i} : A_{i-1} \to A_i$ so that we form a simplex $\alpha \in \nd_n$ with $\alpha(i) = A_i$. It is also easy to check  that the simplicial operations of $\UU$ agree with those of $\UUa$.

\end{proof}

The above classification has immediate important consequences.

\begin{corollary} \label{prop_thicksimplex}
    Let $\Q \in \soperad$. There is a natural bijection
    $$\nc (\LQ)_n \cong \coprod_{\alpha \in \nd_n} \soperad(\VVa ,  \Q)$$
    for all $[n] \in \D{}$.
\end{corollary}

\begin{proof}
    This follows by the discussion in the beginning of this section.
\end{proof}

\begin{corollary}
    For all $\alpha \in \nd$, the simplicial operad $\VVa$ is cofibrant.
\end{corollary}

\begin{proof}
    Consider a lifting problem 
    \[\begin{tikzcd}
	& {\cc{K}} \\
	\VVa & \Q
	\arrow[from=2-1, to=2-2]
	\arrow["\sim", two heads, from=1-2, to=2-2]
	\arrow[dashed, from=2-1, to=1-2]
\end{tikzcd}\]
    in the category of simplicial operads, where the map on the right is a trivial fibration. By the previous result, we may transpose this lifting problem into the category of simplicial categories 
    \[\begin{tikzcd}
	& {\CL \cc{K}} \\
	{\bb{\D{n}}} & \LQ
	\arrow[from=2-1, to=2-2]
	\arrow["\sim", two heads, from=1-2, to=2-2]
	\arrow[dashed, from=2-1, to=1-2]
\end{tikzcd}\]
    The morphism on the right is still a trivial fibration, since the functor $\CL$ preserves trivial fibrations (Lemma \ref{lemma : Lwefib}). Hence, since $\bb{\D{n}}$ is a cofibrant simplicial category, this lifting problem has a solution. By commutativity, any chosen lift has to be of shape $\alpha$. Therefore, the original lifting problem has a solution.
\end{proof}

\begin{corollary}
    For all $\alpha \in \nd$, $\VVa$ is a cofibrant replacement of $T_\alpha$ in the category of simplicial operads.
\end{corollary}

\begin{remark}
    To justify the title of this section, we can define a \emph{fundamental} simplicial object in $\slist$, say $\UU$, to be equipped with a perfect listing $u : N(\DD{n}_\bullet) \xslashedrightarrow{} \UU$ for some $n \geq 0$ and such that the sequence $u(n)$ is a singleton. Then, Theorem \ref{prop_classsslist} says that fundamental simplicial objects in $\slisto$ are of the form $\UUa$ for rooted simplices $\alpha \in \nd^{root}$.
\end{remark}

\subsection{The homotopy coherent list nerve} \label{subsection : coherent nerve}

It is not difficult to see that $\alpha \mapsto \VVa$ yields a functor $\VV : \Upsilon \to \soperad$. By Yoneda extension, we obtain a \emph{list rigidification} functor $\fo : \slisto \to \soperad$. We define the \emph{list coherent nerve functor} $\ncl : \soperad \to \slisto$ to be the right adjoint of $\fo$.
\[\begin{tikzcd}
	& \slisto \\
	\Upsilon && \soperad
	\arrow["{U_\bullet}", from=2-1, to=1-2]
	\arrow["{\alpha \mapsto \VVa}"', from=2-1, to=2-3]
	\arrow["\ncl"{description}, from=2-3, to=1-2]
	\arrow["\fo", curve={height=-18pt}, dashed, from=1-2, to=2-3]
\end{tikzcd}\]
For $\Q \in \soperad$ we will have
$$(\ncl \Q)_n = \coprod_{\alpha \in \nd_n^{root}} \soperad(\VVa , \Q) .$$

\begin{proposition} \label{prop : commute}
    The following squares of categories and functors commute (up to natural isomorphism)
    \[
    \begin{tikzcd}
	\scat & \soperad \\
	\sset & \slisto
	\arrow["\subseteq"', from=2-1, to=2-2]
	\arrow["\subseteq", from=1-1, to=1-2]
	\arrow["\nc"', from=1-1, to=2-1]
	\arrow["\ncl", from=1-2, to=2-2]
\end{tikzcd}
    ,
    \begin{tikzcd}
	\soperad & {\cat_\D{}} \\
	\slisto & \sset 
	\arrow["{\cc{L}}", from=1-1, to=1-2]
	\arrow["{\cc{L}}"', from=2-1, to=2-2]
	\arrow["\ncl"', from=1-1, to=2-1]
	\arrow["\nc", from=1-2, to=2-2]
\end{tikzcd}\]
\end{proposition}

\begin{proof}
    The first square says that our coherent nerve for simplicial operads extends the coherent nerve of simplicial categories. This is clear. To verify the commutativity of the second square, let
     $\Q \in \soperad$. By Proposition \ref{prop_thicksimplex}, simplices in the set $\nc (\LQ)_n$ are in bijection with morphisms $\VVa \to \Q$ for $\alpha \in \nd_n$, for all $[n] \in \D{}$. Each particular $\alpha \in \nd_n$ has a canonical rooted decomposition as 
    $$\alpha = \coprod_{a \in A_n} \alpha_a$$
    Given that $\alpha \mapsto \VVa$ respects disjoint union (this is easy to see), each morphism $\VVa \to \Q$ determines and is determined by a sequence of simplices in $\ncl \Q$, namely the sequence $(\VV_{\alpha_a} \to \Q)_{a \in A_n}$. Hence the conclusion.
\end{proof}

\begin{remark}
    We could have used the mapping $\alpha \mapsto FU_* T_\alpha$ in order to obtain a coherent nerve. While $FU_* T_\alpha$ is a cofibrant replacement of $T_\alpha$, it has too many simplices in the mapping spaces and does not give the precise results we want. 
\end{remark}

\begin{theorem}
    If $\Q \in \soperad$ is enriched in Kan complexes, then $\ncl \Q$ is a quasi-operad.
\end{theorem}

\begin{proof}
    By Proposition \ref{prop : commute}, we have $\cc{L} (\ncl \Q) \cong \nc (\LQ)$. In case the mapping spaces in $Q$ are Kan complexes, then so are the mapping spaces in $\LQ$ and hence $\nc (\LQ)$ is an $\infty$-category. Thus, $\ncl \Q$ is an $\infty$-operad. 
\end{proof}

\begin{example}[The $\infty$-operads of spaces and $\infty$-categories]
\label{example : operad of spaces}
    Let $\mathsf{Kan}_{\D{}}$ denote the simplicially enriched category of Kan complexes. We may form the simplicial operad $\mathsf{Kan}_{\D{}}^\times$ via the cartesian product. Since the spaces of operations are Kan complexes, we obtain an $\infty$-operad
    $$\cc{S}^{\times} = \ncl (\mathsf{Kan}_{\D{}}^\times)$$
    which may be regarded as the $\infty$-operad of spaces. 

    Similarly, let $\mathsf{qCat}_{\D{}}$ be the simplicially enriched category of $\infty$-categories (quasi-categories). Since the mapping spaces are not Kan complexes, we consider $\mathsf{qCat}_{\D{}}^\circ$, which is the simplicial subcategory with mapping spaces the maximal Kan complexes in the mapping spaces of $\mathsf{qCat}_{\D{}}$. In similar fashion, we may form the simplicial operad $(\mathsf{qCat}_{\D{}}^\circ)^\times$ via the cartesian product and apply the coherent nerve functor to obtain an $\infty$-operad of $\infty$-categories
    $$\cat_\infty^\times = \ncl (\mathsf{qCat}_{\D{}}^\circ)^\times$$
\end{example}

\subsection{Some properties of rigidification}

Rigidification is usually hard to understand. The reason is that, for $X \in \slisto$, the simplicial operad $\fo X$ is given as a colimit of simplicial operads and the latter are not easy to compute. In case $X$ is a simplicial set, even though the rigidification $\fc X$ is a colimit of simplicial categories, we do know something about the shape of the mapping spaces via the necklace theorem of Dugger and Spivak (\cite{dugger2011rigidification}, see also \cite{haderi2023colimits} for a related discussion on colimits of categories). This provides us a little window into understanding some facts about $\fo X$, at least for the case where $X = N^l(\CP)$ for some operad $\CP$, in terms of the simplicial category $\fc (\LP)$.

Let $\CP$ be an operad and let us denote by $\fo \CP$ the list rigidification of $N^l \CP$. We can describe the simplicial category $\fc (\LP)$ as follows. Its objects are those of $\LP$. Given two sequences $\ua , \ub \in \LP$, a $k$-simplex (or $k$-morphism) in the mapping space $\fc (\LP)(\ua, \ub)_k$ is represented by a pair $(f,B)$ consisting of:
\begin{itemize}
    \item [-] A simplex $f : \D{n} \to \LP_{\ua, \ub}$, where notation indicates that $0 \mapsto \ua$ and $n \mapsto \ub$. 

    \item[-] A $k$-morphism $B \in \DD{n}(0,n)_k$ (which is just a bracketing formula, see Appendix \ref{appendix_free}).
\end{itemize}
We identify two such pairs $(f: \D{n} \to \LP_{\ua, \ub},B)$ and $(g : \D{m} \to \LP_{\ua, \ub}, B^\prime)$ if there is a simplicial operation $\sigma : [m] \to [n]$ such that $\sigma^* f = g$
\[\begin{tikzcd}
	{\D{m}} && {\D{n}} \\
	& \LP_{\ua, \ub}
	\arrow["g"', from=1-1, to=2-2]
	\arrow["f", from=1-3, to=2-2]
	\arrow["\sigma", from=1-1, to=1-3]
\end{tikzcd}\]
and $ B^\prime \mapsto B$ under the induces map $\DD{m} \to \DD{n}$.

On the other hand, by Proposition \ref{prop : simplex classification}, the first piece of data above is just a morphism of operads 
    $$f : T_\alpha \to \CP_{\ua, \ub}$$
for some $\alpha \in \nd_n$ and such that $f(A_0) = \ua$ and $f(A_n) = \ub$.
By rooted decomposition of simplices in $\dplus$, we obtain a decomposition functor
\[
\begin{tikzcd}
     \fc(\LP)_k \arrow[r] &\CL \fc (\LP)_k \\[-2em]
     (f: T_\alpha \to \CP, B)  \arrow[r, maps to] & ((f_a : T_{\alpha_a} \to \CP, B))_{a \in A_n}
\end{tikzcd}
\]
This is well defined, since the relation we have on representatives is preserved by rooted decomposition. 

\begin{definition}[Sequenced morphism]
    Let $\CP$ be an operad and $\Q$ be a simplicial operad. We say that a morphism of simplicial categories $$\phi : \fc (\LP) \to \LQ$$ is \emph{sequenced} if for all pairs $(f : T_\alpha \to \CP, B)$ representing a morphism in $\fc(\LP)$ we have that:
    \begin{itemize}
        \item [-] $\phi$ preserves the above decomposition, meaning that we have
    $$\phi (f : T_\alpha \to \CP, B) = (\phi(f_a : T_{\alpha_a} \to \CP, B))_{a \in A_n}$$

        \item[-]  In case $\alpha$ is rooted, $\phi(f : T_\alpha \to \CP, B)$ is a singleton sequence as a morphism in $\LQ$. 
    \end{itemize}

\end{definition}

\begin{construction}[The universal sequenced morphism]

For an operad $\CP$, we have an identity on colors morphism
$$S_\CP : \fc (\LP) \to \CL \fo (\CP)$$
as follows. Given a $k$-morphism $(f, B) \in \fc (\LP) (\ua, b)_k$, for some sequence of colors $\ua$ and color $b$, represented by a pair $f : T_\alpha \to \CP_{\ua, b}$ for some rooted $\alpha \in \nd_n^{root}$ and $B : 0 \to n$ in $\DD{n}_k$, we let $S_\CP (f,B)$ be the operation in $\fc(\CP)$ which is the image under the structure map $\VVa \to \fo(\CP)$ 
of the operation $(B, r_\alpha) : A_0 \to A_n$. The latter operation lives in $\UUa$, when regarding $B$ as a morphism $\D{1} \to \DD{n}_k$. Then, we extend $S_\CP$ sequentially, so that it is a sequenced morphism by construction.

Equivalently, we can define $S_\CP$ via a series of transpositions through adjunctions as follows. Consider the identity morphism $id : \fo (\CP) \to \fo (\CP)$, which induces the unit $N^l(\CP) \to \ncl (\fo (\CP))$ in $\slist$. By applying the functor $\CL$, which commutes with both nerves, we obtain a morphism $N(\LP) \to \nc (\CL \fo(\CP))$ in $\sset$. The latter transposes into a morphism $S_\CP : \fc (\LP) \to \CL \fo (\CP)$ in $\scat$ by adjunction. 

\end{construction}

\begin{lemma}[Universal property of $S_\CP$]

Let $\CP$ be an operad and $\Q$ be a simplicial operad. Given a sequenced morphism $\phi: \fc (\LP) \to \LQ$, there is a unique morphism of simplicial operads $\psi : \fo(\CP) \to \Q$ such that $(\CL \psi) \circ S_\CP = \phi$. 
\[\begin{tikzcd}[column sep=small]
	{\fc (\LP)} && \LQ \\
	{\CL  \fo (\CP)}
	\arrow["{S_\CP}"', from=1-1, to=2-1]
	\arrow["\phi", from=1-1, to=1-3]
	\arrow["{\exists ! \CL \psi}"', dashed, from=2-1, to=1-3]
\end{tikzcd}\]
    
\end{lemma}

\begin{proof}
    We have a commutative diagram
    \[\begin{tikzcd}[column sep=small]
	{\soperad (\fo (\CP), \Q)} && {\slist(N\CP, \nc \Q)} \\
	{\scat (\CL \fo (\CP),\CL \Q)_{mon}} && {\sset (N(\LP), \nc (\LQ))_{mon}} \\
	& {\scat (\fc (\LP), \LQ)_{seq}}
	\arrow["\CL"', from=1-1, to=2-1]
	\arrow["{S_\CP^*}"', from=2-1, to=3-2]
	\arrow["\CL", from=1-3, to=2-3]
	\arrow["\cong", from=1-1, to=1-3]
	\arrow["\cong", from=2-1, to=2-3]
	\arrow["\cong"', from=3-2, to=2-3]
\end{tikzcd}\]
where $\soperad (\CL \fo (\CP),\CL \Q)_{mon}$ and $\sset (N(\LP), \nc (\LQ))_{mon}$ are the sets of morphisms which respect monoid structure, and 
$\scat (\fc (\LP), \LQ)_{seq}$ is the set of sequenced morphisms. The unlabelled isomorphisms in the top square are the maps given by adjunction. Notice that the vertical maps induced by applying $\CL$ are also isomorphisms. 

For the triangle on the bottom, it is evident by scrutiny that the restriction of the adjunction isomorphism $\scat (\fc (\LP), \LQ) \xrightarrow{\cong} \sset (N(\LP), \nc (\LQ))$ to sequenced morphisms is a bijection into maps of monoids. Hence, the function $S_\CP^*$ induced by composing with $S_\CP$ is also a bijection. 
\end{proof}

\begin{proposition}
    For all operads $\CP$, the simplicial operad $\fo (\CP)$ is cofibrant.
\end{proposition}

\begin{proof}
    Consider a lifting problem in the category of simplicial operads
    \[\begin{tikzcd}
	& {\cc{K}} \\
	{\fo (\CP)} & \Q
	\arrow[from=2-1, to=2-2]
	\arrow["\sim", two heads, from=1-2, to=2-2]
	\arrow[dashed, from=2-1, to=1-2]
\end{tikzcd}\]
    where the map on the right is a trivial fibration. By, applying the lists functor $\CL$, and in light the universal property of $\fo(\CP)$, consider the following diagram:
    \[\begin{tikzcd}
	&& {\CL \cc{K}} \\
	{\fc(\LP)} & {\CL \fo (\CP)} & \LQ
	\arrow[from=2-2, to=2-3]
	\arrow["\sim", two heads, from=1-3, to=2-3]
	\arrow[dashed, from=2-2, to=1-3]
	\arrow["S"', from=2-1, to=2-2]
	\arrow[dashed, from=2-1, to=1-3]
\end{tikzcd}\]
    The map on the right is a trivial fibration of simplicial categories since the functor $\CL$ preserves trivial fibrations (Lemma \ref{lemma : Lwefib}), hence there exists a lift $\fc(\LP) \to \CL \cc{K}$ since $\fc(\LP)$ is cofibrant. Moreover, this lift has to be sequenced because of commutativity. Hence, we obtain the desired lift $\fo(\CP) \to \cc{K}$.
    
\end{proof}

\begin{remark}[Cofibrant replacement]

  By construction, we have the colimit formula $\fo (\CP) \cong 
\displaystyle \colim_{T_\alpha \to \CP} \VVa$. Since $\CP \cong \tau_1 (N^l\CP)$, we also have the colimit formula $\CP \cong \displaystyle \colim_{T_\alpha \to \CP} T_\alpha$. Thus, we obtain a structure map $\fo (\CP) \to \CP$, where $\CP$ is regarded as a discrete simplicial operad.

In a previous version of this paper, we \emph{erroneously} claimed that this map is weak equivalence, so that $\fo (\CP)$ serves as a cofibrant replacement for $\CP$. We believe the result is true, but we leave it for subsequent work. 

Notive that $\fo(\CP) \to \CP$ being a weka equivalence would imply that the the universal sequenced morphisms $S_\CP : \fc (\LP) \to \CL \fo(\CP)$ is also a weak equivalence since we have a commutative triangle
\[\begin{tikzcd}[column sep=small]
	{\fc (\LP)} && {\CL  \fo (\CP)} \\
	& \LP
	\arrow["{S_\CP}", from=1-1, to=1-3]
	\arrow["\sim"', from=1-1, to=2-2]
	\arrow[ from=1-3, to=2-2]
\end{tikzcd}\]

In the proof of Proposition \ref{prop : natwe} we constructed a natural weak equivalence $D : \fc (\LP) \to \CL (FU_* \CP)$. By manner of construction, $D$ is sequenced, hence it induces a morphism $R : \fo \CP \to FU_* \CP$. The map $\CL R$ fits into the following commutative square
    \[\begin{tikzcd}
	& {\fc (\LP)} \\
	{\CL (\fo \CP)} && {\CL (FU_* \CP)} \\
	& \LP
	\arrow[from=2-1, to=3-2]
	\arrow[from=2-3, to=3-2]
	\arrow["{S_\CP}"', from=1-2, to=2-1]
	\arrow["D", from=1-2, to=2-3]
	\arrow["{\CL R}"{description}, from=2-1, to=2-3]
\end{tikzcd}\]
in which the two maps on the bottom are obtained by applying $\CL$ to the structure morphisms $\fo \CP \to \CP$ and $FU_* \CP \to \CP$. The top triangle commutes by definition, while the bottom triangle commutes by inspection. Thus, we have a commutative triangle 
\[\begin{tikzcd}
	{ \fo \CP} && { FU_* \CP} \\
	& \CP
	\arrow[from=1-1, to=2-2]
	\arrow["\sim", from=1-3, to=2-2]
	\arrow["{ R}", from=1-1, to=1-3]
\end{tikzcd}\]
which would imply that $R$ is a weak equivalence as well.

\end{remark}

\begin{remark}[Homotopy coherence]

Let $\CP$ be an operad and $\Q$ be a simplicial operad. In Section \ref{subsec : free resolution} we discussed how it is sensible, in analogy with classical theory, to say that a homotopy coherent $\Q$-valued algebra of $\CP$ can be said to be a morphism of simplicial operads $FU_* \CP \to \Q$. 

In light of the coherent nerve we proposed earlier, it also makes sense for homotopy coherent algebras to be defined as morphisms $\fo \CP \to \Q$. 
The advantage of studying maps $\fo \CP \to \Q$ is that they can be transposed into morphisms of simplicial lists $N^l\CP \to \ncl \Q$. 

The true benefit of such a transposition for coherence can only come to surface when more theory is developed. Nonetheless, we deem the ability to state that a homotopy coherent algebra is a map of simplicial lists an important simplification in the study of coherent phenomena. We illustrate with a couple of examples.
    
\end{remark}

\begin{example}[$A_\infty$-spaces and $A_\infty$-algebras]

We can model $A_\infty$-spaces as morphisms of simplicial lists $N^l\assoc \to \cc{S}^\times$ and $A_\infty$-algebras as morphisms of simplicial lists $N^l\assoc \to \cat_\infty^\times$, where $\cc{S}^\times$ and $\cat_\infty^\times$ are the $\infty$-operads of spaces and $\infty$-categories constructed in Example \ref{example : operad of spaces}.
    
\end{example}

\begin{remark}
    The above notions ought to correspond, via a good theory of fibrations, to monoidal $\infty$-groupoids and monoidal $\infty$-categories. We hope to achieve such results in the sequel papers.
\end{remark}

\begin{example}[Weak enrichment]

Let $S$ be a set and $P$ be an $\infty$-operad. We can say that a $P$-enriched category with set of objects $S$ is a morphism of simplicial lists $N^l \mathsf{Hom}_S \to P$, where $\mathsf{Hom}_S$ is the operad from Example 
\ref{example : hom operad}. In particular, categories weakly enriched in spaces give us a variant for $(\infty, 1)$-categories, while enriching in $\infty$-categories produces a variant for $(\infty,2)$-categories. 

\end{example}

\section{Homology of simplicial lists} \label{section : homology}

We briefly discuss a natural variant for the homology of a simplicial list. Other possible variants, and more sophisticated results about homology and its applications are left for future work.

Let $\Ab $ denote the category of abelian groups. Let $\bb{Z} :\llist \to \Ab$ be the functor that sends an object $X$ in $\llist $ to the free abelian group $\mathbb{Z}X$ in $\Ab$ and send a morphism $u :X\xslashedrightarrow{} Y$ in  $\llist $  to the homomorphism $\bb{Z}(u):\mathbb{Z}X\to \mathbb{Z}Y$ in  $\Ab $  that linearly extends the assignment
$$\bb{Z}(u)(x)=\sum _{i=1}^{|u(x)|} u(x)_i$$
for $x$ in $X$ where $u(x)$ is the list $\left(u(x)_1, u(x)_2,\dots u(x)_{|u(x)|}\right)$. Here, note that if $u(x) = \emptyset$ is the empty list then $\bb{Z}(u)(x)=0$.

Let  $\sAb$ denote the category of simplicial abelian groups, $\cAb$ the category of nonnegatively graded chain complexes, and $\listdelta$ the functor category from $\Delta ^{op}$ to   $\llist$, as discussed in \ref{subsec : slist}. Recall that  $\slist$ is a wide subcategory of $\listdelta$.
The functor  $\bb{Z}:\llist \to \Ab$ induces a functor  from $\listdelta $
to $\sAb$. Hence composing with the full alternating face map complex functor we get a functor  
\[
\begin{tikzcd}
    C: &[-3em] \listdelta \arrow[r]  & \cAb
\end{tikzcd}
\]

\begin{definition}[Homology]

Let $\grAb$ be the category of graded abelian groups. Let  $H:\cAb\to \grAb$ denote the homology functor. For $X$ in $\listdelta $, we define the homology $H(X)$ of $X$ as the homology $H(C(X))$ of the chain complex $C(X)$. 
    
\end{definition}

Note that an augmented simplicial object $X \dashrightarrow X_{-1}$  in $\listdelta $ induces a morphism $X\to X_{-1}$ in $\listdelta $ where $X_{-1}$ is considered as the constant simplicial list. Hence it induces a morphism of homology groups $H(X)\to H(X_{-1})$. Notice that we have $H(X_{-1})\cong \mathbb{Z}X_{-1}$ where $\mathbb{Z}X_{-1}$ denotes free abelian group generated by the set $X_{-1}$ considered as a graded abelian group concentrated at degree $0$.
\begin{theorem}\label{thm:aug_homology}
	If an augmented simplicial object $X \dashrightarrow X_{-1}$ in $\listdelta $ has extra-degeneracies, then the induces morphisms $H(X)\to \mathbb{Z}X_{-1}$ is an isomorphism of graded abelian groups. 
\end{theorem}
\begin{proof}
	The sequence of morphisms $\left\{h_n:C_n(X)\to C_{n+1}(X)\right\}_{n=0}^{\infty }$ induced by the morphism  $s_{-1}:X_n\to X_{n+1}$  for $n\geq 0$ gives a chain homotopy from the chain map $C(X)\stackrel{C(d_0) }{\to} C(X_{-1}) \stackrel{C(s_0) }{\to}C(X)$ to the identity chain map on $C(X)$.	
\end{proof}

Notice that the empty set $\emptyset $ is the initial and terminal object in $\llist$. Hence the empty set $\emptyset $  considered as a constant simplicial list is an initial and terminal object in $\listdelta$. Also note that the homology of the simplicial list $\emptyset$ is the trivial abelian group $0$. 

\begin{proposition} \label{prop : homology assoc}
	$H(N^l\assoc)\cong 0$
\end{proposition}
\begin{proof} We have $N^l\assoc \cong \nd^{root}$. Hence, it is enough to show that the unique augmentation $\nd^{root}\dashrightarrow\emptyset$ has extra-degeneracies.  Let $\alpha$ be $A_0\stackrel{\alpha_1 }{\to}A_1\dots   \stackrel{\alpha_{m} }{\to}A_{m}$  in $ \nd^{root}_{m}$. Define  $s_{-1}(\alpha)$ as  $\emptyset \to A_0\stackrel{\alpha_1 }{\to}A_1\dots   \stackrel{\alpha_{m} }{\to}A_{m}$  in $ \nd_{m+1}$. It is straight forward to check that this assignment gives the required extra degeneracy maps. Hence the result follows by Theorem \ref{thm:aug_homology}. \end{proof}




\begin{theorem}\label{thm:homology_of_T_alpha}
	 Let $\alpha$ be  $A_0\stackrel{\alpha_1 }{\to}A_1\dots   \stackrel{\alpha_k }{\to}A_k$  in $ \nd_k$. Then  $$H(N^lT_\alpha)\cong \mathbb{Z}A_0$$ where $\mathbb{Z}A_0$ is considered as the graded abelian group concentrated at degree $0$.
\end{theorem}
\begin{proof}
 Let $d_0$ be the morphism from $(N^lT_{\alpha})_0= \displaystyle \coprod_{i=0}^k A_i$ to $A_0$ that sends $x$ in $A_i$ to $\alpha _{0,i}^{-1}(x)$ considered as list of elements in $A_0$. Let $\beta $ be depicted by $B_0\to B_1$. Then for any operad morphism $f:T_{\beta}\to T_{\alpha}$ in $(N^lT_{\alpha})_1$  we have $d_0d_0(f)=\alpha _{0,i}^{-1}(f(B_1))=d_0d_1(f)$. Hence we have an augmentation 
$$N^lT_{\alpha}\dashrightarrow A_0.$$ 
Now we construct extra degeneracies for this augmentation. Define $s_{-1}:A_0\to (N^lT_{\alpha})_0$ as the inclusion. Then clearly $d_0s_{-1}$ is identity on $A_0$. Let $\beta$ be  $B_0\stackrel{\beta_1 }{\to}B_1\dots   \stackrel{\beta_s }{\to}B_s$  in $ \nd_s$ and let $f$ be an operad morphism from $T_{\beta}$ to $T_{\alpha}$. This means $f$ is an $s$-simplex in $N^lT_{\alpha}$. Then we define $\gamma $ depicted by  $C_0\stackrel{\gamma_1 }{\to}C_1\dots   \stackrel{\gamma_{s+1} }{\to}C_{s+1}$ in  $ \nd_{s+1}$ as follows: $C_{i+1}=C_i$ for $i\geq 0$ and
$$C_0=\displaystyle \coprod_{b\in B_0} \alpha_{0,j(f(b))}^{-1}(f(b))$$
where $j(b)$ is determined by saying $f(b)\in A_{j(f(b))}$ and $\gamma_{i+1}=\beta _i$ for $i\geq 1$ and $\gamma_1:C_0\to C_1$ be given by $\gamma_1(x)=b$ if $x\in \alpha_{0,j(f(b))}^{-1}(f(b))$.
Note that $T_{\beta}$ is a suboperad of $T_{\gamma}$. We extend $f$
to an operad morphism $s_{-1}(f)$ from $T_{\gamma }$ to $T_{\alpha }$
by sending $p_{0,1}^{(b)}$ to $p_{0,j(f(b))}^{(f(b))}$ for $b$ in $C_1=B_0$. It is straight forward to see that $d_is_{-1}=s_{-1}d_{i-1}$ for $i>0$ and $s_is_{-1}=s_{-1}s_{i-1}$ for $i>0$. Hence the result follows by Theorem \ref{thm:aug_homology}.
\end{proof}

Let $\listtwo$ denote the category whose objects are pairs of sets $(X,Y)$ where $Y$ is a subset of $X$ and morphisms from $(X,Y)$ to $(W,Z)$ are morphisms $f:X\to Z$ in $\llist$ which sends elements in $Y$ to lists in $Z$. Let $\listdeltatwo$ denote simplicial objects in $\listtwo$. Let $Q:\listtwo\to\llist$ denote the functor that sends a pair $(X,Y)$ to the set difference $X-Y$ and a morphism $f:(X,Y)\to (W,Z)$ in $\listtwo$ to the morphism $Q(f):(X-Y)\to (W-Z)$ where $Q(f)(x)$ is the list obtained from $f(x)$ by deleting every element in $Z$. Let $\iota:\llist\to\listtwo$ be the functor that sends $X$ to $(X,\emptyset)$ and morphisms to itself. The functors $Q$, $\iota$ respectively induce the functors $Q:\listdeltatwo\to\listdelta$ and $\iota:\listdelta\to\listdeltatwo$

Let $Y$ be a subsimplicial list of a simplicial list $X$. We define the relative homology groups as follows:
$$H(X,Y)=H(C(X)/C(Y)).$$ 
Then the short exact sequence $$0\to C(X)\to C(Y)\to C(X)/C(Y)\to 0$$ induces the following long exact 
$$ \cdots \to  H_{n+1}(Y)
\to H_{n+1}(X)
\to  H_{n+1}(X,Y)
\stackrel{\delta_n }{\to} H_{n}(Y)
\to H_{n}(X)
\to \cdots $$
where $\delta_n$ is the $n$th connecting homomorphism for $n\geq 0$.  This long exact sequence of homology groups will be called the \emph{homology long exact sequence} associated to the pair $(X,Y)$.


Assume that $Y$ is a subsimplicial list of a simplicial list $X$. Then it is straightforward to prove that we have 
$$H(X,Y)\cong H(Q(X,Y)).$$

\begin{proposition}
		 If $\alpha$ is in $ \nd_k$ then we have an isomorphism
		 $$H(U_{\alpha})\cong H(N^lT_\alpha).$$
		 induced by the inclusion $U_{\alpha }\to N^lT_\alpha$ which picks the $\alpha$-shaped simplex $1_{T_\alpha} : T_\alpha \to T_\alpha$.
\end{proposition}
\begin{proof}
	 Let $s_{-1}$ be the extra degeneracies on the augmentation $N^l T_{\alpha} \dashrightarrow A_0$ constructed in the proof of Theorem \ref{thm:homology_of_T_alpha}. These extra degeneracies restrict to extra degeneracies on the augmentation  $Q(N^lT_{\alpha},U_{\alpha}) \dashrightarrow \emptyset$. Hence we are done by considering the long exact homology sequence for the pair $(N^lT_\alpha,U_{\alpha })$.
\end{proof}

\begin{remark}
    The homology of dendroidal sets has been studied in \cite{bavsic2014homology}. 
    We do not study functors relating dendroidal sets and simplicial lists, so we cannot give a complete comparison between the homology defined in this section and the homology defined in \cite{bavsic2014homology}. However, we can say that they both generalize the singular homology of simplicial sets, where a simplicial set is considered as a dendroidal set supported on the linear trees and considered as a simplicial list whose all degeneracy and face maps are functions. 
    
    Moreover, considering the list nerve of a non-symmetric operad and the dendroidal nerve of its symmetrization, the two homologies agree on some operads like the one obtained from a planar tree and the associative operad. For this, compare Proposition \ref{prop : homology assoc} and Theorem \ref{thm:homology_of_T_alpha} in this section with Theorem 6.5 and Corollary 5.13  in \cite{bavsic2014homology}. 

    Another variant for dendroidal homology has been studied in \cite{Hoffbeck2024homology}. Again, there seems to be a point of comparison between Proposition \ref{prop : homology assoc} and \cite[Example 3.2]{Hoffbeck2024homology}. 
\end{remark}

\appendix 

\section{Augmentations and extra degeneracies} \label{appendix_Aug}

\begin{definition}[Augmented simplicial set]

An augmented simplicial set is a functor $X : \dplus^{op} \to \set$.
    
\end{definition}

The augmented simplex category $\dplus$ is the category of finite ordinals in which the empty ordinal $[-1]$ is included. Hence, an augmented simplicial set is determined by specifying:
\begin{itemize}
    \item [-] A simplicial set $X$. 
    \item[-] A set $X_{-1}$ which is equipped with a map $d_0 : X_0 \to X_{-1}$ such that $d_0d_0 = d_1d_0$. 
\end{itemize}
In other words, we regard augmentation as extra structure on a simplicial set and use the notation $X \dashrightarrow X_{-1}$ to indicate data as above. Notice that we do have an actual map of simplicial sets $X \to X_{-1}$ into the constant simplicial object which maps a simplices of $X$ to the face $d_0$ of one of its vertices (this is well-defined by the above identity). 

\begin{definition}[Extra degeneracies, \protect{\cite[p.200]{goerss2009simplicial}}]

An augmented simplicial set $X \dashrightarrow X_{-1}$ is said to have extra degeneracies in case there exist maps $s_{-1} : X_{n-1} \to X_n$, $n \geq 0$, which satisfy $d_0s_{-1} = 1$ and for all $i,j \geq 0$, $d_{i+1}s_{-1} = s_{-1}d_i$ and $s_{j+1}s_{-1} = s_{-1}s_j$.

\end{definition}

\begin{lemma}[\protect{\cite[Lemma 5.1]{goerss2009simplicial}}] \label{lemma_contractible}

If an augmented simplicial set $X \dashrightarrow X_{-1}$ has extra-degeneracies, then the map $X \to X_{-1}$ is a homotopy equivalence.
    
\end{lemma}

The idea behind the above lemma can be sketched as follows. Let us depict the $0$-simplices $X_0$ which are degeneracies of elements of $X_{-1}$ as generic points $\bullet$ and see how the extra degeneracies do provide a recipe of contraction for simplices of $X$ to bullet points. For a $0$-simplex $x \in X_0$, the degeneracy $s_{-1} x$ has to be of the following form
\[\begin{tikzcd}
	x && \bullet
	\arrow["{s_{-1}x}"', from=1-3, to=1-1]
\end{tikzcd}\]
and thus provides a path for $x$ to “slide back" into $\bullet$. Given a $1$-simplex $x_0 \xrightarrow{\sigma} x_1$ in $X_1$, the degeneracy $s_{-1}\sigma$ is of the following form
\[\begin{tikzcd}
	{x_0} && {x_1} \\
	& \bullet
	\arrow["{s_{-1}x_0}", from=2-2, to=1-1]
	\arrow[""{name=0, anchor=center, inner sep=0}, "\sigma", from=1-1, to=1-3]
	\arrow["{s_{-1}x_1}"', from=2-2, to=1-3]
	\arrow["{s_{-1}\sigma}"{description}, draw=none, from=2-2, to=0]
\end{tikzcd}\]
and thus provides a path to slide the whole simplex back into $\bullet$. And so on for higher simplices. Natural examples of augmentations with extra degeneracies arise from free resolutions, where the term extra degeneracy makes good sense.

\begin{remark}
    The notion of augmentation and extra degeneracy applies to simplicial objects in any category. A statement which vastly generalizes the above lemma can be found in \cite[Corollary 4.5.2]{riehl2014categorical}. A detailed account of extra degeneracies and variants of this notion can be found in \cite{barr2019contractible}. In the context of $\infty$-categories, they appear as \emph{split simplicial objects} in \cite[Definition 10.2.6.3]{luriekerodon10.2}.
\end{remark}

\section{Free resolutions} \label{appendix_free}

The idea of a free resolution is simple and follows the same pattern in various algebraic structures. Informally, by algebraic structure we mean any mathematical structure which has incorporated in it some sort of operation that allows to write equations of the form $A \cdot B = C$. To resolve $\cdot$ involves forgetting the algebraic identities in the structure, and thus, in many examples a free-forgetful adjunction appears to control resolutions. 

For pedagogical purposes, we first consider the free-forgetful adjunction between monoids and pointed sets
\[
\begin{tikzcd}
    F :&[-3em] \pset \arrow[r, shift left = .5ex ] &\monoid : U \arrow[l, shift left = .5ex]
\end{tikzcd}
\]
For a monoid $M$, the pointed set $UM$ is simply the underlying set of $M$ pointed by the unit element $1 \in M$. For a pointed set $(X, e)$, the monoid $F(X, e)$ is the monoid whose elements are words in $X$ modulo a relation which says that $e$ is the unit for the operation given by concatenation. More precisely, for all words $(x_1, \dots , x_k)$, we identify
$$(x_1, \dots , x_k) \equiv (x_1, \dots, e , \dots , x_k)$$

This adjunction induces a functor 
\[
\begin{tikzcd}
    FU_* :&[-3em] \monoid \arrow[r] & \monoid^{\D{op}}
\end{tikzcd}
\]
from monoids to simplicial monoids. For a monoid $M$, the monoid of $m$-simplices of $FU_* M$ is the monoid $FU^{(m+1)}M$ obtained by iterating the endofunctor $FU$ $(m+1)$-times. The simplicial structure follows formally by applications of the unit and counit of the adjunction and it is equipped with an augmentation $FU_*M \dashrightarrow M$ with extra degeneracies. 
We are interested in a heuristic unpacking in terms of \emph{bracketing}. 

Let $M$ be a monoid. An element in $FUM$ is just a word in $M$ (technically a class, but we neglect the unit as it can always be removed). The augmentation morphism is given by the counit of the adjunction
\[
\begin{tikzcd}
    d_0 : &[-3em] FUM  \arrow[r] & M \\[-2em]
    & (x_1, \dots , x_k)  \arrow[r, maps to] & x_k \cdot \dots \cdot x_1
\end{tikzcd} , 
\]
while the extra degeneracy $s_{-1}$ is given by
\[
\begin{tikzcd}
    s_{-1} : &[-3em] M  \arrow[r] & FUM \\[-2em]
    & x  \arrow[r, maps to] & (x)
\end{tikzcd}
\]
Note how this really feels like a degeneracy operation. 
This iteration of the resolution resolves multiplication in $M$.

The next bit of resolution (and extra degeneracy)
\[\begin{tikzcd}
	FUFUM && FUM
	\arrow["{d_0}"{description}, shift left=4, from=1-1, to=1-3]
	\arrow["{d_1}"{description}, shift right=4, from=1-1, to=1-3]
	\arrow["{s_0}"{description}, from=1-3, to=1-1]
	\arrow["{s_{-1}}"{description}, curve={height=-24pt}, dashed, from=1-3, to=1-1]
\end{tikzcd}\]
can be described as follows. An element of the monoid $FUFUM$ is a word of words in $M$, say 
$$((x_1, x_2), (x_3, x_4, x_5))$$
which we think of as a \emph{bracketing} of the sequence $(x_1, x_2, x_3, x_4, x_5)$. Then, the map $d_1$ removes inner brackets, for instance
$$d_1((x_1, x_2), (x_3, x_4, x_5)) = (x_1, x_2, x_3, x_4, x_5)$$
The map $d_0$ applies multiplication inside the inner brackets and then removes them, for instance
$$d_0((x_1, x_2), (x_3, x_4, x_5)) = (x_2 \cdot x_1, x_5 \cdot x_4 \cdot x_3)$$
The degeneracy $s_0$ adds a layer of inner brackets
$$s_0(x_1, x_2, x_3, x_4, x_5) = ((x_1), (x_2), (x_3), (x_4), (x_5))$$
The extra-degeneracy $s_{-1}$ adds brackets on the exterior
$$s_{-1}(x_1, x_2, x_3, x_4, x_5) = ((x_1, x_2, x_3, x_4, x_5))$$
Again, this feels like a degenerate thing to do. 

In general, the elements of $FU^{(m)}M$ are sequences with $m$-layers of brackets. The faces operations in the simplicial structure remove layers, except for the face $d_0$ which applies multiplication in the inner brackets before removing them. The degeneracies add extra brackets around the layers. The extra degeneracy adds a pair of brackets outwardly. 

In case we want to be more rigorous about what we mean by bracketing, we would say the following:
\begin{itemize}
    \item [-] A word in $M$ is the choice of an element $x$ together with a factorization $x = x_k \cdot \dots \cdot x_1$ of $x$.
    \item[-] A bracketed word is the choice of a factorization for each of the factors $x_i$.
    \item[-] And so on. 
\end{itemize}
From this perspective, the extra degeneracies are the choice of the trivial factorization. 

When applied to the free-forgetful adjunction between abelian groups and pointed sets, we obtain a simplicial abelian group which, under the Dold-Kan correspondence, provides the free resolution of the abelian group in the usual sense of chain complexes. The quasi-isomorphism to the original group is precisely the presence of the augmentation with extra degeneracies. In fact, the backward pointing morphisms involved in the definition of null-homotopy for chain complexes can be seen to derive by the extra degeneracies.
Another classical example is the free resolution for categories obtained by the free-forgetful adjunction between graphs and categories. 

\begin{definition}[Graph]
    A (reflexive, directed) \emph{graph} $G$ is a $1$-truncated simplicial set, i.e. it consists of a set of vertices $G_0$ and a set of edges $G_1$ related by faces and degeneracy operations
    \[\begin{tikzcd}
	{G_1} && {G_0}
	\arrow["{d_0}"{description}, shift left=4, from=1-1, to=1-3]
	\arrow["{d_1}"{description}, shift right=4, from=1-1, to=1-3]
	\arrow["{s_0}"{description}, from=1-3, to=1-1]
\end{tikzcd}\]
    which satisfy $d_0s_0 = d_1s_0$.

    A morphism of graphs is a natural transformation. We denote by $\mathsf{Graph}$ the resulting category.
\end{definition}

The evident adjunction
\[
\begin{tikzcd}
    F :&[-3em] \mathsf{Graph} \arrow[r, shift left = .5ex ] & \cat : U \arrow[l, shift left = .5ex]
\end{tikzcd}
\]
assigns to a graph $G$ the category $FG$ which has as objects the vertex set $G_0$ and as morphisms chains of “composable" edges in $G_1$ (modulo a relation which says that degeneracies serve as identities in $FG$). For a category $\C$, the graph $U\C$ has the objects of $\C$ as vertices and morphisms of $\C$ as edges, with identities being the degeneracies. 

The free resolution $FU_* \C$, by construction, resolves the composition operation in $\C$. The simplicial structure of this resolution can be put in the exact same terms as the resolution for monoid described above, except that words are replaced by sequences of composable morphisms in $\C$. Moreover, given that $FU^{(m)}\C$ has always the same objects of $\C$ and that the simplicial operations are identity on objects, we may regard $FU_* \C$ as a simplicially enriched category, which means we have $m$-simplices in the mapping spaces provided by morphism sets in the resolution. 

By the augmentation and extra degeneracies, we do get for free that for all objects $a,b \in \C$ we have a weak-equivalence of simplicial sets $(FU_* \C)(a,b) \to \C(a,b)$, which means that the morphism $FU_*\C \to \C$ is a weak equivalence\footnote{These are also called Dwyer-Kan equivalence, or DK-equivalences for short. They serve as weak equivalence in the Bergner model structure on simplicial categories (\cite{bergner2007model}).} of simplicial categories. In fact, $FU_* \CP$ serves as cofibrant replacement in the Bergner model structure (see \cite[Section 16.2]{riehl2014categorical} for a proof). 

More is true. It is an easy exercise to determine that there is an isomorphism\footnote{Sometimes this is taken to be a definition.} $FU_*\D{n} \cong \bb{\D{n}}$, where $\bb{\D{n}}$ is the simplicial thickening of $\D{n}$ typically defined combinatorially as follows:
\begin{itemize}
    \item [-] The objects are $0, \dots, n$.
    \item[-] An $m$-simplex in the mapping space $\bb{\D{n}}(i,j)$ is a chain of inclusions $U_0 \subseteq \dots U_m$ of subsets of $\{0, \dots ,n \}$ such that $i,j \in U_0$.
    \item[-] Faces and degeneracies are obtained by removing or repeating terms of chains of subsets as above.
    \item[-] Composition is given by union of sets. 
\end{itemize}
The cosimplicial object $\D{} \to \cat_\D{}, [n] \to \bb{\D{n}}$, induces an adjunction 
\[
\begin{tikzcd}
    \fc :&[-3em] \sset \arrow[r, shift left = .5ex] & \scat : \nc \arrow[l, shift left = .5ex]
\end{tikzcd}
\]
between simplicial sets and simplicial categories, which is one of the fundamental constructions in homotopy theory as it establishes a Quillen equivalence between two models for $(\infty, 1)$-categories. The functor $\fc$ is also referred to as rigidification, since, when applied to an $\infty$-category it outputs a simplicial category where some compositions are strict. It is trickier to prove (for instance, by applying the Necklace Theorem of Dugger and Spivak proved in \cite{dugger2011rigidification}, see Theorem 16.4.7 in \cite{riehl2014categorical}), but true, that we have an isomorphism 
$$\fc(N\C) \cong FU_* \C$$
for all categories $\C$. 

\begin{remark}
    In full generality, given an adjunction $F : \C \leftrightarrows \cc{D} : G $, the comonad structure of $F \circ U$ equips $(F \circ U)^{\circ(\bullet+1)}$ with the structure of an augmented
simplicial object with extra degeneracies in the endofunctor category of $\cc{D}$. 
\end{remark}

\printbibliography

\end{document}